\documentclass[a4paper,12pt]{article}
\usepackage[arrow, matrix, curve]{xy}

\usepackage{amssymb}
\usepackage{relsize}
\usepackage{enumerate}
\usepackage[clearempty]{titlesec}
\usepackage{url}
\usepackage{hyperref}

\usepackage{extpfeil}
\usepackage[latin1]{inputenc}
\usepackage{amsmath}
\usepackage{amscd}
\usepackage{amsfonts}
\usepackage{amsthm}
\usepackage{bm}
\usepackage{mathtools}
\usepackage{txfonts}
\usepackage[english]{babel}
\usepackage{makeidx}

\theoremstyle{plain}
\newtheorem{satz}{Theorem}[section]
\newtheorem*{theorem*}{Theorem}
\newtheorem{lem}[satz]{Lemma}
\newtheorem{folg}[satz]{Corollary}
\newtheorem{prop}[satz]{Proposition}

\theoremstyle{definition}
\newtheorem{mydef}[satz]{Definition}
\newtheorem{bem}[satz]{Remark}
\newtheorem{prob}[satz]{Problem}

\begin{document}
\bibliographystyle{plain}

\title{\textbf{Symmetry Groups of Principal Bundles Over Non-Compact Bases}}
\author{Jakob Sch\"utt}
\maketitle
\thispagestyle{empty}

\begin{abstract}
In this work we describe how to obtain the structure of an infinite-dimen\-sional Lie group on the group of compactly carried bundle automorphisms $\mathrm{Aut}_c(\mathcal{P})$ for a locally convex prinicpal bundle $\mathcal{P}$ over a finite-dimensional smooth $\sigma$-compact base $M$. For this we first consider the Lie group structure on the group of compactly carried vertical bundle morphisms $\mathrm{Gau}_c(\mathcal{P})$ (in both cases ``compactly carried" refers to being compactly carried on the base in a certain sense). We then introduce the Lie group structure on $\mathrm{Aut}_c(\mathcal{P})$ as an extension of a certain open Lie subgroup of the compactly carried diffeomorphisms $\mathrm{Diff}_c(M)$ by the gauge group $\mathrm{Gau}_c(\mathcal{P})$. We find an explicit condition on $\mathcal{P}$ ensuring that $\mathrm{Gau}_c(\mathcal{P})$ can be equipped with a Lie group structure enabling the extension just mentioned and show that this condition is satisfied by selected classes of bundles.\\
[\baselineskip]
\\[\baselineskip]
\end{abstract}

\bigskip

\section{Introduction}
In this paper, we will introduce two new examples of infinite-dimensional Lie groups, the gauge group $\mathrm{Gau}_c(\mathcal{P})$ and the group of bundle automorphisms $\mathrm{Aut}_c(\mathcal{P})$ over a prinicipal $G$-bundle $\mathcal{P}$ with $\sigma$-compact finite-dimensional base $M$. Both Lie group structures will be described in an explicit and accessible way. The gauge group will be shown to be isomorphic to a closed subgroup of a weak direct product of certain mapping groups. Then, $\mathrm{Aut}_c(\mathcal{P})$ can be obtained by extending (an open subgroup of) $\mathrm{Diff}_c(M)$, the group of compactly carried diffeomorphisms, by the gauge group. In this sense the present work relates two major classes of infinite-dimensional Lie groups in a non-trivial way, namely groups of diffeomorphisms and mapping groups. To make this thesis more readable many of the requirements needed for the later parts are organized at the beginning. 
The general structure is as follows.
\bigskip

In the first part we introduce our framework of infinite-dimensional differential calculus. As the groups of mappings we are interested in are defined on compact manifolds we need to find a way to cover our $\sigma$-compact base $M$ with such manifolds. If there exist components of $M$ that are not compact such a cover cannot be expected to exist, however we recall the concept of manifolds with corners to remedy this problem. For this, we will reiterate the basics about differentiable mappings between not necessarily open subsets of locally convex spaces.\\
Next, we explain how to turn spaces of mappings with values in locally convex spaces into locally convex spaces themselves. These spaces will serve as modelling space for groups of mappings, i.e., Lie group valued functions and it will turn out that the charts are given by push-forwards of suitable Lie group charts. Very briefly, we will touch upon spaces of sections. Furthermore, we state several useful smoothness result about maps between spaces of mappings and groups of mappings. Most results of these two sections are taken from \cite{GloBO}.\\
As mentioned above, we want to use vector (subspaces) of direct sums of locally convex spaces as modelling space for $\mathrm{Gau}_c(\mathcal{P})$, hence we introduce them at this point. Additionally, we define weak direct products of Lie groups which are also modelled on said sums and state various facts about differentiable mappings between these objects. We also consider direct limits to understand the topology of spaces of compactly carried maps. However, in the later parts we do not use these results very much, as the aforementioned facts about direct sums prove to be very powerful.
\bigskip

Next, we have a short part about principal bundles and trivialization systems. There are many ways to define principal bundles. To make this work more accessible for readers who are not familiar with the concept, we introduce them in a straightforward and simple way and explain some fundamentals about transition functions. Trivializing systems are covers of the base $M$ by suitable trivializing sets. Frequently we want them to have nice properties, for example to be relative compact and locally finite. Therefore, we will provide several lemmas and remarks concerning covers and trivializing systems. At the end of this section, we define the associated bundle. Later we will show that its sections can be related to the direct sums we are interested in. Again, as it is more convenient to work directly with the sums we will not need the associated bundle for our main results but it may become useful in the future because there are some results about so called almost local mappings between spaces of sections (see \cite[Theorem F.30]{GloTF}). We continue with our two main parts.
\bigskip

First, we discuss several isomorphic Lie algebras, most importantly the closed Lie subalgebra $\mathfrak{g}_{\overline{\mathcal{V}}}(\mathcal{P})$ of the direct sum $\bigoplus_{i\in\mathbb{N}}C^\infty(\overline{V}_i,\mathfrak{g})$ for an adequate trivializing system $\overline{\mathcal{V}}:=(\overline{V}_i,\sigma_i)_{i\in\mathbb{N}}$. Throughout the remainder this will be our modelling space of choice.
Next, we want to equip the gauge group with a Lie group structure. To do so we show that it is isomorphic as an abstract group to the closed subgroup $G_{\overline{\mathcal{V}}}(\mathcal{P})$ of the weak direct product $\prod^\ast_{i\in\mathbb{N}}C^\infty(\overline{V}_i,G)$. Unfortunately in the infinite-dimensional case closed subgroups are not automatically Lie subgroups, which is why we need to recall from \cite{Wockel3} the so called ``property SUB" and extend it to a certain property $\text{SUB}_\oplus$ which is better adapted to infinite direct sums. In short, this is a condition that assures the existence of a submanifold chart, modelling $G_{\overline{\mathcal{V}}}(\mathcal{P})$ on $\mathfrak{g}_{\overline{\mathcal{V}}}(\mathcal{P})$. We will prove that the Lie group structure of $G_{\overline{\mathcal{V}}}(\mathcal{P})$ is independent (up to isomorphism) of the trivializing system and we show several cases for which the property $\text{SUB}_\oplus$ holds. However, whether there exist principal bundles that do not satisfy this property remains an open question at this point.
\bigskip

In the last part, we finally turn to the group of bundle automorphisms. We want to establish the Lie group structure by an extension of Lie groups. After a short introduction to the theory of non-abelian Lie group extensions it will turn out that we need to find a so called smooth factor system. We spend the rest of this work with checking the various requirements of a smooth factor system and finally arrive at:
\begin{theorem*}
Let $\mathcal{P}$ be a smooth principal $G$-bundle over a finite-dimensional $\sigma$-compact smooth manifold $M$. If $\mathcal{P}$ has the property $\text{SUB}_\oplus$, then $\mathrm{Aut}_c(\mathcal{P})$ carries a Lie group structure such that we have an extension of smooth Lie groups
\[
\mathrm{Gau}_c(\mathcal{P})\hookrightarrow\mathrm{Aut}_c(\mathcal{P})\xtwoheadrightarrow{\mathcal{Q}}\mathrm{Diff}_c(M)_{\mathcal{P}},
\]
where $\mathcal{Q}:\mathrm{Aut}_c(\mathcal{P})\rightarrow\mathrm{Diff}_c(M)_{\mathcal{P}}$ is the natural homomorphism and $\mathrm{Diff}_c(M)_{\mathcal{P}}:=\mathrm{im}(\mathcal{Q})$. 
\end{theorem*}
In the proof of this theorem we give an explicit description of the Lie group structure of the bundle automorphisms in terms of charts. Note that short of the property $\text{SUB}_\oplus$ mentioned above there are no other requirements for the Lie group $G$ or the principal bundle $\mathcal{P}$ for this theorem to hold.
\bigskip

This work is based on \cite{Wockel3}, where the base $M$ was assumed to be compact. While the general structure of our approach is similar, some details of the proofs differ considerably. At many points the smoothness of mappings with values in direct sums must be checked. If $M$ is compact these sums are finite and it is enough to check smoothness af all components. This does not suffice in our situation, where we have to deal with infinite direct sums. In general, we will reduce the smoothness of the maps in question to the smoothness of suitable maps between direct sums, where we have powerful tools at our disposal (see for example \cite{GloMM}). Also, the ``property SUB" introduced by Wockel is too restrictive for our setting and we need to change it to the property $\text{SUB}_\oplus$ mentioned above. Furthermore, to find a smooth factor system, we have to fragment a compactly carried diffeomorphism of $M$ into diffeomorphisms which are carried on specific compact sets. For compact $M$ this fragmentation is always finite but in our situation it arises as some limit and this makes several adjustments necessary. Finally, we were able to use Theorem \ref{exp} and Theorem \ref{eval} recently proven in \cite{Hamsa} to simplify some of the smoothness arguments.

\section{Setting}
In this work all locally convex spaces are assumed to be Hausdorff $\mathbb{R}$-vector spaces. Furthermore, all Lie groups are smooth locally convex Lie groups.
\subsection{Manifolds with corners}
This section introduces the elementary notions of differential calculus on locally convex spaces for not necessarily open subsets with dense interior and manifolds modelled on such spaces. 
In the case of open subsets, this approach is known as Keller's $C^K_c$-theory, see \cite{GloComplete}, \cite{GloBO}, \cite{Michor} and \cite{Milnor} for a streamlined introduction.
However, at several points it will be necessary to cover a $\sigma$-compact manifold with compact subsets that also carry some sort of manifold structure and to this end, we introduce manifolds with corners, a special case of manifolds with rough boundary from \cite{GloBO}.
\begin{mydef}\label{2.1.2}
Let $E$ and $F$ be locally convex space, $U\subseteq E$ be open, $f\colon U\rightarrow F$ be a map and $r\in\mathbb{N}\cup\{\infty\}$. We say that $f$ is $C^r$ if it is continuous and, for all $k\in\mathbb{N}$ such that $k\leq r$, the iterated directional derivatives $d^{(k)}f(x,y_1,\ldots, y_k):=D_{y_1}\ldots D_{y_k}f(x)$ exist for all $x\in U$ and $y_1,\ldots, y_k\in E$, and define a continuous map $d^{(k)} f\colon U\times E^k\rightarrow F$. The $C^{\infty}$-maps are called \textit{smooth}.\\
More generally, if $V\subseteq E$ is a locally convex subset with dense interior we say that a continuous map $g\colon V\rightarrow F$ is $C^r$ if $g\big|_{ V^\circ}$ is $C^r$ in the sense above and for each $k\in\mathbb{N}$ with $k\leq r$ the map $d^{(k)} g\colon V^\circ\times E^k\rightarrow F$ extends to a continuous map defined on $V\times E^k$, where $d^{(0)}g:=g$.
\end{mydef}

\begin{mydef}\label{kartenwechsel}
Let $E$ be a locally convex space, $\Lambda:=(\lambda_1,\ldots,\lambda_n)$ be continuous linearly independent linear functionals on $E$ and $E^{+}_\Lambda:=\bigcap_{k=1}^n\lambda_k^{-1}(\mathbb{R}_0^{+})$. A Hausdorff space $M$ is called a \textit{smooth manifold with corners} if there exists a collection $(U_i,\varphi_i)_{i\in I}$ of homeomorphisms $\varphi_i\colon U_i\rightarrow\varphi_i(U_i)$ onto open subsets $\varphi_i(U_i)$ of $E^{+}_\Lambda$ (where $\Lambda$ can depend on $i$) such that $\bigcup_{i\in I}U_i=M$  and for each pair $\varphi_i$ and $\varphi_j$ with $U_i\cap U_j\neq\varnothing$, the coordinate change
\begin{align*}
\varphi_i(U\cap U_j)\ni x\mapsto\varphi_j\left(\varphi_i^{-1}(x)\right)\in\varphi_j(U_i\cap U_j)
\end{align*}
is smooth in the sense of \ref{2.1.2}. The set $\{\varphi_i\colon i\in I\}$ is called a (smooth) \textit{atlas} of $M$ and the maps $\varphi_i$ are called charts. Note that $E^{+}_\Lambda$ is locally convex as it is the finite intersection of locally convex spaces $\lambda_k^{-1}(\mathbb{R}_0^{+})$. Furthermore, $\varphi_j(U_i\cap U_j)$ is open in $\varphi_j(U_j)$ and thus a locally convex subset with dense interior in $E$.
\end{mydef}

\begin{bem}
Note that the set of atlases for a smooth manifold with corners can be ordered by inclusion and every atlas is contained in a unique \textit{maximal atlas}. We will always assume our manifolds with corners to be equipped with the maximal atlas.\\
Let $M$ be a smooth manifold with corners and $x\in M$. If $\varphi\colon U_\varphi\rightarrow V_\varphi$ and $\psi\colon U_\psi\rightarrow V_\psi$ are two charts of $M$ with $x\in U_\varphi\cap U_\psi$, it can be shown that $\psi(x)\in\partial V_\psi$ if and only if $\varphi(x)\in\partial V_\varphi$. Hence we can define the \textit{(formal) boundary} $\partial M$ of $M$ to consist of all $x\in M$ such that there exists a chart for which the image of $x$ lies in the boundary. A manifold with corners for which $\partial M=\varnothing$ is called a \textit{manifold without corners}.
\end{bem}

\begin{mydef}[Submanifolds with corners]\cite[cf. Chapter 3]{DiffTop}
 Let $M$ be a smooth manifold with corners modelled on a locally convex space $E$. A subset $M'\subseteq M$ is called a \textit{submanifold with corners} of $M$ if for every $x\in M$ there exists a chart $\varphi\colon U\rightarrow\varphi(U)$ with $x\in U$, a closed subspace $E'$ of $E$ and a finite system $\Lambda':=(\lambda'_1,\ldots,\lambda'_n)$ of continuous linearly independent linear functionals on $E'$, such that 
\begin{enumerate}
 \item $\varphi(U\cap M')=\varphi(U)\cap E'^{+}_{\Lambda'}$ and
 \item $\varphi(U)\cap E'^{+}_{\Lambda'}$ is an open subset of $E'^{+}_{\Lambda'}$,
\end{enumerate}
where $E'^{+}_{\Lambda'}:=\bigcap_{k=1}^n\lambda'^{-1}_k(\mathbb{R}_0^{+})$.
\end{mydef}

\begin{bem}\label{umfkt}
Let $M$ be a manifold with corners and $M'$ be a submanifold with corners of $M$. From the definition it is immediately obvious that the inclusion $i\colon M'\rightarrow M$ is smooth. \\
If $M$ is finite-dimensional, we can assume that its modelling space is $[0,\infty)^k\times\mathbb{R}^{n-k}$ for some $0\leq k\leq n$. In this situation, there are many instances where we want to find for a given $x\in M$ a relatively compact open neighbourhood $U$ of $x$ such that $\overline{U}$ is a submanifold with corners of $M$. This can be achieved as follows. We may assume w.l.o.g. that there exists a chart $\varphi\colon W\rightarrow W'$ with $x\in W$, such that $\varphi(x)\in[0,1)^k\times(0,1)^{n-k}=:U'$ and $[0,1]^n\subseteq W'$. Then $U:=\varphi^{-1}(U')$ is a relatively compact open neighbourhood of $x$ and $\overline{U}=\varphi^{-1}([0,1]^n)$ is a submanifold with corners of $M$.
\end{bem}

\begin{mydef}[Smooth mappings between manifolds]
Let $M$ and $N$ be smooth manifolds with corners modelled on locally convex spaces $E$ and $F$. A map $f\colon M\rightarrow N$ is called smooth if $f$ is continuous and, for every chart $\psi$ of $M$ and $\phi$ of $N$, the map
\[
\phi\circ f\circ\psi^{-1}\colon E\supseteq\psi(f^{-1}(U_\phi)\cap U_\psi)\rightarrow F
\]
is smooth. Note that $\psi(f^{-1}(U_\phi)\cap U_\psi)$ is an open subset of $U_\psi$ and hence a locally convex subset with dense interior in $E$. 
\end{mydef}

\begin{bem}[The tangent bundles of manifolds with corners]
Similarly to ordinary manifolds, we can define the tangent bundle $TM$ for a smooth manifold with corners $M$ modelled on a locally convex space $E$. The exact construction is a bit technical and too long to be reiterated completely at this point but the general idea is as follows: We define a certain equivalence relation between triples $(\phi,x,v)$, where $\phi$ is a chart of $M$ such that $x=\phi(p)$ for some $p\in U_\phi$ and $v\in E$. Then we equip $TM$ with the final topology with respect to the maps
\[
T\phi^{-1}\colon V_\phi\times E\rightarrow TM,\quad (T\phi^{-1})(x,y):=[\phi,x,y],
\]
where $\phi$ runs through the charts of a maximal atlas of $M$. It turns out that these $T\phi^{-1}$ define homeomorphisms and $TM$ has the structure of a smooth manifold with corners. Furthermore, $TM$ is a vector bundle with fibres $T_xM$ isomorphic to $E$. For a detailed construction see \cite{GloBO}. We define the iterated tangent bundles inductively by $T^kM:=T(T^{k-1}M)$.
\end{bem}

\begin{mydef}{(Vector bundles)}
Let $M$ be a smooth manifold with corners modelled over a locally convex space and $F$ be a locally convex space. A smooth \textit{vector bundle} over $M$ with typical fibre $F$ is a smooth manifold with corners $E$, together with a smooth surjection $\pi\colon E\rightarrow M$ and equipped with a vector space structure on each fibre $E_m:=\pi^{-1}(\{m\})$, such that for each $m_0\in M$ there exists an open neighbourhood $M_\psi$ of $m_0$ and a diffeomorphism
\[
\psi\colon\pi^{-1}(M_\psi)\rightarrow M_\psi\times F
\]
(called a ``local trivialization of $E$ about $m_0$") such that $\psi(E_m)=\{m\}\times F$ for each $m\in M_\psi$ and $\mathrm{pr}_F\circ\psi\big|_{E_m}\colon E_m\rightarrow F$ is linear (and thus an isomorphism of topological vector spaces with respect to the topology on $E_m$ induced by $E$). We will often write $\pi\colon E\rightarrow M$ for such a vector bundle.\\
We call a vector bundle \textit{trivial} if there exists a global trivialization. 
For a local trivialization like above we denote by $E\big|_{M_\psi}:=\pi^{-1}(M_\psi)$ the \textit{restricted bundle}. It is obviously a trivial vector bundle.
\end{mydef}

\begin{bem}
Let $M$ be a smooth manifold with corners and $f\colon M\rightarrow E$ be a $C^r$-map to a locally convex space $E$. Then we define the tangent map of $f$ by $Tf\colon TM\rightarrow TY= Y\times Y$ and it has the form $(x,v)\mapsto (f(x),df(x,v))$ for $x\in M$ and $v\in T_xM$. We set $d^0f:=f$ and define $d^kf\colon T^kM\rightarrow E$ recursively by $d^kf:=d(d^{k-1}f)$ for all $k\leq r$.
\end{bem}

\subsection{Spaces of mappings}

In the following, we define spaces of mappings $C^r(M,E)$ between manifolds with corners $M$ and locally convex spaces $E$ and turn them into locally convex spaces. We will also turn groups $C^r(K,G)$ of mappings between a compact manifold with corners $K$ and a Lie group $G$, into Lie groups. For both concepts several useful facts are mentioned.

\begin{mydef}[The compact-open topology]
Let $X$ and $Y$ be Hausdorff topological spaces. We write $C(X,Y)$ for the space of continuous maps from $X$ to $Y$. Given a compact subset $K\subseteq X$ and an open subset $U\subseteq Y$, we define
\[
\lfloor K, U\rfloor:=\{\gamma\in C(X,Y)\colon\gamma(K)\subseteq U\}.
\]
Then the sets
\[
\lfloor K_1,U_1\rfloor\cap\cdots\cap\lfloor K_n, U_n\rfloor,
\]
where $n\in\mathbb{N}$, $K_1,\ldots K_n$ are compact subsets of $X$ and $U_1,\ldots U_n$ open subsets of $Y$, form a basis for a topology on $C(X,Y)$, called the \textit{compact-open topology}. We write $C(X,Y)_{c.o.}$ for $C(X,Y)$ equipped with the compact-open topology. 
It can be easily shown that $C(X,Y)_{c.o.}$ is Hausdorff and the point evaluation \[
\mathrm{ev}_x\colon C(X,Y)_{c.o.}\rightarrow Y,\quad \mathrm{ev}_x(\gamma):=\gamma(x)\] 
is continuous for every $x\in X$.
\end{mydef}

\begin{mydef}[The topology of $C^r(M,E)$]
Let $r\in\mathbb{N}_0\cup\{\infty\}$ and $M$ be a smooth manifold  (with or without corners). If $E$ is a locally convex space, we let $C^r(M,E)$ be the vector space of $E$-valued $C^r$-maps on $M$. We give $C^r(M,E)$ the initial topology with respect to the mappings
\[
d^k\colon C^r(M,E)\rightarrow C(T^kM,E)_{c.o.},\quad \gamma\mapsto d^k\gamma,
\]
where $k\in\mathbb{N}_0$ such that $k\leq r$. Given a compact subset $K$ of $M$, we define
\[
C^r_K(M,E):=\left\{\gamma\in C^r(M,E)\colon\gamma\big|_{M\setminus K}=0\right\},
\]
and equip this vector subspace of $C^r(M,E)$ with the induced topology.
\end{mydef}

\begin{bem}\label{4.2.3}
The topology on $C^r(M,E)$ defined above makes the linear map
\[
C^r(M,E)\rightarrow\prod_{\mathbb{N}\ni k\leq r} C(T^kM,E)_{c.o.},\quad \gamma\mapsto(d^k\gamma)_{0\leq k\leq r}
\]
a topological embedding. This implies that $C^r(M,E)$ is a Hausdorff locally convex space. Since the inclusion map $d^0\colon C^r(M,E)\rightarrow C(M,E)_{c.o.}$ is continuous and linear and the point evaluations $\mathrm{ev}_x\colon\gamma\mapsto\gamma(x)$ on $C(M,E)_{c.o.}$ are continuous linear maps, we have that
\[
C^r_K(M,E)=\bigcap_{x\in M\setminus K}(\mathrm{ev}_x\circ d^0)^{-1}(\{0\})
\]
is a closed vector subspace of $C^r(M,E)$.
\end{bem}

\begin{lem}\label{A.14}\cite[cf. Lemma 2.2.6]{Diss}
If $M$ is a smooth manifold with corners, $E$ is a locally convex vector space and $U\subseteq M$ is an open set, then the restriction map
\[
\mathrm{res}_U\colon C^{\infty}(M,E)\rightarrow 
C^{\infty}(U,E),\quad\eta\mapsto\eta\big|_U
\]
is continuous and linear. If additionally $\overline{U}$ is a submanifold with corners of $M$, then
\[
\mathrm{res}_{\overline{U}}\colon C^{\infty}(M,E)\rightarrow C^{\infty}(\overline{U},E),\quad\eta\mapsto\eta\big|_{\overline{U}}
\]
is also continuous and linear.
\end{lem}
\begin{proof}
Obviously both maps are linear thus we only need to show that the maps are continuous. For every $n\in\mathbb{N}_0$ each compact set $K\subseteq T^nU$ or $K'\subseteq T^n\overline{U}$ is also compact in $T^nM$. Hence the lemma follows directly from the definition of the topology above.
\end{proof}

\begin{lem}\label{cont}cf. \cite[Lemma 4.24]{GloTF}
Let $M$ be a smooth manifold with corners, $K\subseteq M$ compact and $E$ be a locally convex space. If $V\subseteq M$ is open such that $K\subseteq V$, then the continuation
\[
\kappa\colon C^r_K(V,E)\rightarrow C^r_K(M,E),\quad f\mapsto\tilde{f},
\]
where
\[
\tilde{f}(x):=\begin{cases}
f(x)&\quad\text{if }x\in V\\
0&\quad\text{else}
\end{cases}
\]
is continuous and linear. Additionally we get $C^r_K(V,E)\cong C^r_K(M,E)$.
\end{lem}

\begin{lem}\cite{GloBO}\label{pullback}
Let $M$ and $N$ be smooth manifolds with corners, E be a locally convex space, and $f\colon M\rightarrow N$ be a $C^r$-map. Then the pullback
\[
(f^\ast=)C^r(f,E)\colon C^r(N,E)\rightarrow C^r(M,E),\quad\gamma\mapsto\gamma\circ f
\]
is a continuous linear mapping.
\end{lem}

\begin{lem}\cite{GloBO}
If $M$ is a compact manifold with corners, $E$ and $F$ are locally convex spaces, $U\subseteq E$ is open and $f\colon M\times U\rightarrow F$ is smooth, then the mapping
\[
f_\ast\colon C^{\infty}(M,U)\rightarrow C^{\infty}(M,F),\gamma\mapsto f\circ(\mathrm{id}_M,\gamma) 
\]
is smooth.
\end{lem}

\begin{lem}\cite{GloBO}\label{pushforward}
If $M$ is a compact finite-dimensional smooth manifold with corners, $E$ and $F$ are locally convex spaces, $U\subseteq E$ is open and $f\colon U\rightarrow F$ is smooth, then the push-forward
\[
f_\ast\colon C^{\infty}(M,U)\rightarrow C^{\infty},\gamma\mapsto f\circ\gamma
\]
is smooth.
\end{lem}

\begin{lem}\label{fstern}
Let $P$ be a smooth manifold with corners and $E$ a locally convex space. If $f\colon P\times E\rightarrow E$ is a smooth map, then the map
\[
f_\star\colon C^\infty(P,E)\rightarrow C^\infty(P,E),\quad \eta\mapsto \bigg( x\mapsto f\big(x,\eta(x)\big)\bigg)
\]
is also smooth.
\end{lem}
\begin{proof}
See \cite{GloBO} (\cite[cf. Proposition 4.16]{GloTF}).
\end{proof}

\begin{lem}\label{mult}
Let $E$ be a locally convex space, $M$ be a finite-dimensional smooth manifold with corners and $f\colon M\rightarrow\mathbb{R}$ be a $C^r$-map. Then
\[
m_f\colon C^r(M,E)\rightarrow C^r(M,E),\quad\gamma\mapsto f\cdot\gamma
\]
(pointwise product) is a continuous linear map.
\end{lem}
\begin{proof}
This follows by substituting \cite[p.6 Proposition 2.5]{GloSS} in the proof of \cite[p.8 Corollary 2.7]{GloSS} with the respective lemma for manifolds with rough boundary in \cite{GloBO}.
\end{proof}

Next, we briefly introduce spaces of sections.

\begin{mydef}
A smooth \textit{section} of a smooth vector bundle $\pi\colon E\rightarrow M$ is a smooth mapping $\sigma\colon M\rightarrow E$ such that $\pi\circ\sigma=\mathrm{id}_M$. Its \textit{support} $\mathrm{supp}(\sigma)$ is the closure of $\{x\in M\colon \sigma(x)\neq 0_x\}$. In this situation we let $C^r(M,E)$ be the set of all $C^r$-sections and $C^r_K(M,E)$ be the set of all $C^r$-sections with support contained in a given compact set $K\subseteq M$.\\
Further, if the bundle has typical fibre $F$, $\sigma\colon M\rightarrow E$ is a smooth section and $\psi\colon\pi^{-1}(M_\psi)\rightarrow M_\psi\times F$ is a local trivialization, we define $\sigma_\psi:=\mathrm{pr}_F\circ\psi\circ\sigma\big|_{M_\psi}\colon M_\psi\rightarrow F$.
\end{mydef}

\begin{bem}\cite[p.10 Lemma 3.7 and Definition 3.8]{GloSS}\label{vectorbundletop}
If $\pi\colon E\rightarrow M$ is a smooth vector bundle with typical fibre $F$, we call a set of local trivializations $\psi$ of $E$ whose domains cover $E$ an \textit{atlas of local trivializations}. Let $\mathcal{A}$ be such an atlas. Then the map
\[
\Gamma\colon C^r(M,E)\rightarrow\prod_{\psi\in\mathcal{A}}C^r(M_\psi,F),\quad\sigma\mapsto(\sigma_\psi)_{\psi\in\mathcal{A}}
\]
is an injection with closed image. We give $C^r(M,E)$ the locally convex Hausdorff topology turning this linear map into a topological embedding and for every compact subset $K\subseteq M$ we equip the subspace $C^r_K(M,E)$ with the induced topology. This topology does not depend on the choice of the atlas $\mathcal{A}$ (see \cite[p.11 Lemm 3.9]{GloSS}).
\end{bem}

The Lie groups we are interested in will be weak direct products of Lie groups of the following form. In this sense the next result forms the basis of this thesis.

\begin{satz}[The Lie group structure on $C^{\infty}(K,G)$]\cite{GloBO}
Let $K$ be a compact smooth manifold with corners and let $G$ be a Lie group modelled on the locally convex space $\mathfrak{g}$. Furthermore, let $W\subseteq G$ be a symmetric unity neighbourhood and let $\varphi\colon W\rightarrow\varphi(W)\subseteq\mathfrak{g}$ be a centered chart of $G$, i.e., $\varphi(1)=0$. Furthermore denote $\varphi_\ast\colon C^{\infty}(K,W)\rightarrow C^{\infty}(M,\mathfrak{g}),\gamma\mapsto\varphi\circ\gamma$. Then $\varphi_\ast$ induces a locally convex manifold structure on $C^{\infty}(K,G)$, turning it into a Lie group with respect to the pointwise operations.
\end{satz}

\begin{lem}\label{A.17}\cite[Lemma 2.2.20]{Diss}
If $M$ is a compact finite-dimensional manifold with corners, $G$ is a Lie group and $\overline{U}\subseteq M$ is a submanifold with corners of $M$, then the restriction
\[
\mathrm{res}\colon C^{\infty}(M,G)\rightarrow C^{\infty}(\overline{U},G),\quad\gamma\mapsto\gamma\big|_{\overline{U}}
\]
is a smooth homomorphism of Lie groups.
\end{lem}

\begin{prop}\label{A.18}\cite[Proposition 2.2.21]{Diss}
Let $G$ be a Lie group, $M$ be compact smooth manifold with an open cover $\mathcal{V}=(V_i)_{i=1,\ldots,n}$ such that $\overline{\mathcal{V}}=(\overline{V}_i)_{i=1,\ldots,n}$ is a cover by compact submanifolds with corners of $M$. Then
\[
G_{\overline{\mathcal{V}}}:=\left\{(\gamma_i)_{i=1,\ldots,n}\in\prod_{i=1}^n C^{\infty}(\overline{V}_i,G)\colon\gamma_i(x)=\gamma_j(x)\text{ for all }x\in\overline{V}_i\cap\overline{V}_j\right\}
\]
is a closed subgroup of $\prod_{i=1,\ldots,n}C^{\infty}(\overline{V}_i,G)$, which is a Lie group modelled on the closed subspace
\[
\mathfrak{g}_{\overline{\mathcal{V}}}:=\left\{(\eta_i)_{i=1,\ldots,n}\in\bigoplus_{i=1,\ldots,n}C^{\infty}(\overline{V}_i,\mathfrak{g})\colon\eta_i(x)=\eta_j(x)\text{ for all }x\in\overline{V}_i\cap\overline{V}_j\right\}
\]
of $\bigoplus_{i=1,\ldots,n}C^{\infty}(\overline{V}_i,\mathfrak{g})$ and the gluing map
\[
\mathrm{glue}\colon G_{\overline{\mathcal{V}}}\rightarrow C^{\infty}(M,G),\quad\mathrm{glue}\big((\gamma_i)_{i\in\mathbb{N}}\big)(x)=\gamma_i(x)\text{  if }x\in\overline{V}_i
\]
is an isomorphism of Lie groups.
\end{prop}

\begin{satz}\label{exp}\cite[cf.]{Hamsa}
Let $N$ be a smooth manifold, $M$ be a compact smooth manifolds with corners and $G$ a smooth Lie group. Then a map
\[
f\colon N\rightarrow C^\infty(M,G)
\]
is smooth if and only if the map
\[
\hat{f}\colon N\times M\rightarrow G,\quad \hat{f}(n,m)\mapsto f(n)(m)
\]
is smooth.
\end{satz}

\begin{satz}\label{eval}\cite[cf.]{Hamsa}
If $M$ is a smooth compact manifold with corners and $G$ is a smooth Lie group, then the evaluation map
\[
\epsilon\colon C^\infty(M,G)\times M\rightarrow G,\quad(f,m)\mapsto f(m)
\]
is smooth.
\end{satz}

\subsection{Direct limits and direct sums}
As the theory of direct limits of locally convex spaces is well known we shall only briefly reiterate some facts at this point. For more details see \cite{Jarch}, \cite{Schaef} or \cite{Bour}.
\begin{satz}\cite[cf. p.84, Theorem 1 and p.112 Proposition 9]{Jarch}\label{directlim}
Let $(E_n)_{n\in\mathbb{N}}$ be a sequence of Hausdorff locally convex spaces such that $E_n\subseteq E_{n+1}$ and the inclusion $\iota_n\colon E_n\rightarrow E_{n+1}$ is a topological embedding for all $n\in\mathbb{N}$. If $E=\bigcup_{n\in\mathbb{N}}E_n$, then the locally convex direct limit topology defined by the inclusion $i_n\colon E_n\hookrightarrow E$ turns $E$ into a Hausdorff locally convex space. Furthermore, the topology of $E$ induces the respective topology on all $E_n$ and a convex set $U\subseteq E$ is open if and only if $U\cap E_n$ is open in $E_n$ for all $n\in\mathbb{N}$. We write $E=\varinjlim E_n$.
\end{satz}
\begin{bem}\label{directlimstetig}
Let $F$ be a locally convex space. In the situation of \ref{directlim} we have that if the image of $f\colon F\rightarrow E$ is contained in some $E_n$ then $f$ is continuous if and only if $f\big|^{E_n}\colon F\rightarrow E_n$ is continuous. Conversely, if $g\colon E\rightarrow F$ is a linear map, then $g$ is continuous if and only if $g\big|_{E_n}$ is continuous for all $n\in\mathbb{N}$. Indeed, let $U\subseteq F$ be open and convex, then $g^{-1}(U)\cap E_n=(g\big|_{E_n})^{-1}(U)$ is open in $E_n$ if each $g\big|_{E_n}$ is continuous and then $g$ is continuous. On the other hand, if $g$ is continuous $g\circ i_n$ is continuous.
\end{bem}

Later we want to construct a sequence of locally convex spaces as in \ref{directlim} out of spaces of mappings that are compactly carried in a certain sense. For this the following definition will be essential.

\begin{mydef}\label{exhaust}
Let $M$ be a $\sigma$-compact finite-dimensional smooth manifold with corners. Then a sequence $(K_n)_{n\in\mathbb{N}}$ of compact subsets $K_n\subseteq M$ such that $K_n\subseteq K_{n+1}^\circ$ and $\bigcup_{n\in\mathbb{N}}K_n=M$ is called \textit{exhaustion by compact sets}. Note that an  exhaustion by compact sets always exists in this situation (\cite[p. 94, Proposition 15]{Bour2}).
\end{mydef}

\begin{bem}\label{1.6patched}
Given a $\sigma$-compact finite-dimensional smooth manifold with corners $M$, an exhaustion by compact sets $(K_n)_{n\in\mathbb{N}}$ of $M$ and a locally convex space $E$, we observe that $\iota_n\colon C^r_{K_n}(M,E)\rightarrow C^r_{K_{n+1}}(M,E)$ is a topological embedding since the respective topologies are both induced by $C^r(M,E)$. Thus the direct limit
\[
C_c^r(M,E):=\varinjlim C^r_{K_n}(M,E)
\]
carries the topology of a locally convex Hausdorff space. Since for every compact set $K\subseteq M$ there exists an $N\in\mathbb{N}$ with $K\subseteq K_n$ for all $n>N$ the topology on $C_c^r(M,E)$ does not depend on the exhaustion. In fact, if we consider the locally convex direct limit with respect to the directed set of all compact sets $K\subseteq M$, we arrive at the same topology for $C_c^r(M,E)$. In particular, if $M$ is compact the locally convex spaces $C_c^r(M,E)$, $C^r_M(M,E)$ and $C^r(M,E)$ coincide.\\
Next, we will expand the above construction to spaces of sections. Let $\pi\colon E\rightarrow M$ be a smooth vector bundle, with the locally convex space $F$ as typical fibre. With \ref{vectorbundletop} and the above we see that the space of compactly carried $C^r$-sections $C^r_c(M,E):=\varinjlim C^r_K(M,E)$ carries a locally convex Hausdorff topology that induces the respective topology on each $C^r_K(M,E)$ for $K\subseteq M$ compact. In case of smooth vector fields we shall write $\mathcal{V}_c(M)$ and $\mathcal{V}_K(M)$.
\end{bem}

A special case of locally convex direct limits are direct sums of locally convex vector spaces. They will be very important objects later in this work.

\begin{mydef}\label{vectorspacesum}
If $(E_j)_{j\in J}$ is a family of locally convex spaces, then we define the \textit{direct sum}
\[
S:=\bigoplus_{j\in J}E_j:=\left\{(x_j)_{j\in J}\in\prod_{j\in J}E_j\colon x_j=0\text{ for almost all }j\in J\right\}.
\]
We define the inclusion
\[
\lambda_j\colon E_j\rightarrow S,x\mapsto (x_i)_{i\in J}\text{, where }
\begin{cases}
x_i=x&\text{ if }i=j\\
0&\text{ else}
\end{cases}
\]
and will momentarily see that there exists a locally convex vector topology on $S$ making all $\lambda_j$ continuous.
\end{mydef}

\begin{lem}\label{boxtop}
In the situation of \ref{vectorspacesum} there exists a finest Hausdorff locally convex vector topology $\mathcal{O}$ on the direct sum $S$ such that all $\lambda_j$ are continuous. This topology has the following properties:
\begin{enumerate}[(a)]
\item A convex set $U\subseteq S$ is a zero neighbourhood in $(S,\mathcal{O})$ if and only if $U\cap E_j:=\lambda_j^{-1}(U)$ is a zero neighbourhood in $E_j$ for all $j\in J$.
\item If the index set $J$ is countable, then the boxes $\bigoplus_{j\in J}U_j:=\bigoplus_{j\in J}E_j\cap\prod_{j\in J}U_j$, where $(U_j)_{j\in J}$ is a family of zero neighbourhoods $U_j\subseteq E_j$, constitute a basis of zero neighbourhoods for $S$.
\end{enumerate}
In the case of (b) we say that $S$ carries the \textit{box topology}.
\end{lem}
\begin{proof}
The basic construction of this topology can be found in \cite[p. 78 Section 4.3]{Jarch}, where \cite[p. 80 Proposition 3]{Jarch} shows that the resulting topology is Hausdorff. The sum is a locally convex space by \cite[p. 111, Proposition 6]{Jarch} and by \cite[p. 112, Proposition 9]{Jarch} we see that we have the box topology in case $J$ is countable.
\end{proof}

\begin{bem}\label{directsumstetig}
Let $F$ be a locally convex space, $E:=\bigoplus_{n\in\mathbb{N}}E_n$ be a direct sum of locally convex spaces and $f\colon E\rightarrow F$ be linear. Along the lines of \ref{directlimstetig} we see that $f$ is continuous if and only if $f\big|_{E_n}$ is continuous for all $n\in\mathbb{N}$.
\end{bem}

\begin{lem}\label{7.1}\cite[cf. p. 993, Proposition 7.1]{GloMM}
Let $(E_i)_{i\in\mathbb{N}}$ and $(F_i)_{i\in\mathbb{N}}$ be families of locally convex spaces, with locally convex direct sums $E:=\bigoplus_{i\in\mathbb{N}}E_i$ and $F:=\bigoplus_{i\in\mathbb{N}}F_i$. Further, let $U_i\subseteq E_i$ be open zero neighbourhoods and $f_i\colon U_i\rightarrow F_i$ be $C^k$-maps for $k\in\mathbb{N}_0\cup\{\infty\}$ such that that $f_i(0)=0$. Then $U:=\bigoplus_{i\in\mathbb{N}}U_i:=E\cap\prod_{i\in\mathbb{N}}U_i$ is an open subset of $E$, and
\[
f:=\bigoplus_{i\in\mathbb{N}}f_i\colon U\rightarrow F,\quad(v_i)_{i\in\mathbb{N}}\mapsto (f_i(v_i))_{i\in\mathbb{N}}
\]
is a mapping of class $C^k$.
\end{lem}

Now, we will define the Lie group equivalent of the direct sum, the weak direct product.

\begin{mydef}\label{GloDl4.1}(\cite[cf. Section 4]{GloDL})
Let $(G_i)_{i\in\mathbb{N}}$ be a family of smooth Lie groups, then we define $\prod^\ast_{i\in\mathbb{N}}G_i\leq\prod_{i\in\mathbb{N}}G_i$ as the subgroup (of topological groups) of all families $(g_i)_{i\in\mathbb{N}}$ such that $g_i=1$ for almost all $i\in\mathbb{N}$. A \textit{box} is a set of the form $\prod^\ast_{i\in\mathbb{N}}U_i:=(\prod^\ast_{i\in \mathbb{N}}G_i)\cap\prod_{i\in\mathbb{N}}U_i$, where $U_i\subseteq G_i$ is open and $1\in U_i$ for almost all $i\in\mathbb{N}$. 
\end{mydef}

The following proposition shows that $\prod^\ast_{i\in\mathbb{N}}G_i$ can be turned into a smooth Lie group modelled on the locally convex direct sum $\bigoplus_{i\in\mathbb{N}}L(G_i)$. Furthermore, the boxes from above form a basis for the topology of $\prod^\ast_{i\in\mathbb{N}}G_i$ because $\bigoplus_{i\in\mathbb{N}}L(G_i)$ carries the box topology (see \ref{boxtop}).

\begin{prop}\label{directsum}\cite[p. 995, Proposition 7.3]{GloMM}
Let $(G_i)_{i\in\mathbb{N}}$ be a family of smooth Lie groups. Then there exists a uniquely determined Lie group 
structure on the weak direct product $\prod^\ast_{i\in\mathbb{N}} G_i$ that is modelled on the locally convex direct sum $\bigoplus_{i\in\mathbb{N}}L(G_i)$, 
such that for some charts $\varphi_i\colon R_i\rightarrow S_i\subseteq L(G_i)$ of $G_i$ around $1$ taking $1$ to $0$, the mapping
\[
\bigoplus_{i\in\mathbb{N}}S_i\rightarrow\sideset{}{^\ast}\prod_{i\in\mathbb{N}}G_i,\quad (x_i)_{i\in \mathbb{N}}\mapsto\left(\varphi^{-1}_i(x_i)\right)_{i\in\mathbb{N}}
\]
is a diffeomorphism of smooth manifolds onto an open subset of $\prod^\ast_{i\in\mathbb{N}}G_i$.
\end{prop}

Frequently we are interested in mappings from a locally convex direct sum to a weak direct product. The next lemma will be our most important tool for these cases. 

\begin{lem}\label{2.12a}
Let $(E_i)_{i\in\mathbb{N}}$ be a family of locally convex spaces and $(G_i)_{i\in\mathbb{N}}$ be a family of smooth Lie groups. If there exist open neighbourhoods $U_i\subseteq E_i$ such that $0\in U_i$ for almost all $i\in\mathbb{N}$ and smooth maps $f_i\colon U_i\rightarrow G_i$ such that $f_i(0)=1_{G_i}$, if $0\in U_i$, then the map
\[
f\colon\bigoplus_{i\in\mathbb{N}}U_i\rightarrow\sideset{}{^\ast}\prod_{i\in\mathbb{N}} G_i,\quad(x_i)_{i\in\mathbb{N}}\mapsto\left(f_i(x_i)\right)_{i\in\mathbb{N}}
\]
is smooth.
\end{lem}
\begin{proof}
If $y:=(y_i)_{i\in\mathbb{N}}\in\bigoplus_{i\in\mathbb{N}}U_i$, then we have open neighbourhoods $R_i\subseteq G_i$ of $f_i(y_i)$ such that there exist charts $\phi_i\colon R_i\rightarrow S_i\subseteq L(G_i)$. This leads to a chart of the open neighbourhood $\prod_{i\in\mathbb{N}}^\ast R_i\subseteq\prod_{i\in\mathbb{N}}^\ast G_i$ of $f(y)$ (see \cite[Proposition 7.3]{GloMM}). Since $\bigoplus_{i\in\mathbb{N}}E_i$ is a locally convex space the translation is a diffeomorphism and hence \[
\phi\colon\sideset{}{^\ast}\prod_{i\in\mathbb{N}} R_i\rightarrow\bigoplus_{i\in\mathbb{N}}S_i-(\phi_i(f_i(y_i))),
\quad(x_i)_{i\in\mathbb{N}}\mapsto\bigg(\phi_i(x_i)-\phi_i\big(f_i(y_i)\big)\bigg)_{i\in\mathbb{N}}
\]
is also a chart, mapping $f(y)$ to zero. Further, possibly after shrinking $U_i$, we assume $U_i=f_i^{-1}(R_i)$ which is open because the maps $f_i$ are continuous. Now, using translations again,
\[
g_i\colon U_i-y_i\rightarrow S_i,\quad v-y_i\mapsto\phi_i\big(f_i(v-y_i+y_i)\big)
\]
is smooth and maps zero to zero. It follows from \ref{7.1} that $g:=\bigoplus_{i\in\mathbb{N}}g_i$ is smooth and since it is a local representation of $f$ (by means of translation on the left-hand side) we proved the assertion.
\end{proof}

\begin{folg}\label{productmorph}
 Let $(G_i)_{i\in\mathbb{N}}$ and $(G'_i)_{i\in\mathbb{N}}$ be families of smooth Lie groups and let $f_i\colon G_i\rightarrow G'_i$ be Lie group morphisms for $i\in\mathbb{N}$. Then the map
\[
 f\colon \sideset{}{^\ast}\prod_{i\in\mathbb{N}} G_i\rightarrow \sideset{}{^\ast}\prod_{i\in\mathbb{N}} G'_i,\quad (g_i)_{i\in\mathbb{N}}\mapsto \Big(f_i(g_i)\Big)_{i\in\mathbb{N}}
\]
is a morphism of Lie groups.
\end{folg}
\begin{proof}
 Since $f$ is obviously a group morphism, we only need to check smoothness on a unity neighbourhood. Using a chart as in \ref{directsum} and applying \ref{2.12a} we immediately get the result.
\end{proof}

\begin{lem}\label{topemb}
Let $F$ be a locally convex space, $M$ be a smooth finite-dimensional $\sigma$-compact manifold with corners and $(V_i)_{i\in I}$ a relatively compact locally finite open cover of $M$. Then the map
\[
\Phi\colon C^{\infty}_c(M,F)\rightarrow\bigoplus_{i\in I}C^{\infty}(V_i,F),\quad\gamma\mapsto \left(\gamma\big|_{V_i}\right)_{i\in I}
\]
is a linear topological embedding with closed image.
\end{lem}
\begin{proof}
Because $M$ is $\sigma$-compact we may assume $I=\mathbb{N}$.
The map is obviously linear and hence it suffices to show that for every compact set $K\subseteq M$ the restriction $\Phi\big|_{\mathrm{C}^{\infty}_K(M,F)}$ is continuous (see \ref{directlimstetig}). Because the cover $(V_i)_{i\in\mathbb{N}}$ is locally finite there exists an $N\in\mathbb{N}$ such that we have $V_n\cap K=\varnothing$ for all $n>N$. Thus 
\[
\Phi\big(C^{\infty}_K(M,F)\big)\subseteq\prod_{i=1}^N C^{\infty}(V_i,F)\subseteq\bigoplus_{i\in\mathbb{N}}C^{\infty}(V_i,F)
\]
and since the sum on the right-hand side induces the product topology on finite products we only need to show that the components of $\Phi\big|_{\mathrm{C}^{\infty}_K(M,F)}$ are continuous. However, the components are given by $C^{\infty}_K(M,F)\rightarrow C^{\infty}(V_i,F),\gamma\mapsto\gamma\big|_{V_i}$ and are continuous because $C^{\infty}_K(M,F)$ carries the induced topology and \ref{A.14} applies.\\
Let $h_i\colon M\rightarrow\mathbb{R}$ be a partition of unity subordinated to $(V_i)_{i\in\mathbb{N}}$ and $K_i:=\mathrm{supp}(h_i)\subseteq V_i$. Now we define
\[
\Psi\colon\bigoplus_{i\in\mathbb{N}}C^{\infty}(V_i,F)\rightarrow\bigoplus_{i\in\mathbb{N}} 
C^\infty_{K_i}(V_i,F),\quad (f_i)_{i\in\mathbb{N}}\mapsto(h_i\cdot f_i)_{i\in\mathbb{N}}
\]
and show that this map is continuous. Since $\Psi$ is linear it suffices to show that 
$\Psi\big|_{C^\infty(V_i,F)}$ is continuous. Furthermore, because the image of $\Psi\big|_{C^{\infty}(V_i,F)}$ is contained in $C^{\infty}_{K_i}(V_i,F)$ we only need to show that 
$\Psi\big|_{C^{\infty}(V_i,F)}\colon C^{\infty}(V_i,F)\rightarrow C^{\infty}_{K_i}(V_i,F)$ is continuous (see \ref{directsumstetig}). But this is the case since this map is only the multiplication with $h_i$ which is smooth by \ref{mult}.\\
Thus $\Theta:=\Psi\circ\Phi$ is linear and continuous. Finally we only need to add up the components again. For this we define
\[
\Sigma\colon\bigoplus_{i\in\mathbb{N}}C_{K_i}^{\infty}(V_i,F)\rightarrow C^{\infty}_c(M,F),\quad(f_i)_{i\in\mathbb{N}}\mapsto\sum_{i\in\mathbb{N}}\tilde{f}_i,\]
where
\[\tilde{f}_i(x):=
\begin{cases}
f_i(x)&\text{ if }x\in V_i\\
0 &\text{else.}
\end{cases}
\]
If we assume $\Sigma$ is continuous we obtain
\[
(\Sigma\circ\Theta)(f)=\sum_{i\in\mathbb{N}}\left((h_i\cdot (\tilde{f}\big|_{V_i}))\right)=\sum_{i\in\mathbb{N}}h_i\cdot f=f
.\]
This implies $\Sigma\circ\Theta=\mathrm{id}_{C^{\infty}_c(M,F)}$ and thus $\Sigma\big|_{\mathrm{im}\Theta}=(\Theta\big|^{\mathrm{im}\Theta})^{-1}$ is continuous. Hence $\Theta=\Psi\circ\Phi$ is a topological embedding which implies $\Phi$ is also a topological embedding.\\
Once again $\Sigma$ is continuous if it is continuous on each summand because $\Sigma$ is linear. Since the map $C^{\infty}_{K_i}(M,F)\rightarrow C^{\infty}_c(M,F)$ is an embedding it suffices to show that $C^{\infty}_{K_i}(V_i,F)\rightarrow C^{\infty}_{K_i}(M,F),f\mapsto\tilde{f}$ is continuous but this follows from \ref{cont}.
\end{proof}

\begin{folg}\label{algebraglue}
Let $F$ be a locally convex space, $M$ be a finite-dimensional $\sigma$-compact smooth manifold with corners and $\mathcal{V}=(V_i)_{i\in\mathbb{N}}$ be a locally finite relatively compact open cover of $M$.
Then
\[
\mathfrak{g}_{\mathcal{V}}:=\left\{(\eta_i)_{i\in\mathbb{N}}\in\bigoplus_{i\in\mathbb{N}}C^{\infty}(V_i,F)\colon\eta_i(x)=\eta_j(x)\text{ for all }x\in U_i\cap U_j \right\}
\]
is a closed vector subspace of $\bigoplus_{i\in\mathbb{N}}C^{\infty}(V_i,F)$. Also, the gluing map
\[
\mathrm{glue}\colon \mathfrak{g}_{\mathcal{V}}\rightarrow C^{\infty}_c(M,F),\quad\mathrm{glue}\big((\eta_i)_{i\in\mathbb{N}}\big)(x):=\eta_i(x)\quad\text{if }x\in U_i
\]
is well-defined and the continuous inverse to the restriction map.
\end{folg}
\begin{proof}
We have $E:=\mathfrak{g}_{\mathcal{V}}=\mathrm{im}(\Phi)$, where
\[
\Phi\colon C^{\infty}_c(M,F)\rightarrow\bigoplus_{i\in\mathbb{N}}C^{\infty}(V_i,F),\quad\gamma\mapsto \left(\gamma\big|_{V_i}\right)_{i\in\mathbb{N}}
\] is the embedding from \ref{topemb}. Because point evaluations $C^\infty(V_i,F)\rightarrow F$ are continuous as well as the projections from the direct sum to $C^\infty(V_i,F)$, $E$ is closed.
 Furthermore, the gluing map is obviously well-defined and it is continuous as it is the inverse of $\Phi\big|^E$.
\end{proof}

\section{Principal Bundles and Their Associated Bundles}
Roughly speaking, principal bundles are manifolds that locally look like the product of open subsets of a base manifold $M$ and a Lie group $G$. This Lie group acts on the principal bundle and similarly to vector bundles there exists a projection onto the base space that is compatible with the action of $G$. It will turn out that the sections of this projection only differ by smooth functions with values in $G$, so called transition functions. \\
At the end, we introduce the associated bundle, a certain vector bundle that is obtained in a natural way from a given principal bundle via the adjoint action. It will turn out that there is a nice correspondence between the sections of the associated bundle and certain Lie algebras that arise from the principal bundle.
\bigskip

Above, we have already used the term locally finite cover. Although it is a well-known concept because of its importance for this work we make it explicit here.

\begin{mydef}
Let $X$ be a topological space. An \textit{open cover} is a family $(U_i)_{i\in I}$ of open sets in $X$ such that $\bigcup_{i\in I}U_i=X$. 
We define \textit{closed cover}, \textit{compact cover} and \textit{relatively compact cover} analogously. A cover is called \textit{locally finite} if for every $x\in X$ there exists an open neighbourhood $V_x$ such that only finitely many $U_i$ have nonempty intersection with $V_x$. 
\end{mydef}

\begin{lem}\label{finiteinter}
Let $X$ be a topological space and $\mathcal{V}=(V_i)_{i\in\mathbb{N}}$ a locally finite open cover of $X$. Then every compact set $K\subseteq X$ intersects only finitely many of the $V_i$. 
\end{lem}
\begin{proof}
Let $K\subseteq X$ be compact. Since $\mathcal{V}$ is locally finite, for every $x\in K$ there exists an open neighbourhood $V_x$ of $x$ such that $V_x$ intersects only finitely many of the $V_i$. These neighbourhoods cover $K$ and by compactness of $K$ there exists a finite subcover. Since every neighbourhood in this finite subcover intersects only finitely many $V_i$ there are only finitely many sets $V_i$ that intersect with $K$.
\end{proof}

\begin{folg}
Let $M$ be a finite-dimensional $\sigma$-compact smooth manifold and $\mathcal{V}=(V_i)_{i\in\mathbb{N}}$ a locally finite open cover by relative compact subsets of $M$. Then every $V_i$ intersects only finitely many sets of $\mathcal{V}$.
\end{folg}
\begin{proof}
Since $V_i\cap V_j$ is nonempty if and only if $\overline{V}_i\cap V_j$ is nonempty this follows directly from \ref{finiteinter}.
\end{proof}

\subsection{Smooth principal bundles and trivializing systems}
\begin{mydef}\label{principal}
Let $G$ be a smooth Lie group modelled on a locally convex space $E$ and let $M$ be a smooth manifold with corners. A smooth \textit{principal} $G$\textit{-bundle with base} $M$ is a smooth manifold with corners $P$ together with a smooth right action $\rho\colon P\times G\rightarrow P,\: (p,g)\mapsto p\cdot g$ and a smooth map $\pi\colon P\rightarrow M$ such that
\begin{enumerate}[(a)]
\item for all $m\in M$ there exists an open neighbourhood $U\subseteq M$ of $m$ and a smooth map $\sigma\colon U\rightarrow P$ such that $\pi\circ\sigma=\mathrm{id}_U$. A map with these properties is called a smooth \textit{section} of $P$,
\item the orbit maps $\rho_p\colon G\rightarrow P,g\mapsto\rho(p,g)$ are injective for all $p\in P$,
\item for all $p\in P$ we have $\pi^{-1}\Big(\pi(p)\Big)=\rho_p(G)$ and
\item one can choose the neighbourhood $U$ from (a) such that the map 
\[\theta^{-1}\colon U\times G\rightarrow\pi^{-1}(U),\quad(x,g)\mapsto\rho\left(\sigma(x),g\right)\] is a diffeomorphism (called a \textit{trivialization}).
\end{enumerate}
The orbits of $\rho(p,G)=\pi^{-1}(\pi(p))$ are called \textit{fibres} of $P$, thus (b) implies that $G$ acts freely on the fibres of $P$ and (c) means $G$ acts transitively on the fibres of $P$. We will write $\mathcal{P}=(G,\pi\colon P\rightarrow M)$ in short for this principal $G$-bundle and $x\cdot g:=\rho(x,g)$ for the right action. If we choose a section we will always assume that it satisfies condition (d) and call sets $U\subseteq M$ for which such a section exists \textit{trivializing}. A principal bundle is called \textit{trivial} if there exists a global trivialization.
\end{mydef}

\begin{bem}\label{transition}
Let $\mathcal{P}=(G,\pi\colon P\rightarrow M)$ be a smooth principal $G$-bundle and $\sigma_i\colon U_i\rightarrow P$, $\sigma_j\colon U_j\rightarrow P$ sections of $P$ with $U_i\cap U_j\neq\varnothing$. Additionally, let $\theta^{-1}_i\colon U_i\times G\rightarrow\pi^{-1}(U_i)$ and $\theta^{-1}_j\colon U_i\times G\rightarrow\pi^{-1}(U_j)$ be the corresponding diffeomorphisms. Because $\pi\circ\theta^{-1}_i(x,g)=x=\pi\circ\theta^{-1}_j(x,g)$ for all $x\in U_i\cap U_j$ and $g\in G$, the diagram
\[
\begin{xy}
  \xymatrix{
  					& (U_i\cap U_j)\times G  \ar[dl]_{\theta^{-1}_j|_{(U_i\cap U_j)\times G}}\ar[rd]^{\theta^{-1}_i|_{(U_i\cap U_j)\times G}}  &\\
         \pi^{-1}(U_i\cap U_j)  \ar[rr]^{\theta_{ij}}   &     &   \pi^{-1}(U_i\cap U_j)  
  }
\end{xy}
\]
commutes and for the same reason the smooth map $\theta_{ij}:=\theta_i^{-1}\circ\theta_j$ induces a smooth map $k_{ij}\colon U_i\cap U_j\rightarrow G$ such that 
\[\sigma_i(x)\cdot k_{ij}(x)=\sigma_j(x)\text{ for all }x\in U_i\cap U_j.\] 
The maps $k_{ij}$ are called \textit{transition functions} and we obviously have $k_{ij}^{-1}=k_{ji}$. Furthermore, the transition functions meet the so called \textit{cocycle condition}
\[
k_{ij}(x)k_{jl}(x)=k_{il}(x)\text{ for all }x\in U_i\cap U_j\cap U_l
\]
because $\sigma_i(x) k_{ij}(x)k_{jl}(x)=\sigma_j(x)k_{jl}(x)=\sigma_l(x)$. Occasionally, we want to lift transition functions to $P$. We do so by defining $k_{\sigma_i}(p)\in G$ as the unique group element such that $\sigma_i(\pi(p))\cdot k_{\sigma_i}(p)=p$ for all $p\in\pi^{-1}(U_i)$. The map $k_{\sigma_i}$ is well-defined since $G$ acts freely and transitively on the fibres of $P$. Because we have 
\[ 
\sigma_i(\pi(p\cdot g))\cdot k_{\sigma_i}(p\cdot g)=\sigma_i(\pi(p))\cdot k_{\sigma_i}(p\cdot g)=p\cdot g
\]
for all $g\in G$ it immediately follows that $k_{\sigma_i}(p\cdot g)=k_{\sigma_i}(p)\cdot g$, which implies $k_{\sigma_i}(\sigma_j(x))=k_{\sigma_i}(\sigma_i(x)k_{ij}(x))=k_{ij}(x)$ for all $x\in U_i\cap U_j$. Finally, $k_{\sigma_i}$ is smooth because $k_{\sigma_i}=\mathrm{pr}_2\circ\theta_i$ if $\theta_i$ is the corresponding trivialization.
\end{bem}

\begin{mydef}
Let $\mathcal{P}=(G,\pi\colon P\rightarrow M)$ be a smooth principal $G$-bundle. If $(U_i)_{i\in I}$ is an open cover of $M$ such that $U_i$ is trivializing with smooth sections $\sigma_i\colon U_i\rightarrow P$ for every $i\in I$, 
then $\mathcal{U}:=(U_i,\sigma_i)_{i\in I}$ is called a \textit{smooth open trivializing system} of $\mathcal{P}$. Note that such a system always exists because of \ref{principal}(a) and (d).\\
If additionally $\overline{U}_i$ is a submanifold with corners of $M$ for every $i\in I$ such that the sections $\sigma_i$ can be extended to smooth sections $\overline{\sigma}_i\colon\overline{U}_i\rightarrow P$, then we call $\overline{\mathcal{U}}=(\overline{U}_i,\overline{\sigma}_i)_{i\in I}$ a \textit{smooth closed trivializing system} of $\mathcal{P}$. In this case, $\mathcal{U}$ is called the trivializing system \textit{underlying} $\overline{\mathcal{U}}$.\\
If $\mathcal{U}=(U_i,\sigma_i)_{i\in I}$ and $\mathcal{V}=(V_j,\tau_j)_{j\in J}$ are two smooth open trivializing systems of $\mathcal{P}$, then $\mathcal{V}$ is an \textit{open refinement} of $\mathcal{U}$ if there exists a map $J\ni j\mapsto i(j)\in I$ such that $V_j\subseteq U_{i(j)}$ and $\tau_j=\sigma_{i(j)}|_{V_j}$, i.e., $(V_j)_{j\in J}$ is a refinement of $(U_i)_{i\in I}$ in the topological sense and the sections $\tau_j$ are obtained from sections $\sigma_i$ by restriction.\\
If $\mathcal{U}=(U_i,\sigma_i)_{i\in I}$ is a smooth open trivializing system and $\overline{\mathcal{V}}=(\overline{V}_j,\tau_j)_{j\in J}$ is a smooth closed trivializing systems of $\mathcal{P}$, then $\overline{\mathcal{V}}$ is a \textit{closed refinement} of $\mathcal{U}$ if there exists a map $J\ni j\mapsto i(j)\in I$ such that $\overline{V}_j\subseteq U_{i(j)}$ and $\tau_j=\sigma_{i(j)}|_{\overline{V}_j}$.\\
We say a closed or open smooth trivializing system $\mathcal{V}$ is a \textit{refinement} of a closed smooth trivializing system $\overline{\mathcal{U}}$ if it is a refinement of the underlying open trivializing system $\mathcal{U}$.
\end{mydef}

\begin{lem}\label{finite}
If $\mathcal{P}=(G,\pi\colon P\rightarrow M)$ is a smooth principal $G$-bundle and $M$ is a finite-dimensional $\sigma$-compact manifold, then there exists a smooth open trivializing system $\mathcal{V}=(V_n,\sigma_n)_{n\in\mathbb{N}}$ of $\mathcal{P}$ such that $(V_n)_{n\in\mathbb{N}}$ is locally finite and the sets $V_n$ are relatively compact. We call this a locally finite relatively compact smooth trivializing system of $\mathcal{P}$.
\end{lem}
\begin{proof}
Let $\mathcal{U}=(U_i,\sigma_i)_{i\in I}$ be a smooth open trivializing system of $\mathcal{P}$. We assume that $M$ is not compact because otherwise we simply can choose a finite subset of $(U_i)_{i\in I}$ to get a suitable cover. Since $M$ is $\sigma$-compact and locally compact there exists an exhaustion by compact sets $(K_n)_{n\in\mathbb{N}}$ of $M$ such that $K_n\subseteq K_{n+1}^\circ$ for all $n\in\mathbb{N}$ (see \ref{exhaust}).
Because $(U_i)_{i\in I}$ is an open cover of $M$ we may cover $K_2$ with finitely many relatively compact sets $(V_k)_{k=1,\ldots,N}$ such that for every $k\in\{1,\ldots N\}$ there exists an $i_k\in I$ with 
$V_k\subseteq U_{i_k}\cap K_3^\circ$. In the same way we cover the compact sets $K_n\setminus K_{n-1}^\circ$ with finitely many open sets $\overline{V}_l\subseteq K_{n+1}^\circ\setminus K_{n-2}\cap U_{i_l}$ for all $n\geq 3$ and thus obtain countably infinite many compact sets $(V_n)_{n\in\mathbb{N}}$ that cover $M$ and form a smooth open trivializing system by relative compact sets $\mathcal{V}=(V_n,\sigma_{i_n}|_{V_n})_{n\in\mathbb{N}}$.\\
It remains to show that $\mathcal{V}$ is locally finite. For $x\in M$ we have $x\in K_3^\circ$ or $x\in K_{n+1}^\circ\setminus K_{n-2}$ for $n\geq 3$. But per construction $x$ can at most be in two sets of this kind and only finitely many $V_n$ have a nonempty intersection with these sets, thus $\mathcal{V}$ is locally finite.\\
By our choice of $(V_n)_{n\in\mathbb{N}}$ it is obvious that $(V_n)_{n\in\mathbb{N}}$ together with the restricted sections forms a smooth open trivializing system by relatively compact open sets.
\end{proof}

\begin{folg}\label{common}
If $M$ is a $\sigma$-compact finite-dimensional manifold and $(U_i)_{i\in I}$ and $(V_j)_{j\in J}$ are two open covers, then there exists a locally finite common refinement $(W_l)_{l\in L}$ of $(U_i)_{i\in I}$ and $(V_j)_{j\in J}$ in the topological sense.  
\end{folg}
\begin{proof}
Since $(U_i)_{i\in I}$ and $(V_j)_{j\in J}$ cover $M$ so does $(U_i\cap V_j)_{i\in I, j\in J}$ and with this open cover the result follows along the lines of the proof of \ref{finite}. 
\end{proof}

\begin{bem}\label{commonb}
Note that we usually cannot find a common refinement in the sense of trivializing systems because in this case the sections of the refinement would need to arise from the sections of the trivializing system.
\end{bem}

\begin{folg}\label{refinement}
Let $\mathcal{P}=(G,\pi\colon P\rightarrow M)$ be a smooth principal $G$-bundle with finite-dimensional $\sigma$-compact base $M$. Then for any smooth open (or closed) locally finite trivializing system $\mathcal{U}=(U_i,\sigma_i)_{i\in\mathbb{N}}$ of $\mathcal{P}$, there exists an open or closed refinement $\mathcal{V}=(V_j,\tau_j)_{j\in\mathbb{N}}$ that is relatively compact (or compact) and locally finite. 
Furthermore, in the open case it can be achieved that $V_i\subseteq U_i$ and $\tau_i=\sigma_i\big|_{V_i}$. 
\end{folg}
\begin{proof}
In the open case, the first claim follows directly from \ref{finite}, whereas in the closed case, we need to choose closed submanifolds with corneres $V_i$ to cover the compact sets $K_n$ in the proof of \ref{finite}.\\
We now construct a refinement meeting the second statement. Following \ref{commonb}, we can choose $\mathcal{V}$ in a way such that there exists a surjective map $l\colon\mathbb{N}\rightarrow\mathbb{N}$ with $V_j\subseteq U_{l(j)}$. We define the finite, nonempty sets $I_n:=\{i\in\mathbb{N}\colon l(i)=n\}$ and set $W_n:=\bigcup_{i\in I_n}V_i$ with the corresponding section $\tau_n:=\sigma_n\big|_{W_n}$. This construction obviously yields a locally finite open trivializing system.
\end{proof}

\begin{bem}\label{refinementcorners}
In the situation of \ref{refinement}, we can find a refinement $\mathcal{V}$ of $\mathcal{U}$ that consists of manifolds with corners diffeomorphic to $[0,1]^n$ for some $n\in\mathbb{N}$ analogously to the proof of \ref{finite}, using \ref{umfkt}. Since unions of manifolds with corners are not automatically manifolds with corners we cannot directly achieve that $V_j\subseteq U_j$ holds as in \ref{refinement}. However, if $U_j$ is already diffeomorphic to $[0,1]^n$ via $\varphi_j$, then the closed set $\varphi_j(\bigcup_{i\in I_j}V_i)$ is contained in $[\epsilon,1-\epsilon]^n\subseteq (0,1)^n$ for some $0<\epsilon<1$.
Thus, the sets $V'_j:=\varphi_j^{-1}([\epsilon,1-\epsilon]^n)\subseteq U_j$ constitute a locally finite compact refinement such that $V'_j\subseteq U_j$ for all $j\in\mathbb{N}$.
\end{bem}

\subsection{Associated vector bundles}
\begin{mydef}\label{B.1.14}
Let $\mathcal{P}=(G,\pi\colon P\rightarrow M)$ be a smooth principal $G$-bundle over a finite-dimensional base with corners $M$, $F$ a locally convex space and let the map $\lambda\colon G\rightarrow \mathrm{GL}(F)$ be a homomorphism of groups such that \[
\hat{\lambda}\colon G\times F\rightarrow F,\quad (g,v)\mapsto\lambda(g)(v)\]
is smooth. Consider the right action 
\[
\rho\colon(P\times F)\times G\rightarrow P\times F,\quad(p,v).g:=\left(p\cdot g,\lambda(g^{-1})(v)\right).
\]
We define $P_F=P\times_G F:=P\times F/_\sim$, where $\sim$ is the equivalence relation such that $(p,v)\sim (p',v')$ if they are contained in the same orbit with respect to the right action $\rho$. We denote by $[p,v]$ the equivalence classes of $\sim$. Furthermore, we define $\pi_F\colon P_F\rightarrow M,\quad\pi([p,v]):=\pi(p)$. Note that $\pi_F$ is well-defined since if $[p,v]=[p',v']$, then there exists a $g\in G$ with $(p',v')=(p\cdot g,\lambda(g^{-1})(v))$ and then $\pi(p')=\pi(p\cdot g)=\pi(p)$.\\
Let $\mathcal{A}:=\{\theta\colon\pi^{-1}(U_\theta)\rightarrow U_\theta\times G\colon\theta\text{ is a smooth trivialization of } P\}$. We equip $P_F$ with the final topology with respect to the mappings
\[
\widetilde{\theta}^{-1}\colon U_\theta\times F\rightarrow P_F,\quad(x,v)\mapsto[\theta^{-1}(x,1),v],
\]
where $\theta\in\mathcal{A}$. 
It will turn out that $\pi_F\colon P_F\rightarrow M$ defines a vector bundle with standard fibre $F\cong\pi^{-1}_F(\{x\})$, we call it the \textit{vector bundle associated to} $\mathcal{P}$ \textit{via the action} $\hat{\lambda}$ or if there can be no confusion the \textit{associated bundle}.
\end{mydef}
We will now check the various details, showing that $P_F$ is indeed a vector bundle with standard fibre $F$ in the situation above. We proceed as in the respective proof for the tangential bundle in \cite{GloBO}.

\begin{lem}
In the situation of \ref{B.1.14}, the following holds:
\begin{enumerate}[(a)]
\item $\pi_F^{-1}(U_\theta)$ is open for every smooth trivialization $\theta\colon\pi^{-1}(U_\theta)\rightarrow U_\theta\times G$.
\item $\widetilde{\psi}\circ\widetilde{\theta}^{-1}\colon (U_\theta\cap U_\psi)\times F\rightarrow (U_\theta\cap U_\psi)\times F$ is a diffeomorphism for all $\theta$, $\psi\in\mathcal{A}$.
\item $\widetilde{\theta}\colon\pi^{-1}_F(U_\theta)\rightarrow U_\theta\times F$ is a homeomorphism.
\item $P_F$ is a smooth manifold with corners.
\item $\pi_F\colon P_F\rightarrow M$ is smooth.
\item $\pi^{-1}_F(\{x\})$ is homeomorphic to $F$ for all $x\in M$.
\end{enumerate}
\end{lem}
\begin{proof}
Let $\theta\in\mathcal{A}$ be a trivialization. Set $\widetilde{U}_\theta:=\pi_F^{-1}(U_\theta)$ and let $\widetilde{\theta}\colon\widetilde{U}_\theta\rightarrow U_\theta\times F$ be defined as above. We note that $\widetilde{\theta}$ is bijective.\\ 
(a) Let $\psi\in\mathcal{A}$ and let $\widetilde{\psi}\colon\widetilde{U}_{\psi}\rightarrow U_\psi\times F$ be defined as above, then we have the preimage $(\widetilde{\psi}^{-1})^{-1}(\widetilde{U}_\theta)=\widetilde{\psi}(\widetilde{U}_\theta\cap\widetilde{U}_\psi)=(U_\theta\cap U_\psi)\times F$ which is open in $U_\theta\times F$. Hence $\widetilde{U}_\theta\subseteq P_F$ is open by definition of the final topology since $\psi$ was arbitrary.\\
(b) Let $\theta$, $\psi\in\mathcal{A}$. W.l.o.g. we may assume $U_\theta=U_\psi$ and we have $(\theta\circ\psi^{-1})(x,g)=(x,k(x)g)$ for a smooth transition map $k\colon U_\theta\rightarrow G$. Hence
\begin{align*}
(\widetilde{\theta}\circ\widetilde{\psi}^{-1})(x,v)&=\widetilde{\theta}([\psi^{-1}(x,1),v])\\
&=\widetilde{\theta}([\theta^{-1}(x,k(x)\cdot 1),v])=\widetilde{\theta}([\theta^{-1}(x,1),\lambda(k(x))(v)])\\
&=(x,\lambda(k(x))(v)),
\end{align*}
which defines a diffeomorphism.\\
(c) By definition of the final topology, $\widetilde{\theta}^{-1}$ is continuous. We now show that $\widetilde{\theta}$ is an open map. Let $U\subseteq U_\theta\times F$ be open and let $\psi\in\mathcal{A}$. Then the set $W:=U\cap((U_\theta\cap U_\psi)\times F)$ is open in $U_\theta\times F$, whence
\begin{align*}
(\widetilde{\psi}^{-1})^{-1}(\widetilde{\theta}^{-1}(U))&=\widetilde{\psi}(\widetilde{U}_\psi\cap(\widetilde{\theta})^{-1}(U))\\
&=\widetilde{\psi}((\widetilde{\theta})^{-1}(W))=\widetilde{\psi}\circ\widetilde{\theta}^{-1}(W)
\end{align*}
is open in $U_\psi\times F$, the map $\widetilde{\psi}\circ\widetilde{\theta}^{-1}\colon (U_\theta\cap U_\psi)\times F\rightarrow  (U_\theta\cap U_\psi)\times F$ being a homeomorphism. Hence $\widetilde{\theta}(U)$ is open in $P_F$.\\
(d) In view of (a) and (c) each $\widetilde{\theta}$ is an homeomorphism from an open subset of $P_F$ onto an open subset of $M\times F$. The sets $\widetilde{U}_\theta$ obviously cover $P_F$. Since $M\times F$ is itself a manifold with corners modelled on some $[0,\infty)^n\times F$ we obtain the structure of a smooth manifold with corners on $P_F$ using (b) if we can show that the topology on $P_F$ is Hausdorff. Let $x,y\in P_F$ be two distinct points. If there exists a $\theta\in\mathcal{A}$ with $x,y\in \widetilde{U}_\theta$, then we can find disjoint open neighbourhoods $U_x,U_y\subseteq U_\theta\times F$ of $\widetilde{\theta}(x)$ and $\widetilde{\theta}(y)$. Applying $\widetilde{\theta}^{-1}$ to these sets yields disjoint open sets in $P_F$. Otherwise we may assume that there exist $\theta,\psi\in\mathcal{A}$ with $x\in\widetilde{U}_\theta$ and $y\in\widetilde{U}_\psi$ but $y\notin\widetilde{U}_\theta$ and $y\notin\widetilde{U}_\psi$ which finishes the proof. \\
(e) For $\theta\in\mathcal{A}$ we obtain the local description $\pi_F\big|_{\widetilde{U}_\theta}=\mathrm{pr}_1\circ\widetilde{\theta}$ which is smooth.\\
(f) Let $x\in U_\psi\subseteq M$. Then $\widetilde{\psi}\big|_{\pi^{-1}_F(\{x\})}\colon\pi^{-1}_F(\{x\})\rightarrow\{x\}\times F\cong F$ is a diffeomorphism.
\end{proof}

\begin{mydef}
Let $\mathcal{P}=(G,\pi\colon P\rightarrow M)$ be a smooth principal $G$-bundle. Then the adjoint representation $\mathrm{Ad}\colon G\rightarrow\mathrm{GL}(\mathfrak{g})$ defines a group operation as needed in Definition \ref{B.1.14}. The corresponding associated vector bundle is denoted by $\mathrm{Ad}(\mathcal{P})$. Its space of compactly carried smooth sections is denoted by $S^\infty_c(\mathrm{Ad}(\mathcal{P}))$. For a given compact subset $K$ of $M$ we further define $S^\infty_K(\mathrm{Ad}(\mathcal{P}))$ as the smooth sections of $\mathrm{Ad}(\mathcal{P})$ vanishing outside of $K$.
\end{mydef}

\section{The Gauge Group as an Infinite-Dimensional Lie Group}
Consider a principal bundle $\mathcal{P}$ over a $\sigma$-compact base $M$. 
In this section we describe a way to turn the group of vertical bundle automorphisms on $\mathcal{P}$, the so called gauge group, into an infinite-dimensional Lie group. We will require these automorphisms to be compactly carried on the base in a certain sense and will then be able to identify the gauge group with the space of $G$-equivariant smooth mappings $C^\infty_c(P,G)^G$.\\ As $P$ will not be compact we cannot topologise this space directly, however we can identify it with 
a closed subgroup $G_{\overline{\mathcal{V}}}(\mathcal{P})$ of the Lie group $\prod^\ast_{i\in\mathbb{N}}C^\infty(\overline{V}_i,G)$, where the components are Lie groups as all $\overline{V}_i\subseteq M$ will be compact. Unfortunately, in the infinite-dimensional case, closed subgroups are not automatically Lie subgroups. To remedy this, we extend Wockel's property SUB to $\text{SUB}_\oplus$, which enables us in many situations to equip $G_{\overline{\mathcal{V}}}(\mathcal{P})$ with a Lie group structure. Furthermore, we will prove interesting identifications for the respective Lie algebras.\\
\bigskip

Throughout this section, $M$ will be a smooth $\sigma$-compact finite-dimensional manifold with corners.

\begin{mydef}\label{1.1}
Let $G$ be a Lie group, $M$ a finite-dimensional smooth manifold and $\mathcal{P}=(G,\pi\colon P\rightarrow M)$ a smooth principal bundle with the right action $\rho_g\colon P\rightarrow P,\quad p\mapsto p\cdot g$ for all $g\in G$. We define the group of \textit{compactly carried bundle automorphisms} of $P$ by
\begin{align*}\mathrm{Aut}_c(\mathcal{P}):=\{f\in \mathrm{Diff}(P)\colon\rho_g\circ f=f\circ\rho_g \text{ for  all } g\in G\\
 \text{ and supp}(f)\subseteq\pi^{-1}(K) \text{ for some compact } K\subseteq M\}.\end{align*}
Note that this definition of being compactly carried on the base differs from the usual use of the word but there will be no ambiguities. In this context we further define the group of (smooth) compactly carried vertical bundle automorphisms or short the \textit{gauge group} of $\mathcal{P}$ as
\[\mathrm{Gau}_c(\mathcal{P}):=\left\{f\in\mathrm{Aut}_c(\mathcal{P})\colon\pi\circ f=\pi\right\}.\]
This means that $\mathrm{Gau}_c(\mathcal{P})$ keeps the fibres invariant.
\end{mydef}

\begin{bem}\label{1.1b}
If we identify $M$ with the orbit space $P/G$, then each $F\in\mathrm{Aut}_c(\mathcal{P})$ gives rise to an element $F_M\in\mathrm{Diff}_c(M):=\{f\in\mathrm{Diff}(M):f|_{M\setminus K}=\mathrm{id}|_{M\setminus K},K\subseteq M\text{ compact}\}$ by $F_M(p\cdot G):=F(p)\cdot G$ for all $p\in P$. 
Let $\sigma\colon U\rightarrow P$ be a smooth section.
Then $F_M(x)=\pi(F(\sigma(x)))$ holds, thus $F_M\big|_U$ is smooth and since the section was arbitrary this implies that $F_M$ is smooth.. 
In addition, we have 
\[
F_M\circ F'_M(p\cdot G)=F_M(F'(p)\cdot G)=(F\circ F')_M(p\cdot G)
\] 
which entails $F_M^{-1}=(F^{-1})_M$ and leads to a homomorphism
\[
\mathcal{Q}\colon\mathrm{Aut}_c(\mathcal{P})\rightarrow\mathrm{Diff}_c(M),\quad F\mapsto F_M
\]
because the support of $F_M$ is obviously compact. The gauge group is clearly the kernel of $\mathcal{Q}$. We denote by $\mathrm{Diff}_c(M)_\mathcal{P}$ the image of $\mathcal{Q}$.
\end{bem}

\begin{mydef}\label{1.2}
In the situation above, we denote by
\begin{align*}
C_c^{\infty}(P,G)^G:=\{\gamma\in &C^{\infty}(P,G)\colon\gamma(p\cdot g)=g^{-1}\cdot\gamma(p)\cdot g\text{ for all }p\in P,g\in G\text{ and }\\
&\text{supp}(\gamma)\subseteq\pi^{-1}(K)\text{ for some compact }K\subseteq M\}
\end{align*}
the group of $G$-equivariant compactly carried smooth maps from $P$ to $G$. We will often identify $\mathrm{Gau_c(\mathcal{P})}$ and $C_c^{\infty}(P,G)^G$ through the following isomorphism:
\[
C_c^{\infty}(P,G)^G\ni f\mapsto\Big(p\mapsto p\cdot f(p)\Big)\in\mathrm{Gau}_c(\mathcal{P}).
\]
Note that the latter map is indeed an element of the gauge group because $\pi(p\cdot f(p))=\pi(p)$, and a compactly carried $f$ obviously translates to a compactly carried map in the gauge group.\\
Conversely, if $\gamma\in\mathrm{Gau}_c(\mathcal{P})$, then there exists a map $f\colon P\rightarrow G$ with $\gamma(p)=p\cdot f(p)$ because the elements of the gauge group keep the fibres constant and $G$ acts freely and transitively on the fibres. This map must be $G$-equivariant because $\gamma(p\cdot g)=p\cdot g\cdot f(p\cdot g)=\gamma(p)\cdot g=p\cdot f(p)g$ implies $f(p\cdot g)=g^{-1}f(p)g$ for all $p\in P$ and $g\in G$. 
To see that $f$ is smooth let $\sigma\colon U\rightarrow P$ be a smooth section with the respective diffeomorphism $\theta\colon\pi^{-1}(U)\rightarrow U\times G$. For $p:=\sigma(x)\cdot g\in\pi^{-1}(U)$, we have $\gamma(p)=p\cdot f(p)=\sigma(x)\cdot gf(p)=\theta^{-1}(x,gf(p))$ which implies 
$gf(p)=\mathrm{pr}_2(\theta(\gamma(p)))=k_\sigma(\gamma(p))$ and leads to
\[
f(\theta^{-1}(x,g))=f(p)=g^{-1}k_\sigma(\gamma(p))=g^{-1}k_\sigma\Big(\gamma\big(\theta^{-1}(x,g)\big)\Big)
\] 
and hence, $f\big|_{\pi^{-1}(U)}$ is smooth. Again, a compactly carried $\gamma$ carries over to a compactly carried $f$.
\end{mydef}

Next we will define the algebraic counterpart of the gauge group, the gauge algebra. As we will topologise $C^\infty_c(P,G)^G$ via certain weak products of Lie groups and the respective direct sums of Lie algebras, we will not need the gauge algebra in the later parts, however it  may be useful in future works. 

\begin{mydef}
For a smooth principal bundle $\mathcal{P}=(G,\pi\colon P\rightarrow M)$ over the $\sigma$-compact finite-dimensional manifold with corners $M$ and $K\subseteq M$ compact we define the space
\[
C^{\infty}_K(P,\mathfrak{g}):=\left\{\eta\in C^{\infty}(P,\mathfrak{g})\colon\mathrm{supp}(\eta)\subseteq\pi^{-1}(K)\right\}
\]
and equip it with the induced topology of $C^{\infty}(P,\mathfrak{g})$. With the same argument as in \ref{4.2.3} we see that $C^{\infty}_K(P,\mathfrak{g})$ is a closed vector subspace of $C^{\infty}(P,\mathfrak{g})$. Now, 
\[
C^{\infty}_K(P,\mathfrak{g})^G:=\left\{\eta\in C^{\infty}_K(P,\mathfrak{g})\colon\eta(p\cdot g)=\mathrm{Ad}(g^{-1}).\eta(p)
\text{ for all }p\in P,g\in G\right\}
\]
is a closed vector subspace of $C^{\infty}_K(P,\mathfrak{g})$ (cf. \ref{1.4b}) and it is a Lie algebra with the pointwise Lie bracket because $[\mathrm{Ad}(g).h,\mathrm{Ad}(g).h]=\mathrm{Ad}(g).[h,h']$ for all $g\in G$ and $h,h'\in\mathfrak{g}$. Since for $K,K'\subseteq M$ compact with $K\subseteq K'$ the maps $C^{\infty}_K(P,\mathfrak{g})^G\hookrightarrow C^\infty_{K'}(P,\mathfrak{g})^G$ are continuous we can continue as in \ref{1.6patched} and define the \textit{gauge algebra} by
\[
\mathfrak{gau}_c(\mathcal{P}):=C^{\infty}_c(P,\mathfrak{g})^G:=\varinjlim C^{\infty}_K(P,\mathfrak{g})^G.
\]
\end{mydef}
The next lemma gives us a topological Lie algebra isomorphic to $\mathfrak{gau}_c(\mathcal{P})$ that will indirectly serve as the modelling space for the gauge group.
\begin{lem}\label{1.4b}
Let $\mathcal{P}=(G,\pi\colon P\rightarrow M)$ be a smooth principal $G$-bundle over the $\sigma$-compact finite-dimensional manifold $M$. If $\overline{\mathcal{V}}:=(\overline{V}_i,\sigma_i)_{i\in\mathbb{N}}$ is a locally finite compact trivializing system of $\mathcal{P}$ with transition functions $k_{ij}\colon\overline{V}_i\cap\overline{V}_j\rightarrow G$, we denote by $\mathfrak{g}_{\overline{\mathcal{V}}}(\mathcal{P})$ the set
\[
\left\{(\eta_i)_{i\in\mathbb{N}}\in\bigoplus_{\substack{i\in\mathbb{N}}}C^{\infty}(\overline{V}_i,\mathfrak{g})\colon\eta_i(m)=\mathrm{Ad}\left(k_{ij}(m)\right).\eta_j(m)\text{ for all }m\in\overline{V}_i\cap\overline{V}_j\right\}.
\]
Then $\mathfrak{g}_{\overline{\mathcal{V}}}(\mathcal{P})$ is a closed topological Lie subalgebra of $\bigoplus_{\substack{i\in\mathbb{N}}}C^{\infty}(\overline{V}_i,\mathfrak{g})$.
\end{lem}
\begin{proof}
It is obvious that $\mathfrak{g}_{\overline{\mathcal{V}}}(\mathcal{P})$ is a topological subspace if we endow it with the induced topology. What is more, for $(\gamma_i)_{i\in\mathbb{N}}$, $(\eta_i)_{i\in\mathbb{N}}\in\mathfrak{g}_{\overline{\mathcal{V}}}(\mathcal{P})$ we have $[\gamma_i,\eta_i](m)=[\mathrm{Ad}(k_{ij}(m)).\gamma_j,\mathrm{Ad}(k_{ij}(m)).\eta_j](m)=\mathrm{Ad}(k_{ij}(m)).[\gamma_j,\eta_j](m)$ for all $i,j\in\mathbb{N}$ and  $m\in\overline{V}_i\cap\overline{V}_j$ and thus $\mathfrak{g}_{\overline{\mathcal{V}}}(\mathcal{P})$ is a topological Lie subalgebra.\\
Because the evaluation map $\mathrm{ev}_{i,m}\colon C^{\infty}(\overline{V}_i,\mathfrak{g})\rightarrow\mathfrak{g}$ is continuous for every component and equalizers of continuous maps are closed, the set 
\[
C:=\bigcap_{i,j\in\mathbb{N}}\bigcap_{m\in\overline{V}_i\cap\overline{V}_j}\left\{(\eta_k)_{k\in\mathbb{N}}\in\bigoplus_{k\in\mathbb{N}}C^{\infty}(\overline{V}_i,\mathfrak{g})\colon\mathrm{ev}_{i,m}\circ\eta_i=\Big(\mathrm{Ad}\big(k_{ij}(m)\big)\big(\mathrm{ev}_{j,m}(\eta_j)\big)\Big)\right\}
\]
is closed and we have $C=\mathfrak{g}_{\overline{\mathcal{V}}}(\mathcal{P})$.
\end{proof}

\begin{prop}\label{1.4}
Let $\mathcal{P}=(G,\pi\colon P\rightarrow M)$ be a smooth principal $G$-bundle over a $\sigma$-compact finite-dimensional manifold $M$. If $\overline{\mathcal{V}}:=(\overline{V}_i,\sigma_i)_{i\in\mathbb{N}}$ is a locally finite compact trivializing system of $\mathcal{P}$,   $k_{ij}\colon\overline{V}_i\cap\overline{V}_j\rightarrow G$ are the respective transition functions and $\mathcal{V}$ denotes the smooth, relatively compact open trivializing system underlying $\overline{\mathcal{V}}$, then we let $\mathfrak{g}_{\mathcal{V}}(\mathcal{P})$ be the set
\[
\left\{(\eta_i)_{i\in\mathbb{N}}\in\bigoplus_{\substack{i\in\mathbb{N}}}C^{\infty}(V_i,\mathfrak{g})\colon\eta_i(m)=\mathrm{Ad}\left(k_{ij}(m)\right).\eta_j(m)\text{ for all }m\in V_i\cap V_j\right\},
\]
and we have isomorphisms of topological vector spaces
\[\mathfrak{gau}_c(\mathcal{P})=C^{\infty}_c(P,\mathfrak{g})^G\cong\mathfrak{g}_{\overline{\mathcal{V}}}(\mathcal{P})\cong\mathfrak{g}_{\mathcal{V}}(\mathcal{P})\cong \mathrm{S}^\infty_c\left(\mathrm{Ad}(\mathcal{P})\right).\]
Furthermore, each of these spaces is a locally convex Lie algebra in a natural way and the isomorphisms are isomorphisms of topological Lie algebras.
\end{prop}
\begin{proof} 
For the first two isomorphisms it suffices to show that we have a sequence
\[C^{\infty}_c(P,\mathfrak{g})^G\xrightarrow{\Psi_1}\mathfrak{g}_{\overline{\mathcal{V}}}(\mathcal{P})\xrightarrow{\Psi_2}\mathfrak{g}_{\mathcal{V}}(\mathcal{P})\xrightarrow{\Psi_3} C^{\infty}_c(P,\mathfrak{g})^G\]
with continuous linear maps that are bijective such that $\Psi_3\circ\Psi_2\circ\Psi_1=\mathrm{id}_{C^{\infty}_c(P,\mathfrak{g})^G}$ holds. We begin by showing that the linear map
\[
\Psi_1\colon C^{\infty}_c(P,\mathfrak{g})^G\rightarrow\mathfrak{g}_{\overline{\mathcal{V}}}(\mathcal{P}),\quad\eta\mapsto(\eta_i)_{i\in\mathbb{N}}:=(\eta\circ\sigma_i)_{i\in\mathbb{N}}
\]
defines an element of $\mathfrak{g}_{\overline{\mathcal{V}}}(\mathcal{P})$, is continuous and bijective. In fact $\sigma_i(m)=\sigma_j(m)\cdot k_{ji}(m)$ for $m\in\overline{V}_i\cap\overline{V}_j$ implies
\begin{align}\label{formula5.6}
\eta_i(m)=\eta(\sigma_i(m))=&\eta\left(\sigma_j(m)\cdot k_{ji}(m)\right)=\mathrm{Ad}\left(k_{ji}(m)\right)^{-1}.\eta\left(\sigma_j(m)\right)=\\
&\mathrm{Ad}\left(k_{ij}(m)\right).\eta_j(m), \nonumber
\end{align}
and thus $(\eta_i)_{i\in\mathbb{N}}$ meets the transitional condition of  $\mathfrak{g}_{\overline{\mathcal{V}}}(\mathcal{P})$. Let $(L_n)_{n\in\mathbb{N}}$ be a compact exhaustion of $M$. Since there exists an $L_n$ with $\mathrm{supp}(\eta)\subseteq \pi^{-1}(L_n)$ we have $\eta\circ\sigma_i=0$ if $\overline{V}_i\cap L_n=\varnothing$ and this holds for almost all $i\in\mathbb{N}$, thus $(\eta_i)_{i\in\mathbb{N}}\in\mathfrak{g}_{\overline{\mathcal{V}}}(\mathcal{P})$. Because the map is obviously linear we only need to show that $\Psi_1\big|_{C^{\infty}_{L_n}(P,\mathfrak{g})^G}$ is continuous for every $n\in\mathbb{N}$ (see \ref{directlimstetig}). But for the same reasons as above for each $n\in\mathbb{N}$ there exists an $N\in\mathbb{N}$ with $\Psi_1\left(C^{\infty}_{L_n}(P,\mathfrak{g})^G\right)\subseteq\prod_{j=1}^NC^{\infty}(\overline{V}_j,\mathfrak{g})\cap\mathfrak{g}_{\overline{\mathcal{V}}}(\mathcal{P})$. On this finite product $\bigoplus_{i\in\mathbb{N}}C^{\infty}(\overline{V}_i,\mathfrak{g})$ 
induces the product topology. It follows that 
$\Psi_1\big|_{C^{\infty}_{L_n}(P,\mathfrak{g})^G}$ is continuous if its components are continuous. However, this holds because the topology on $C^{\infty}_{L_n}(P,\mathfrak{g})^G$ is induced by $C^{\infty}(P,\mathfrak{g})$ and pullbacks along smooth maps are smooth by \ref{pullback}.\\ 
It remains to show that $\Psi_1$ is bijective. As $\Psi_1$ is clearly injective, it remains to show its surjectivity. To this end let $k_{\sigma_i}\colon\pi^{-1}(\overline{V}_i)\rightarrow G$ be given by $p=\sigma_i\left(\pi(p)\right)\cdot k_{\sigma_i}(p)$ and let $(\eta_i)_{i\in\mathbb{N}}\in\mathfrak{g}_{\mathcal{\overline{V}}}(\mathcal{P})$. Then the map
\begin{align}\label{formula5.6b}
\eta:P\rightarrow\mathfrak{g},\quad p\mapsto\mathrm{Ad}\big(k_{\sigma_i}(p)\big)^{-1}.\eta_i\big(\pi(p)\big)\text{  if }\pi(p)\in\overline{V}_i
\end{align}
is well-defined and defines an inverse of $\Psi_1$. In fact, $k_{\sigma_j}(p)=k_{ji}(\pi(p))k_{\sigma_i}(p)$ 
shows that this map is well-defined by (\ref{formula5.6}). It is obviously smooth, $G$-equivariant and we have  $\eta\circ\sigma_i=\eta_i$.\\
Since for every $i\in\mathbb{N}$, the restriction map $\mathrm{res}_{V_i}\colon C^{\infty}(\overline{V}_i,\mathfrak{g})\rightarrow C^{\infty}(V_i,\mathfrak{g})$ is linear and continuous by \ref{A.14}, we can conclude that
\[
\mathrm{res}\colon\bigoplus_{i\in\mathbb{N}}C^{\infty}(\overline{V}_i,\mathfrak{g})\rightarrow\bigoplus_{i\in\mathbb{N}}C^{\infty}(V_i,\mathfrak{g}),\quad(\eta_i)_{i\in\mathbb{N}}\mapsto(\eta_i\big|_{V_i})_{i\in\mathbb{N}}
\] is linear and continuous, hence the linear map
\[ \Psi_2\colon\mathfrak{g}_{\overline{\mathcal{V}}}(\mathcal{P})\rightarrow\mathfrak{g}_{\mathcal{V}}(\mathcal{P}),\quad(\eta_i)_{i\in\mathbb{N}}\mapsto (\eta_i|_{V_i})_{i\in\mathbb{N}}
\] is also continuous. Moreover, this map is clearly injective. \\
To define $\Psi_3\colon\mathfrak{g}_{\mathcal{V}}(\mathcal{P})\rightarrow C^\infty_c(P,\mathfrak{g})^G$, let $h_n\in C^{\infty}(M,[0,1])$ be a partition of unity subordinate to $V_n$, i.e., $K_n:=\mathrm{supp}(h_n)\subseteq V_n$. Now following the same steps as in the proof of \ref{topemb}, the linear map
\[
\Theta\colon\bigoplus_{i\in\mathbb{N}}C^{\infty}(V_i,\mathfrak{g})\rightarrow\bigoplus_{i\in\mathbb{N}}C^\infty_{K_i}(V_i,\mathfrak{g}),\quad(\eta_i)_{i\in\mathbb{N}}\mapsto (h_i\cdot\eta_i)_{i\in\mathbb{N}}
\]
is continuous. Let us show that the linear map
\[
\Sigma\colon\bigoplus_{i\in\mathbb{N}}C_{K_i}^\infty(V_i,\mathfrak{g})\rightarrow C^{\infty}_c(P,\mathfrak{g}),\quad(\eta_i)_{i\in\mathbb{N}}\mapsto\sum_{i\in\mathbb{N}}\tilde{\eta}_i,
\]
where
\[
\tilde{\eta}_i(p):=\begin{cases}
\mathrm{Ad}\left(k_{\sigma_i}(p)\right).\eta_i\left(\pi(p)\right)&\text{ for }p\in\pi^{-1}(V_i)\\ 
0&\text{ else,}
\end{cases}
\]
is continuous.
Since $\Sigma$ is linear we only need to check this for $\Sigma_{K_i}:=\Sigma\big|_{C^\infty_{K_i}(V_i,\mathfrak{g})}$. The image of $\Sigma_{K_i}$ is contained in $C_{K_i}^{\infty}(P,\mathfrak{g})$. We can directly adapt \ref{cont} to the different notion of being compactly carried in $C^\infty(P,\mathfrak{g})$, hence $C^\infty_{K_i}(P,\mathfrak{g})\cong C^\infty_{K_i}(\pi^{-1}(V_i),\mathfrak{g})$ and we may assume that $\mathcal{P}=(G,\pi\colon P\rightarrow M)$ is trivial with a smooth global section $\sigma\colon M\rightarrow P$. Now, it suffices to show that
\[
C^\infty(M,\mathfrak{g})\rightarrow C^{\infty}(P,\mathfrak{g}),\quad\eta\mapsto \Big(p\mapsto\mathrm{Ad}\big(k_{\sigma}(p)\big)^{-1}.\eta\big(\pi(p)\big)\Big)
\]
is continuous. Since the pullback $C^\infty(M,\mathfrak{g})\rightarrow C^\infty(P,\mathfrak{g}),\quad \eta\mapsto\eta\circ\pi$ is continuous linear and hence smooth, this follows from the smoothness of $f\colon P\times\mathfrak{g}\rightarrow\mathfrak{g}$, $f(p,y):=\mathrm{Ad}(k_{\sigma}(p))^{-1}.y$ because then \ref{fstern} implies that 
\[
f_\ast\colon C^\infty(P,\mathfrak{g})\rightarrow C^\infty(P,\mathfrak{g}),\quad f_\ast\big(\gamma\big)(p)=\mathrm{Ad}(k_{\sigma}(p))^{-1}.\gamma(p)
\]
is smooth. 
Note that $\big((\Sigma\circ\Theta)\big|_{\mathfrak{g}_{\mathcal{V}}(\mathcal{P})}(\eta_i)_{i\in\mathbb{N}}\big)(p)=\sum_{i\in\mathbb{N}}h_i\cdot\tilde{\eta}_i(p)=\mathrm{Ad}\left(k_{\sigma_i}(p)\right).\eta_i\left(\pi(x)\right)$ for $p\in\pi^{-1}(V_i)$ because formula (\ref{formula5.6b}) 
is well-defined, independent of $i$ with $\pi(p)\in V_i$. 
Thus $\Psi_2\circ\Psi_1\circ(\Sigma\circ\Theta)=\mathrm{id}_{\bigoplus_{i\in\mathbb{N}}C^{\infty}(V_i,\mathfrak{g})}$, whence $\Psi_2$ is surjective and hence bijective. We set $\Psi_3:=\Sigma\circ\Theta$. By the proceeding, $\Psi_3$ is injective. Also calculating $\Psi_3((\eta_i)_{i\in\mathbb{N}})\circ\sigma_i=\eta_i$, we see that $\Psi_3\circ\Psi_2\circ\Psi_1=\mathrm{id}_{C^{\infty}_c(P,\mathfrak{g})}$ and thus $\Psi_3$ is surjective.\\
Next we will show the isomorphism $S_c^{\infty}(\mathrm{Ad}(\mathcal{P}))\cong\mathfrak{g}_{\mathcal{V}}(\mathcal{P})$ of topological vector spaces. Following \cite[Proposition 2.2.7 and Proposition 2.2.10]{Diss} the map 
\[
S_c^{\infty}(\mathrm{Ad}(\mathcal{P}))\hookrightarrow\bigoplus_{i\in\mathbb{N}}C^{\infty}(\mathrm{Ad}(\mathcal{P})\big|_{V_i}),\quad\sigma\mapsto(\sigma\big|_{V_i})_{i\in\mathbb{N}}
\]
is an embedding with closed image. Using the trivialization 
\[\widetilde{\theta}^{-1}_i\colon V_i\times\mathfrak{g}\rightarrow\mathrm{Ad}(\mathcal{P})\big|_{V_i},\quad (x,v)\mapsto[\sigma_i(x),v],\]
 we can identify the space of sections $C^{\infty}(\mathrm{Ad}(\mathcal{P})\big|_{V_i})$ with $C^{\infty}(V_i,\mathfrak{g})$ (see \ref{vectorbundletop}). It only remains to check that the image of $S_c^{\infty}(\mathrm{Ad}(\mathcal{P}))$ equals $\mathfrak{g}_{\mathcal{V}}(\mathcal{P})$ under this identification. Let $x\in V_i\cap V_j$ and $v\in\mathfrak{g}$. Then we get 
\[
[\theta^{-1}_i(x,1),v]=[\sigma_j(x)k_{ji}(x)\cdot 1,v]=[\sigma_j(x)\cdot 1,\mathrm{Ad}\left(k_{ji}(x)\right).v]
\]
\[
=[\theta^{-1}_j(x,1),\mathrm{Ad}\left(k_{ji}(x)\right).v],
\]
which shows that the transitional property of $\mathfrak{g}_{\mathcal{V}}(\mathcal{P})$ holds if and only if the sections of $\bigoplus_{i\in\mathbb{N}}C^{\infty}(\mathrm{Ad}(\mathcal{P})\big|_{V_i})$ coincide on the intersection of their domains. \\
By virtue of the pointwise definition of the Lie bracket it is obvious that the isomorphisms 
\[
S_c^{\infty}\left(\mathrm{Ad}(\mathcal{P})\right)\cong C^{\infty}_c(P,\mathfrak{g})^G\cong\mathfrak{g}_{\overline{\mathcal{V}}}(\mathcal{P})\cong \mathfrak{g}_{\mathcal{V}}(\mathcal{P})
\] are isomorphisms of abstract Lie algebras.
Now, since $\mathfrak{g}_{\mathcal{V}}(\mathcal{P})$ is a topological Lie algebra and the isomorphisms above are isomorphisms of locally convex spaces it follows that they are isomorphisms of topological Lie algebras.
\end{proof}

We now define the weak direct product that we want to use to topologise the gauge group. 

\begin{mydef}\label{1.5}
If $\mathcal{P}$ is a smooth $G$-bundle over a $\sigma$-compact finite-dimensional base $M$ and $\overline{\mathcal{V}}=(\overline{V}_i,\sigma_i)_{i\in\mathbb{N}}$ is a compact locally finite trivializing system with corresponding transition functions $k_{ij}\colon\overline{V}_i\cap\overline{V}_j\rightarrow G$, then we denote by $G_{\overline{\mathcal{V}}}(\mathcal{P})$ the set
\[\left\{(\gamma_i)_{i\in\mathbb{N}}\in\sideset{}{^\ast}\prod_{i\in\mathbb{N}}
     C^{\infty}(\overline{V}_i,G)
\colon\gamma_i(m)=k_{ij}(m)\gamma_j(m)k_{ji}(m)\text{ for all }m\in\overline{V}_i\cap\overline{V}_j\right\}\]
and turn it into a group with respect to pointwise group operations.
\end{mydef}

\begin{bem}\label{1.6}
In the situation of \ref{1.5}, the map
\begin{align*}
&\Psi\colon G_{\overline{\mathcal{V}}}(\mathcal{P})\rightarrow C^{\infty}_c(P,G)^G,   \\
&\Psi\big((\gamma_i)_{i\in\mathbb{N}}\big)(p):=k^{-1}_{\sigma_i}(p)\cdot\gamma_i\left(\pi(p)\right)
\cdot k_{\sigma_i}(p)\text{   if }\pi(p)\in\overline{V}_i
\end{align*}
is an isomorphism of abstract groups. The right-hand side is well-defined because 
$k_{\sigma_i}(p)=k_{ij}\left(\pi(p)\right)k_{\sigma_j}(p)$ if $\pi(p)\in\overline{V}_i\cap\overline{V}_j$, 
which implies
\begin{align*}
k^{-1}_{\sigma_i}(p)\gamma_i\left(\pi(p)\right) k_{\sigma_i}(p)&=
k_{\sigma_j}(p)^{-1}\
\underbrace{k_{ji}\left(\pi(p)\right)\gamma_i\left(\pi(p)\right)k_{ij}\left(\pi(p)\right)}_
{\gamma_j\left(\pi(p)\right)}
k_{\sigma_j}(p)\\
&=k^{-1}_{\sigma_j}(p)\gamma_j\left(\pi(p)\right) k_{\sigma_j}(p)
\end{align*}
and $\Psi\big((\gamma_i)_{i\in\mathbb{N}}\big)$ is equivariant because
\[
 \Psi\big((\gamma_i)_{i\in\mathbb{N}}\big)(\sigma_i(x)g)=g^{-1}\gamma_i(x)g=g^{-1}\Psi\big((\gamma_i)_{i\in\mathbb{N}}\big)(\sigma_i(x))g.
\]
In particular, this implies that $\Psi((\gamma_i)_{i\in\mathbb{N}})$ is smooth. Obviously,
\begin{align*}
&\Psi^{-1}\colon C^{\infty}_c(P,G)^G\rightarrow G_{\overline{\mathcal{V}}}(\mathcal{P}),\\
&\Psi^{-1}\big(\eta\big):=\big(\eta\circ\sigma_i\big)_{i\in\mathbb{N}}
\end{align*}
is the inverse function of $\Psi$. In view of \ref{GloDl4.1}, we see that $G_{\overline{\mathcal{V}}}(\mathcal{P})$ is closed subgroup of $\prod^\ast_{i\in\mathbb{N}}C^{\infty}(\overline{V}_,G)$ in the same way as in \ref{1.4b} because the point evaluation $\mathrm{ev}_m\colon C^{\infty}(\overline{V}_i,G)\rightarrow G$ is continuous for all $m\in\overline{V}_i$. Note that this does not automatically turn it into a Lie group since we are in the infinite-dimensional case.
\end{bem}

\subsection{The property  \texorpdfstring{$\mathrm{SUB}_\oplus$}{SUB}}

The following definition provides a sufficient condition for turning $G_{\overline{\mathcal{V}}}(\mathcal{P})$ into a closed Lie subgroup of $\prod^\ast_{i\in\mathbb{N}}C^\infty(\overline{V}_i,G)$. Then also $C^\infty_c(P,G)^G$ can be made a Lie group using the isomorphism from \ref{1.6} (and hence also $\mathrm{Gau}_c(\mathcal{P})\cong C^\infty_c(P,G)^G$).
It will turn out that quite different properties of $\mathcal{P}$ ensure this requirement. 

\begin{mydef}
If $\mathcal{P}$ is a smooth $G$-bundle over a $\sigma$-compact finite-dimensional base $M$ and $\overline{\mathcal{V}}=(\overline{V}_i,\sigma_i)_{i\in\mathbb{N}}$ is a locally finite compact trivializing system, then we say that $\mathcal{P}$ has the \textit{property} $\text{SUB}_\oplus$ with respect to $\overline{\mathcal{V}}$ if there exists a family of convex centered charts $\varphi:=(\varphi_i\colon W_i\rightarrow W'_i)_{i\in\mathbb{N}}$ of $G$ with $W'_i\subseteq\mathfrak{g}$ such that $d\varphi_i\big|_{\mathfrak{g}}=\mathrm{id}_{\mathfrak{g}}$ for all $i\in\mathbb{N}$ and
\[\varphi_{\ast}:=\Big((\varphi_i)_\ast\Big)_{i\in\mathbb{N}}\colon G_{\overline{\mathcal{V}}}(\mathcal{P})\cap
   \sideset{}{^\ast}\prod_{i\in\mathbb{N}}C^{\infty}(\overline{V}_i,W_i)
\rightarrow
\mathfrak{g}_{\overline{\mathcal{V}}}(\mathcal{P})\cap
\bigoplus_{\substack{
            i\in\mathbb{N}}}
     C^{\infty}(\overline{V}_i,W'_i),\]
     \[
     (\gamma_i)_{i\in\mathbb{N}}\mapsto(\varphi_i\circ\gamma_i)_{i\in\mathbb{N}}
\]
is bijective. We say that $\mathcal{P}$ has the property $\text{SUB}_\oplus$ if $\mathcal{P}$ has this property with respect to some trivializing system. We also say that $\mathcal{P}$ has the property $\text{SUB}_\oplus$ with respect to some family of charts of $G$ if it has the property $\text{SUB}_\oplus$ with respect to some trivializing system and this family of charts.
\end{mydef}

\begin{bem}
 In the case of a compact base $M$ it is possible to get the following results with Wockel's property SUB, i.e., by just using one chart $\varphi\colon W\rightarrow W'$ of $G$ in the definition above. It turns out that in our situation this is not enough and we need the extended property $\text{SUB}_\oplus$ to prove the most basic results. Here, ``$\oplus$" indicates that we deal with direct sums.
\end{bem}

We will later show that the Lie group structure on $\mathrm{Gau}_c(\mathcal{P})\cong G_{\overline{\mathcal{V}}}(\mathcal{P})$ 
is independent of the trivializing system $\overline{\mathcal{V}}$ as long as it has the property $\text{SUB}_\oplus$ with respect to the trivializing system. Using the following lemma we can now endow $G_{\overline{\mathcal{V}}}(\mathcal{P})$ with its Lie group structure.

\begin{lem}\label{3.5.19}\cite{GloBO}
Let $G$ be a Lie group, modelled on a locally convex space $E$, and $F\subseteq E$ be a sequentially closed vector subspace. Let $H\subseteq G$ be a subgroup. Then $H$ is a Lie subgroup modelled on $F$ if and only if there exists a chart $\phi\colon U\rightarrow V$ of $G$ around $1$ such that $\phi(U\cap H)=V\cap F$.
\end{lem}

\begin{lem}\label{1.8}
Let $\mathcal{P}$ be a smooth $G$-bundle with $\sigma$-compact finite-dimensional base $M$ which has the property $\text{SUB}_\oplus$ with respect to the compact locally finite trivializing system $\overline{\mathcal{V}}$. Let $\varphi=(\varphi_i\colon W_i\rightarrow W'_i)_{i\in\mathbb{N}}$ be the respective family of charts. Then $\varphi_{\ast}$ induces a smooth manifold structure on $G_{\overline{\mathcal{V}}}(\mathcal{P})\cap
\bigoplus_{i\in\mathbb{N}}C^{\infty}(\overline{V}_i,W)$ and $G_{\overline{\mathcal{V}}}(\mathcal{P})$ can be turned into a Lie group modelled on $\mathfrak{g}_{\overline{\mathcal{V}}}(\mathcal{P})$. Furthermore, we have  $\mathrm{L}(G_{\overline{\mathcal{V}}}(\mathcal{P}))\cong\mathfrak{g}_{\overline{\mathcal{V}}}(\mathcal{P})$.
\end{lem}
\begin{proof}
The first statement follows immediately from the property $\text{SUB}_\oplus$ and \ref{3.5.19} since $\mathfrak{g}_{\overline{\mathcal{V}}}(\mathcal{P})$ is a closed vector subspace of $\bigoplus_{i\in\mathbb{N}}C^{\infty}(\overline{V}_i,\mathfrak{g})$. 
Thus, $\mathrm{L}(G_{\overline{\mathcal{V}}}(\mathcal{P}))$ coincides with $\mathfrak{g}_{\overline{\mathcal{V}}}(\mathcal{P})$ as a topological vector space and it remains to show that the Lie bracket coincides as well. Since the point evaluations $\mathrm{ev}_{i,x_i}\colon C^\infty(\overline{V}_i,G)\rightarrow G,\quad \mathrm{ev}_{i,x_i}(\gamma_i)=\gamma_i(x_i)$ and $\mathrm{Ev}_{i,x_i}\colon C^\infty(\overline{V}_i,\mathfrak{g})\rightarrow \mathfrak{g},\quad \mathrm{Ev}_{i,x_i}(\eta_i)=\eta_i(x_i)$ for $x_i\in\overline{V}_i$ are smooth morphisms, the evaluation on the weak direct product 
\[
 \mathrm{ev}_{(x_i)_{i\in\mathbb{N}}}:=(\mathrm{ev}_{i,x_i})_{i\in\mathbb{N}}\colon \sideset{}{^\ast}\prod_{i\in\mathbb{N}}C^{\infty}(\overline{V}_i,G)\rightarrow \sideset{}{^\ast}\prod_{i\in\mathbb{N}}G,\quad (\gamma_i)_{i\in\mathbb{N}}\mapsto\big(\gamma_i(x_i)\big)_{i\in\mathbb{N}}
\] is smooth by \ref{productmorph} and the evaluation on the locally convex direct sum
\[
 \mathrm{Ev}_{(x_i)_{i\in\mathbb{N}}}:=(\mathrm{Ev}_{i,x_i})_{i\in\mathbb{N}}\colon \bigoplus_{i\in\mathbb{N}}C^{\infty}(\overline{V}_i,\mathfrak{g})\rightarrow \bigoplus_{i\in\mathbb{N}}\mathfrak{g},\quad (\eta_i)_{i\in\mathbb{N}}\mapsto\big(\eta_i(x_i)\big)_{i\in\mathbb{N}}
\]
is linear and continuous by \ref{7.1}. Using the chart $d(\varphi_i)_\ast\big|_{T_1C^\infty(\overline{V}_i,G)}$ to identify $T_1C^\infty(\overline{V}_i,G)$ with $C^\infty(\overline{V}_i,\mathfrak{g})$ and $d\varphi_i\big|_{\mathfrak{g}}$ to identify $G$ with $\mathfrak{g}$, as well as the fact that $d\varphi_i\big|_{\mathfrak{g}}=\mathrm{id}_{\mathfrak{g}}$, a straightforward calculation shows $L(\mathrm{ev}_{i,x_i})=T_1(\mathrm{ev}_{i,x_i})=\mathrm{Ev}_{i,x_i}$. As $L((\mathrm{ev}_{i,x_i})$ is a Lie algebra homomorphism, we get $[\eta,\eta'](x)=[\eta(x),\eta'(x)]$ for all $\eta,\eta'\in C^\infty(\overline{V}_i,\mathfrak{g})$ (\cite[compare Section 3.2]{GloQU}). With this in mind and the above, we see that $L(\mathrm{ev}_{(x_i)_{i\in\mathbb{N}}}\big|_{G_{\overline{\mathcal{V}}}(\mathcal{P})})=\mathrm{Ev}_{(x_i)_{i\in\mathbb{N}}}\big|_{\mathfrak{g}_{\overline{\mathcal{V}}}(\mathcal{P})}$ and the assertion follows.
\end{proof}

The next corollary is only an observation. Since it justifies the name ``$\text{SUB}_\oplus$" we will spell it out explicitly. 
\begin{folg}\label{1.9}
If $\mathcal{P}$ is a smooth principal $G$-bundle with finite-dimensional $\sigma$-compact base $M$, having the property $\text{SUB}_\oplus$ with respect to the smooth closed trivializing system $\overline{\mathcal{V}}$, then $G_{\overline{\mathcal{V}}}(\mathcal{P})$ is a closed subgroup of $\prod^\ast_{i\in\mathbb{N}}C^\infty(\overline{V}_i,G)$ and it is a Lie group modelled on $\mathfrak{g}_{\overline{\mathcal{V}}}(\mathcal{P})$.
\end{folg}

That different choices of charts lead to isomorphic Lie group structures follows directly from \ref{3.5.19}. However, compared to the case of compact $M$ it is surprisingly cumbersome to show that different choices of trivializing systems also lead to isomorphic Lie group structures on $\mathrm{Gau}_c(\mathcal{P})$ (compare \cite[Proposition 1.10]{Wockel3}). The next two lemmas are needed to prove this in \ref{1.10}.

\begin{lem}\label{partition}
Let $(F_i)_{i\in I}$ be a family of locally convex spaces and $(I_j)_{j\in J}$ be a partition of $I$ into nonempty finite sets, i.e., $\bigcup_{j\in J}I_j=I$, such that each $I_j$ is finite and nonempty and $I_j\cap I_{j'}=\varnothing$ for $j\neq j'$. Then $\bigoplus_{i\in I}F_i\cong\bigoplus_{j\in J}\bigoplus_{i\in I_j}F_i$.
\end{lem}
\begin{proof}
For every $i\in I$ there exists exactly one $j(i)\in J$ with $i\in I_{j(i)}$. The map 
\[
\bigoplus_{i\in I}F_i\rightarrow\bigoplus_{j\in J}\bigoplus_{i\in I_j}F_i,\quad(\eta_i)_{i\in I}\mapsto\Big(\bigoplus_{i\in I_j}\eta_i\Big)_{j\in J}
\]
is linear and continuous since its restriction to $F_i$ is the inclusion $F_i\hookrightarrow\bigoplus_{k\in I_{j(i)}}F_k$ into a finite direct sum and hence linear and continuous. 
On the other hand, the restrictions to the summands of the inverse map are given by 
\[
\bigoplus_{i\in I_j}F_i\hookrightarrow\prod_{i\in I_j}F_i\subseteq\bigoplus_{i\in I}F_i
\]
because the sum on the right-hand side induces the product topology on finite products. Thus the inverse is also continuous and linear.
\end{proof}

\begin{lem}\label{lem1.10}
Let $(E_n)_{n\in\mathbb{N}}$ and $(F_n)_{n\in\mathbb{N}}$ be families of locally convex spaces, $U_n\subseteq E_n$ open zero neighbourhoods and $r\in\mathbb{N}\cup\{\infty\}$. If there exist finite subsets $I_n\subseteq\mathbb{N}$ such that every $n\in\mathbb{N}$ is only contained in finitely many $I_k$, and there exist functions $f_n\colon\bigoplus_{i\in I_n}U_i\rightarrow F_n$ of class $C^r$ with $f_n(0)=0$, then
\[
f\colon\bigoplus_{n\in\mathbb{N}}U_n\rightarrow \bigoplus_{n\in\mathbb{N}}F_n,\quad(x_n)_{n\in\mathbb{N}}\mapsto\bigg(f_n(\big(x_i)_{i\in I_n}\big)\bigg)_{n\in\mathbb{N}}
\]
is a map of class $C^r$.
\end{lem}
\begin{proof}
Let $\lambda_n\colon E_n\rightarrow\bigoplus_{k\in\mathbb{N}}E_k$ be the inclusion. We may assume $r<\infty$ and show the result by induction. Let $r=0$,  $x=(x_n)_{n\in\mathbb{N}}\in\bigoplus_{n\in\mathbb{N}}U_n$ and $V\subseteq\bigoplus_{n\in\mathbb{N}}F_n$ be an open neighbourhood of $f(x)$. Since $\bigoplus_{n\in\mathbb{N}}F_n$ carries the box-topology we can assume w.l.o.g. that $V=\bigoplus_{n\in\mathbb{N}}V_n$ with open neighbourhoods $V_n\subseteq F_n$ of $f_n((x_i)_{i\in I_n})$. By assumption, $f_k$ is continuous, whence there exist open neighbourhoods $W_{n,k}\subseteq U_n$ of $x_n$ for $n\in I_k$ such that $f_k(\bigoplus_{n\in I_k}W_{n,k})\subseteq V_k$. Since each $n\in\mathbb{N}$ is only contained in finitely many $I_k$ the sets $J_n:=\{k\in\mathbb{N}\colon n\in I_k\}$ are finite and thus $W'_n:=\bigcap_{k\in J_n}W_{n,k}$ is an open neighbourhood of $x_n$ in $U_n$ if $J_n\neq\varnothing$; if $J_n=\varnothing$, we set $W'_n:=U_n$. This implies $\bigoplus_{n\in\mathbb{N}}f_n^{-1}(V_n)\supseteq\bigoplus_{n\in\mathbb{N}}W'_n$ which is open, hence $f$ is continuous in $x$, hence $f$ is continuous.\\
Now let $r\geq 1$. Let $x\in\bigoplus_{n\in\mathbb{N}}U_n$, $y\in\bigoplus_{n\in\mathbb{N}}E_n$ and $N\in\mathbb{N}$ such that $x_l=y_n=0$ for all $l\in I_n$ if $n>N$. This implies
\[
\frac{f(x+ty)-f(x)}{t}\xrightarrow{t\rightarrow 0} \bigg(\mathrm{d}f_n\big((x_k,y_k)_{k\in I_n}\big)\bigg)_{n=1,\ldots,N}\cong\bigg( d f_n\big((x_k,y_k)_{k\in I_n}\big)\bigg)_{n\in\mathbb{N}}.
\]
Hence $ d f(x,y)$ exists and $ d f(x,y)=( d f_n(x_k,y_k)_{k\in I_n})_{n\in\mathbb{N}}$. Using the same sets $I_k$ and the natural isomorphism $\bigoplus_{n\in\mathbb{N}}E_n\times\bigoplus_{n\in\mathbb{N}}E_n\cong\bigoplus_{n\in\mathbb{N}}(E_n\times E_n)$, we see that
\[
 d f\colon\bigoplus_{n\in\mathbb{N}}(U_n\times E_n)\rightarrow\bigoplus_{n\in\mathbb{N}}F_n,\quad(x_n,y_n)_{n\in\mathbb{N}}\mapsto\big( d f_n(x_k,y_k)_{k\in I_n}\big)_{n\in\mathbb{N}}
\]
is of class $C^{r-1}$ and hence in particular continuous. It follows by induction that $f$ is of class $C^r$.
\end{proof}
The above lemma obviously entails lemma \ref{7.1}.

\begin{prop}\label{1.12}
Let $\mathcal{P}$ be a smooth $G$-bundle with $\sigma$-compact finite-dimensional base $M$ and  $\overline{\mathcal{U}}=(\overline{U}_i,\sigma_i)_{i\in\mathbb{N}}$ be a compact locally finite trivializing system of $\mathcal{P}$ that has the property $\text{SUB}_\oplus$. 
Then any refinement $\overline{\mathcal{V}}=(\overline{V}_i,\tau_i)_{i\in\mathbb{N}}$ of $\overline{\mathcal{U}}$ has the property $\text{SUB}_\oplus$. In this case we have $G_{\overline{\mathcal{U}}}(\mathcal{P})\cong G_{\overline{\mathcal{V}}}(\mathcal{P})$ as Lie groups.
\end{prop}
\begin{proof}
Note that $\overline{\mathcal{V}}$ being a refinement of $\overline{\mathcal{U}}$ implies that there exists a map $j\colon \mathbb{N}\rightarrow\mathbb{N}$ such that $\overline{V}_i\subseteq\overline{U}_{j(i)}$ and $\tau_{i}=\sigma_{j(i)}\big|_{\overline{V}_i}$ for all $i\in\mathbb{N}$. Note that this implies $k_{ii'}(x)=k_{j(i)j(i')}(x)$. For $j\in\mathbb{N}$ we set $I_j:=\{i\in\mathbb{N}\colon j(i)=j\}$.
First, we show that the mappings
\[
\Phi_G\colon G_{\overline{\mathcal{U}}}(\mathcal{P})\rightarrow G_{\overline{\mathcal{V}}}(\mathcal{P}),\quad(\gamma_i)\mapsto(\gamma_{j(i)}\big|_{\overline{U}_{j(i)}})\text{ and}\]
\[
\Phi_{\mathfrak{g}}\colon \mathfrak{g}_{\overline{\mathcal{U}}}(\mathcal{P})\rightarrow \mathfrak{g}_{\overline{\mathcal{V}}}(\mathcal{P}),\quad(\eta_i)\mapsto(\eta_{j(i)}\big|_{\overline{V}_{j(i)}})
\]
are bijections and then we use this fact to transfer the property $\text{SUB}_\oplus$ from one trivializing system to the other. For this we need construct the respective inverse functions. 
Starting with $\Phi_G^{-1}$, we construct each component $({\Phi_G})^{-1}_j\colon G_{\overline{\mathcal{V}}}(\mathcal{P})\rightarrow C^{\infty}(\overline{U}_j,G)$ separately. The conditions that $(({\Phi_G})^{-1}_j)_{j\in\mathbb{N}}$ is inverse to $\Phi_G$ are then given by
\begin{align}
({\Phi_G})^{-1}_j\big((\gamma_i)_{i\in\mathbb{N}}\big)\big|_{\overline{V}_i}&=\gamma_i\quad\text{ for all }i\text{ with }j=j(i)\quad\text{and}\label{condition1}\\
({\Phi_G})^{-1}_j\big((\gamma_i)_{i\in\mathbb{N}})(x)&=k_{jl}(x)({\Phi_G})^{-1}_l\big((\gamma_i)_{i\in\mathbb{N}})(x)k_{lj}(x)\quad\text{for}\quad x\in\overline{U}_j\cap\overline{U}_k\cap\overline{U}_l.\label{condition2}
\end{align}
Let $x\in\overline{U}_j\setminus\bigcup_{i\in I_j}V_i$. Then $x\in V_{i_x}$ for some $i_x\in\mathbb{N}$ and thus there exists a relativly open neighbourhood $U_x\subseteq \overline{U}_j$ of $x$ such that $\overline{U}_x$ is a manifold with corners, contained in $\overline{U}_j\cap V_{i_x}$. Now, finitely many $U_{x_1},\cdots,U_{x_l}$ cover $\overline{U}_j\setminus\bigcup_{i\in I_j}V_i$ and we set
\[
({\Phi_G}^{-1})_j\big((\gamma_i)_{i\in\mathbb{N}}\big):=\mathrm{glue}\left((\gamma_i)_{i\in I_j},\left((k_{ji_{x_k}}\cdot\gamma_{i_{x_k}}\cdot k_{i_{x_k}j})\big|_{U_{x_k}}\right)_{k=1,\ldots,l}\right).
\]
Note that because $j(i)=j$ implies $i\in I_j$ this map satisfies (\ref{condition1}). Condition (\ref{condition2}) is satisfied because the transition functions arise as restrictions and we have $k_{ij}(x)=k_{il}(x)k_{lj}(x)$ for $x\in\overline{U}_i\cap\overline{U}_l\cap\overline{U}_j$.\\
The construction of $\Phi_{\mathfrak{g}}^{-1}$ can be done analogously by defining
\[
({\Phi_{\mathfrak{g}}})^{-1}_j\big((\eta_i)_{i\in\mathbb{N}}\big):=\mathrm{glue}\left((\eta_i)_{i\in I_j},\left(\mathrm{Ad}(k_{ji_x}).\eta_{i_{x_k}}\big|_{U_{x_k}}\right)_{k=1,\ldots,l}\right).
\]
With this definition it is apparent that $\Phi_G$ (resp. $\Phi_{\mathfrak{g}}$) is not only bijective but an isomorphism of groups (resp. vector spaces). Now, we assume $\overline{\mathcal{U}}$ has the property $\text{SUB}_\oplus$ with respect to the family of centered charts $\varphi:=(\varphi_i\colon W_i\rightarrow W'_i)_{i\in\mathbb{N}}$. For every $i\in\mathbb{N}$ we choose a unity neighbourhood $\widetilde{W}_i$ such that for all $j\in\mathbb{N}$ with $\overline{V}_i\cap U_j\neq\varnothing$ the conditions
\begin{align}
 k_{ji}(x)\cdot\widetilde{W}_i\cdot k_{ij}(x)\subseteq W_i\quad&\text{and}\label{condition3}\\
 \mathrm{Ad}\big(k_{ji}(x)\big).\varphi_i(\widetilde{W}_i)\subseteq W'_i\quad&\text{for all}\quad x\in\overline{V}_i\cap \overline{U}_j\label{condition4}
\end{align}
hold. This is possible since only finitely many sets $U_j$ intersect $\overline{V}_i$ and $\overline{V}_i\cap \overline{U}_j$ is compact. Consider $(\gamma_i)_{i\in\mathbb{N}}\in G_{\overline{\mathcal{V}}}(\mathcal{P})\cap\sideset{}{^\ast}\prod_{i\in\mathbb{N}}C^{\infty}(\overline{V}_i,\widetilde{W}_i)$. By (\ref{condition3}) we have that $(\Phi_G)^{-1}_j((\gamma_{i\in\mathbb{N}}))\in C^{\infty}(\overline{U}_j,W_j)$ for all $j\in\mathbb{N}$. 
Hence, $\varphi_\ast(\Phi_G^{-1}((\gamma_i)_{i\in\mathbb{N}}))$ is defined and we get 
\[
\Phi_{\mathfrak{g}}\left(\varphi_\ast\left(\Phi_G^{-1}((\gamma_i)_{i\in\mathbb{N}})\right)\right)=\varphi_\ast\Big((\gamma_i)_{i\in\mathbb{N}}\Big).
\]
This shows that $\Phi_{\mathfrak{g}}^{-1}$ is the local representation of $\Phi_G^{-1}$.
We set $\widetilde{W}'_i:=\varphi_i(\widetilde{W}_i)$. Starting with $(\eta_i)_{i\in\mathbb{N}}\in\mathfrak{g}_{\overline{\mathcal{V}}}(\mathcal{P})\cap\bigoplus_{i\in\mathbb{N}}C^{\infty}(\overline{U}_i,\widetilde{W}'_i)$ and using condition (\ref{condition4}) for an analogous argument we see that
\[
 \varphi_\ast\colon G_{\overline{\mathcal{V}}}(\mathcal{P})\cap\sideset{}{^\ast}\prod_{i\in\mathbb{N}}C^{\infty}(\overline{V}_i,\widetilde{W}_i)\rightarrow\mathfrak{g}_{\overline{\mathcal{V}}}(\mathcal{P})\cap\bigoplus_{i\in\mathbb{N}}C^{\infty}(\overline{U}_i,\widetilde{W}'_i)
\]
 is indeed bijective as needed. Thus, $\overline{\mathcal{V}}$ has the property $\text{SUB}_\oplus$ with respect to the family of charts $(\varphi_{j(i)}\big|^{\widetilde{W}'_{j(i)}}_{\widetilde{W}_{j(i)}})_{i\in\mathbb{N}}$.\\
Next, we show that $\Phi_G$ is in fact an isomorphism of Lie groups. Using the charts from above, the components of the local description of $\Phi_G$ extend to the maps
\[
C^{\infty}(\overline{U}_j,W'_j)\rightarrow \bigoplus_{i\in I_j}C^{\infty}(\overline{V}_i,W'_j),\quad
\eta_j\mapsto\bigoplus_{i\in I_j}\varphi_j\circ(\varphi_j^{-1}\circ\eta_j\big|_{\overline{V}_i}),
\]
which are smooth maps. Hence, the resulting map 
\[
\bigoplus_{i\in \mathbb{N}}C^{\infty}(\overline{U}_i,W'_j)\rightarrow \bigoplus_{j\in\mathbb{N}}\bigoplus_{i\in I_j}C^{\infty}(\overline{V}_i,W'_j)\cong\bigoplus_{i\in \mathbb{N}}C^{\infty}(\overline{V}_i,W'_j)
\]is smooth and with \ref{partition} we see that the local representation of $\Phi_G$ is smooth and thus $\Phi_G$ is a smooth morphism.\\
Recall the construction of the components $(\Phi_G^{-1})_j$ from above. It is given by
\[
({\Phi_G}^{-1})_j\big((\gamma_i)_{i\in\mathbb{N}}\big):=\mathrm{glue}\left((\gamma_i)_{i\in I_j},\left((k_{ji_{x_k}}\cdot\gamma_{i_{x_k}}\cdot k_{i_{x_k}j})\big|_{U_{x_k}}\right)_{k=1,\ldots,l}\right).
\] 
for certain $i_x\in\mathbb{N}$ and $U_{x_1},\ldots,U_{x_l}$ such that $\overline{U}_{x_n}$ is a manifold with corners contained in $\overline{U}_j\cap V_{i_{x_n}}$.
Unfortunately we cannot extend this map to a map on a neighbourhood of the whole weak direct product but since $\Phi^{-1}_G$ is a morphism it is enough to check smoothness on a neighbourhood and via charts we then only need to check the smoothness of a map between locally convex direct sums. There it will be possible to express gluing with a partition of unity (compare \ref{topemb}) and define a corresponding map on the whole sum that will be smooth by \ref{lem1.10}.\\
Let $I'_j:=I_j\cup\dot\bigcup_{k=1,\ldots,l}\{i_{x_k}\}$ and note that since every $V_i$ intersects only finitely many sets $\overline{U}_j$, the sets $I'_j$  fulfil the requirements of \ref{lem1.10}. Further, let $(f_n)_{n\in I'_j}$ be a smooth partition of unity subordinated to $(V_i)_{i\in I_j}$ combined with $(U_{x_k})_{k=1,\ldots,l}$ on the manifold $\overline{U}_j$.
Now, consider the map
\[
\Psi_j\colon\bigoplus_{i\in I'_j}C^\infty(\overline{V}_i,\widetilde{W}'_j)\rightarrow C^\infty(\overline{U}_j,\mathfrak{g}),
\]
\[
(\eta_i)_{i\in I'_j}\mapsto
\sum_{i\in I_j}f_i\cdot\big(\varphi_j\circ\varphi_j^{-1}(\eta_i)\big)^{\widetilde{}}+
\sum_{k=1}^lf_{i_{x_k}}\cdot\left(\varphi_j\circ\big(k_{ji_{x_k}}\cdot\varphi_j^{-1}\circ\eta_{i_{x_k}}\cdot k_{i_{x_k}j}\big)\big|_{U_{x_k}}\right)^{\widetilde{}},
\]
where the tilde indicates that maps are continued to $\overline{U}_j$ as in \ref{cont}. This map is smooth as a composition of smooth maps and it defines the components of a map $\Psi\colon\bigoplus_{i\in\mathbb{N}}C^\infty(\overline{V}_{i},\widetilde{W}'_{j(i)})\rightarrow\bigoplus_{j\in\mathbb{N}}C^\infty(\overline{U}_j,\mathfrak{g})$ on the whole sum mapping zero to zero, which is smooth by \ref{lem1.10}. Furthermore, the map $\Psi$ restricts to $\Phi_{\mathfrak{g}}^{-1}$ on the intersection with $\mathfrak{g}_{\overline{\mathcal{V}}}(\mathcal{P})$. Since the $\Phi_{\mathfrak{g}}^{-1}$ is the local representation of $\Phi_G^{-1}$ we have thus shown that $\Phi_G$ is an isomorphism of Lie groups.
\end{proof}

\begin{folg}\label{1.10}
Let $\mathcal{P}$ be a smooth $G$-bundle with $\sigma$-compact finite-dimensional base $M$. If we have two compact locally finite trivializing systems  $\overline{\mathcal{V}}=(\overline{V}_i,\sigma_i)_{i\in\mathbb{N}}$ and
$\overline{\mathcal{U}}=(\overline{U}_i,\tau_i)_{i\in\mathbb{N}}$ and $\mathcal{P}$ 
has the property $\text{SUB}_\oplus$ with respect to $\overline{\mathcal{U}}$ and $\overline{\mathcal{V}}$, then 
$G_{\overline{\mathcal{V}}}(\mathcal{P})$ is isomorphic to $G_{\overline{\mathcal{U}}}(\mathcal{P})$ as a Lie group.
\end{folg}
\begin{proof}
Assume $\mathcal{P}$ has the property $\text{SUB}_\oplus$ with respect to $\overline{\mathcal{V}}$ and $\overline{\mathcal{U}}$. First, we note that if the covers underlying $\overline{\mathcal{V}}$ and $\overline{\mathcal{U}}$ are the same, but the sections differ by smooth functions $k_i\in C^{\infty}(\overline{V}_i,G)$, i.e., $\sigma_i=\tau_i\cdot k_i$, then this induces an isomorphism of Lie groups
\[
G_{\overline{\mathcal{V}}}(\mathcal{P})\rightarrow G_{\overline{\mathcal{U}}}(\mathcal{P}),\quad(\gamma_i)_{i\in\mathbb{N}}\mapsto(k_i\cdot\gamma_i\cdot k_i^{-1})_{i\in\mathbb{N}}.
\]
In fact, in this case $G_{\overline{\mathcal{V}}}(\mathcal{P})$ and $G_{\overline{\mathcal{U}}}(\mathcal{P})$ are closed Lie subgroups of the weak direct product $\sideset{}{^\ast}\prod_{i\in\mathbb{N}}C^\infty(\overline{V}_i,G)$ and the map above arises as the restriction of
\[
 \sideset{}{^\ast}\prod_{i\in\mathbb{N}}C^\infty(\overline{V}_i,G)\rightarrow\sideset{}{^\ast}\prod_{i\in\mathbb{N}}C^\infty(\overline{U}_i,G),\quad(\gamma_i)_{i\in\mathbb{N}}\mapsto(k_i\cdot\gamma_i\cdot k_i^{-1})_{i\in\mathbb{N}},
\]
which is smooth by \ref{productmorph} because the conjugations on the components are Lie group morphisms.\\
Applying the above and using \ref{1.12}, we may assume that $\overline{\mathcal{V}}$ is a refinement of $\overline{\mathcal{U}}$ since the two covers have a common locally finite refinement. The result follows now from \ref{1.12}.
\end{proof}

We come to the main result of this section. 

\begin{satz}\label{1.11}\cite[cf. Theorem 1.11]{Wockel3}
Let $\mathcal{P}$ be a smooth $G$-bundle over the $\sigma$-compact finite-dimen\-sional manifold $M$ that has the property $\text{SUB}_\oplus$. 
If $\mathcal{P}$ has the property $\text{SUB}_\oplus$ with respect to some trivializing system $\overline{\mathcal{U}}$, then the gauge group $\mathrm{Gau}_c(\mathcal{P})\cong C^{\infty}_c(P,G)^G$ carries a Lie group structure modelled on $\mathfrak{g}_{\overline{\mathcal{U}}}(\mathcal{P})$. This Lie group structure is independent of the $\overline{\mathcal{U}}$. If, moreover, $G$ is locally exponential, then $\mathrm{Gau}_c(\mathcal{P})$ is so.
\end{satz}
\begin{proof}
We endow $\mathrm{Gau}_c(\mathcal{P})$ with the Lie group structure induced by the isomorphisms of groups $\mathrm{Gau}_c(\mathcal{P})\cong C^{\infty}_c(P,G)^G\cong G_{\overline{\mathcal{U}}}(\mathcal{P})$ for some trivializing system $\overline{\mathcal{U}}=(\overline{U}_i,\sigma_i)_{i\in\mathbb{N}}$ with respect to which $\mathcal{P}$ has the property $\text{SUB}_\oplus$. Now, let $\overline{\mathcal{V}}=(\overline{V}_i,\tau_i)_{i\in\mathbb{N}}$ be another trivializing system for which $\mathcal{P}$ has the property $\text{SUB}_\oplus$. To see that the Lie group structure of $\mathrm{Gau}_c(\mathcal{P})$ is independent of the choice of the trivializing system 
we only need to show that  
\[
    \xymatrix{
		&  G_{\overline{\mathcal{U}}}(\mathcal{P})  \ar[dd]_{}  \\
        C^{\infty}_c(P,G)^G \ar[ru]^{} \ar[rd]_{} & \\
             & G_{\overline{\mathcal{V}}}(\mathcal{P})
    }
\]
is a commutative diagram, using the isomorphisms of Lie groups from \ref{1.6} and \ref{1.10}. 
In the light of the proofs of \ref{1.12} and \ref{1.10} it suffices to check this if $\overline{\mathcal{V}}$ is a refinement of $\overline{\mathcal{U}}$ or both trivializing systems differ only on their sections. Let $\overline{\mathcal{V}}$ be a refinement of $\overline{\mathcal{U}}$ such that $\overline{V}_i\subseteq\overline{U}_{j(i)}$ and $\tau_i=\sigma_{j(i)}\big|_{\overline{V}_i}$ holds for some $j\colon\mathbb{N}\rightarrow\mathbb{N}$. For $\gamma\in C^{\infty}_c(P,G)^G$, we then have $\gamma\circ\tau_i=\gamma\circ\sigma_{j(i)}\big|_{\overline{V}_i}$ which shows that the diagram commutes. If, on the other hand, $\overline{U}_i=\overline{V}_i$ and $\sigma_i=\tau_i\cdot k_i$ for $k_i\in C^\infty(\overline{U}_i,G)$ holds, we conclude from the equality   $\gamma\circ\tau_i=\gamma\circ(\sigma_i\cdot k_i^{-1})=k_i\cdot(\gamma\circ\sigma_i)\cdot k_i^{-1}$ that the diagram commutes.\\
Next, we construct an exponential function for $G_{\overline{\mathcal{U}}}(\mathcal{P})$ under the assumption that $G$ is locally exponential.
We first show that if $G$ is locally exponential and $M$ is a compact manifold with corners, then
\[
(\mathrm{exp}_G)_\ast\colon C^{\infty}(M,\mathfrak{g})\rightarrow C^{\infty}(M,G),\quad\eta\mapsto\mathrm{exp}_G\circ\eta
\]
is the exponential function for $C^{\infty}(M,G)$. Afterwards we use this map to define local diffeomorphisms on the components of $G_{\overline{\mathcal{V}}}(\mathcal{P})$, which will yield the exponential function for $\mathfrak{g}_{\overline{\mathcal{V}}}(\mathcal{P})$. In fact, for $x\in\mathfrak{g}$ let $\gamma_x\in C^{\infty}(\mathbb{R},G)$ be the solution of the initial value problem $\gamma(0)=e$, $\gamma'(t)=\gamma(t).x$. Let $\eta\in C^{\infty}(M,\mathfrak{g})$, then
\[
\Gamma_\eta\colon\mathbb{R}\rightarrow C^{\infty}(M,G),\quad(t,m)\mapsto\gamma_{\eta(m)}(t)=\mathrm{exp}_G\left(t\cdot\eta(m)\right)
\]
is a group homomorphism. Furthermore, $\Gamma_\eta$ is smooth because it is smooth on a zero neighbourhood of $\mathbb{R}$ since the push-forward of the local inverse of $\mathrm{exp}_G$ provides a chart on a unit neighbourhood in $C^{\infty}(M,G)$. For the left logarithmic derivative we then have
\[
\delta^l(\Gamma_\eta)(t)=\Gamma_\eta(t)^{-1}\cdot\Gamma_\eta(t)\cdot\eta=\eta,
\]
understood as an equation in $T(C^{\infty}(M,G))$. This shows that $\eta\mapsto\Gamma_\eta(1)=\mathrm{exp}_G\circ\gamma$ is the exponential function of $C^{\infty}(M,G)$ (compare \cite{GloQU} for an alternative argument). Now, let $W'\subseteq\mathfrak{g}$ such that $\mathrm{exp}_G\big|_{W'}$ is a diffeomorphism onto an open set $W\subseteq G$. 
Because any given $\overline{U}_i$ only intersects with finitely other sets of the trivializing system and the intersections are compact, we can choose an open zero neighbourhood $W'_i\subseteq W'$ such that $\mathrm{Ad}(k_{ij}(x)).W'_i\subseteq W'$ for all $x\in\overline{U}_i\cap\overline{U}_j$. We define $W_i:=\mathrm{exp}_G(W'_i)$ and
for $\mathrm{exp}_{G,i}:=\mathrm{exp}_G\big|_{W'_i}^{W_i}$ the equation
\[ \mathrm{exp}_{G,i}\circ\eta_i(x)=\mathrm{exp}_{G,i}\circ\mathrm{Ad}(k_{ij}(x)).\eta_j(x)=k_{ij}(x)(\mathrm{exp}_{G,i}\circ\eta_j(x))k_{ji}(x)
\]
is well defined and holds. Thus
\[
\mathrm{exp}_\ast:=((\mathrm{exp}_{G,i})_\ast)_{i\in\mathbb{N}}\colon\mathfrak{g}_{\overline{\mathcal{U}}}(\mathcal{P})\cap\bigoplus_{i\in\mathbb{N}}C^{\infty}(\overline{U}_i,W'_i)\rightarrow G_{\overline{\mathcal{U}}}(\mathcal{P})\cap\sideset{}{^\ast}\prod_{i\in\mathbb{N}}C^{\infty}(\overline{U}_i,W_i),\]
\[ 
(\eta_i)_{i\in\mathbb{N}}\mapsto(\mathrm{exp}_{G,i}\circ\eta_i)_{i\in\mathbb{N}}
\]
is a diffeomorphism. Finally, since the componentwise exponential function of the weak direct product restricts to intersections of equalisers the map above is obviously the exponential function on $G_{\overline{\mathcal{V}}}(\mathcal{P})$.
\end{proof}

Next we show that $\mathcal{P}$ automatically has property $\text{SUB}_\oplus$ for large classes of well-known Lie groups.

\begin{lem}\label{1.14}
Let $\mathcal{P}$ be a smooth principal $G$-bundle over the $\sigma$-compact finite\-dimensional manifold $M$. Then $\mathcal{P}$ has the property $\text{SUB}_\oplus$ with respect to each compact locally finite trivializing system if
\begin{enumerate}[(a)]
\item $\mathcal{P}$ is trivial, i.e., there exists a global smooth trivializing system,
\item if $G$ is abelian,
\item if $G$ is a Banach-Lie group or 
\item if $G$ is locally exponential.
\end{enumerate}
\end{lem}
\begin{proof}
(a) If $\mathcal{P}$ is trivial we may choose a compact locally finite trivializing system $\overline{\mathcal{V}}=(\overline{V}_n,\sigma_n)_{n\in\mathbb{N}}$ such that $\sigma_i(x)=\sigma_j(x)$ for all $x\in\overline{V}_i\cap\overline{V}_j$. Then all transition functions are trivial and so is the adjoint action used in $\mathfrak{g}_{\overline{\mathcal{V}}}(\mathcal{P})$, thus we can proceed as in (b).\\
(b) If $G$ is abelian then the conjugation of $G$ on itself and the adjoint action of $G$ on $\mathfrak{g}$ are trivial. Thus for any compact locally finite trivializing system $\overline{\mathcal{V}}=(\overline{V}_n,\sigma_n)_{n\in\mathbb{N}}$, we have $\mathfrak{g}_{\overline{\mathcal{V}}}(\mathcal{P})=\mathfrak{g}_{\overline{\mathcal{V}}}$ (see \ref{algebraglue}) and $G_{\overline{\mathcal{V}}}(\mathcal{P})$ is the set
\[
\left\{(\gamma_i)_{i\in\mathbb{N}}\in\sideset{}{^\ast}\prod_{i\in\mathbb{N}}C^\infty(\overline{V}_i,G)\colon\: \gamma_i(x)=\gamma_j(x)\text{ for all }x\in\overline{V}_i\cap\overline{V}_j\right\}.
\]
Now, we choose a centered chart $\varphi\colon W\rightarrow W'$ such that $d\varphi\big|_{\mathfrak{g}}=\mathrm{id}_{\mathfrak{g}}$ and infer from \ref{directsum} that $\mathcal{P}$ has the property $\text{SUB}_\oplus$.\\
(c) If $G$ is a Banach-Lie group then it is in particular locally exponential and thus it will be enough to show (d).\\ 
(d) Let $G$ be locally exponential and let $\overline{\mathcal{V}}=(\overline{V}_n,\sigma_n)_{n\in\mathbb{N}}$ be a compact locally finite trivializing system. Furthermore, let $W'\subseteq\mathfrak{g}$ be an open zero neighbourhood such that $\mathrm{exp}_G|_{W'}$ is a diffeomorphism onto $W:=\mathrm{exp}_G(W')$. 
Following the proof of \ref{1.11} we find open zero neighbourhoods $W'_i\subseteq W'_i$ such that
\[
\mathrm{exp}_\ast:=((\mathrm{exp}_{G,i})_\ast)_{i\in\mathbb{N}}\colon\mathfrak{g}_{\overline{\mathcal{U}}}(\mathcal{P})\cap\bigoplus_{i\in\mathbb{N}}C^{\infty}(\overline{U}_i,W'_i)\rightarrow G_{\overline{\mathcal{U}}}(\mathcal{P})\cap\sideset{}{^\ast}\prod_{i\in\mathbb{N}}C^{\infty}(\overline{U}_i,W_i),\]
\[ 
(\eta_i)_{i\in\mathbb{N}}\mapsto(\mathrm{exp}_{G,i}\circ\eta_i)_{i\in\mathbb{N}}
\]
is bijective for $W_i:=\mathrm{exp}_G(W'_i)$ and $\mathrm{exp}_{G,i}:=\mathrm{exp}_G\big|_{W'_i}^{W_i}$. Since the inverse of this map arrises as push-forward of restrictions of $\mathrm{exp}^{-1}_G$ and $d(\mathrm{exp}_{G})\big|_{\mathfrak{g}}=\mathrm{id}_{\mathfrak{g}}$, the property $\text{SUB}_\oplus$ is proven.
\end{proof}

While we have shown that $\mathcal{P}$ has the property $\text{SUB}_\oplus$ in many interesting cases, as of now, the author does not know of any examples where $\mathcal{P}$ does not have the property $\text{SUB}_\oplus$. This remains an area for more thorough examination.
\begin{prob}\label{1.16}
Is there a smooth principal $G$-bundle $\mathcal{P}$ over a $\sigma$-compact base $M$ which does not have the property $\text{SUB}_\oplus$?
\end{prob}

\section{The Automorphism Group as an Infinite-Dimen\-sional Lie Group}

In this section we will turn $\mathrm{Aut}_c(\mathcal{P})$ into a Lie group, where $\mathcal{P}$ is a smooth principal $G$-bundle over a finite-dimensional $\sigma$-compact base $M$ \textit{without} boundary. We will do so by constructing an extension of Lie groups
\begin{align*}
\mathrm{Gau}_c(\mathcal{P})\hookrightarrow\mathrm{Aut}_c(\mathcal{P})\xtwoheadrightarrow{\mathcal{Q}}\mathrm{Diff}_c(M)_{\mathcal{P}},
\end{align*}
where $\mathrm{Diff}_c(M)_{\mathcal{P}}$ is the image of the homomorphism $\mathcal{Q}\colon\mathrm{Aut}_c(\mathcal{P})\rightarrow\mathrm{Diff}_c(M)$ from \ref{1.1b}. To this end, we need to find a so called smooth factor system, which will induce a Lie group structure on the automorphisms from the Lie group structures on the gauge group and the group of diffeomorphisms on $M$. The main work in this endeavor comes from finding a suitable section $S$ with respect to $\mathcal{Q}$ and checking the various smoothness conditions. 
\bigskip

Throughout this section we will only use bundles over a base without corners, i.e., $P$ and $M$ will be manifolds without corners. We fix one particular $G$-bundle $\mathcal{P}$ over a $\sigma$-compact finite-dimensional base $M$ without corners and assume that $\mathcal{P}$ has the property $\text{SUB}_\oplus$.

\subsection{Extension of Lie groups}
We begin by defining what extensions of Lie groups are and how to get such an extension via a smooth factor system. Since we use these objects only as a tool, this section is rather brief. For a thorough explanation see \cite{Neeb}.
\begin{mydef}\label{II.1}
Let $G$ be a Lie group. A subgroup $H$ of $G$ is called a \textit{split Lie subgroup} if it carries a Lie group structure for which the inclusion $i_H\colon H\hookrightarrow G$ is a morphism of Lie groups and the right action of $H$ on $G$ defined by restricting the multiplication map of $G$ to a map $G\times H\rightarrow G$ defines a smooth principal $H$-bundle. This means that the coset space $G/H$ is a smooth manifold and that the quotient map $\pi\colon G\rightarrow G/H$ has smooth local sections.
\end{mydef}
\begin{mydef}\label{II.4}
An extension of Lie groups is a short exact sequence
\[
1\rightarrow N\xrightarrow{\iota}\widehat{G}\xrightarrow{q} G\rightarrow 1
\]
of Lie group morphisms, for which $N\cong\mathrm{ker}(q)$ is a split Lie subgroup and $G\cong\widehat{G}/\mathrm{ker}q$. This means $\widehat{G}$ is a smooth principal $N$-bundle over $G$ and $G\cong \widehat{G}/N$. Here and in the following, we identify $N$ with the normal subgroup $\iota(N)\trianglelefteq\widehat{G}$.\\
In the above situation we also call the map $q\colon\widehat{G}\rightarrow G$ an extension of $G$ by $N$. 
We call two extensions $N\hookrightarrow\widehat{G}_1\twoheadrightarrow G$ and $N\hookrightarrow\widehat{G}_2\twoheadrightarrow G$ \textit{equivalent} if there exists a morphism of Lie groups $\psi\colon\widehat{G}_1\rightarrow\widehat{G}_2$ such that the diagram
\[
    \xymatrix{
        N \ar[r] \ar[d]_{\mathrm{id}_N} & \widehat{G}_1 \ar[d]_{\psi} \ar[r] & G \ar[d]_{\mathrm{id}_G}\\
        N \ar[r]       & \widehat{G}_2 \ar[r] & G}
\]
commutes.
\end{mydef}

The following remark gives a motivation for the particular properties the smooth factor systems (defined afterwards) need to have for turning an exact sequence of abstract groups into an extension of Lie groups. It will also serve as a guideline for the construction of a smooth factor system in the later parts.

\begin{bem}\cite[cf. Remark II.5]{Neeb}\label{II.5}
Let $q\colon\widehat{G}\rightarrow G$ be a Lie group extension of $G$ by $N$. By assumption, there exists a smooth section $\hat{\sigma}\colon U\rightarrow\widehat{G}$ on an identity neighbourhood $U\subseteq G$. After multiplication with an element of $N$ we may assume $\hat{\sigma}(1_G)=1_{\widehat{G}}$ and extend $\hat{\sigma}$ to a global section $\sigma\colon G\rightarrow\widehat{G}$ (not necessarily continuous). Then the map
\[
\Phi\colon N\times G\rightarrow\widehat{G},\quad(n,g)\mapsto n\cdot\sigma(g)
\]
is a bijection which, while it is generally not continuous, restricts to a local diffeomorphism on an identity neighbourhood. We will use $\Phi$ to identify  $\widehat{G}$ with $N\times G$ endowed with the operation
\begin{align}
(n,g)(n',g')=(n\cdot T(g)(n')\omega(g,g'),gg'),\label{formula2.1}
\end{align}
where
\[
T:=C_N\circ\sigma\colon G\rightarrow\mathrm{Aut}(N)\quad\text{for}\quad C_N\colon\widehat{G}\rightarrow\mathrm{Aut}(N),\quad C_N(g)(n):=gng^{-1},\]
\[
\text{and}\quad\omega\colon G\times G\rightarrow N,\quad (g,g')\mapsto\sigma(g)\sigma(g')\sigma(gg')^{-1}.
\]
Let $N\times_{(T,\omega)} G$ be the set $N\times G$ with the above operation. We want to show that in this situation $N\times_{(T,\omega)} G$ can be turned into a Lie group. Note that $\omega$ is smooth on an identity neighbourhood and the map $G\times N\rightarrow N,(g,n)\mapsto T(g)(n)$ is smooth on the set $U\times N$, where $U$ is the identity neighbourhood from above. Direct verification shows that the maps $T$ and $\omega$ satisfy the relations 
\begin{align}
\sigma(g)\sigma(g')=\omega(g,g')\sigma(gg'),\label{formula2.4}
\end{align}
\begin{align}
T(g)T(g')=C_N\big(\omega(g,g')\big)T(gg'),\label{formula2.5}
\end{align}
and
\begin{align}
\omega(g,g')\omega(gg',g'')=T(g)\big(\omega(g',g'')\big)\omega(g,g'g'')\label{formula2.6}
\end{align}
for all $g,g',g''\in G$. It can be shown that the operation (\ref{formula2.1}) turns $N\times_{(T,\omega)}G$ into a group if and only if (\ref{formula2.4}), (\ref{formula2.5}) and (\ref{formula2.6}) are satisfied. 
The fact that $T(1_G)=\mathrm{id}_N$ assures the existence of an identity element, (\ref{formula2.5}) and (\ref{formula2.6}) are necessary for the associativity of the group multiplication and together with (\ref{formula2.4}) formulas for the inversion and conjugation can be found (\cite[cf. Lemma II.7]{Neeb}).
\end{bem}

\begin{mydef}\cite[cf. Definition II.6]{Neeb}\label{II.6}
Let $G$ and $N$ be Lie groups and assume we have a map $T\colon G\rightarrow\mathrm{Aut}(N)$ with $T(1_G)=\mathrm{id}_N$ such that 
\[
U\times N\rightarrow N,\quad (g,n)\mapsto T(g)(n)
\]
is smooth on an open identity neighbourhood $U\subseteq G$. Also, let $\omega\colon G\times G\rightarrow N$ be a map that is smooth on an identity neighbourhood such that $\omega(1_G,g)=\omega(g,1_G)=1_N$ for all $g\in G$. We say $\omega$ is \textit{normalized} if it has the latter property. Furthermore, we require that the maps
\[
\omega_g\colon G\rightarrow N,\quad x\mapsto \omega(g,x)\omega(gxg^{-1},g)^{-1}
\]
are also smooth on an identity neighbourhood $U_g$ of $G$ for every $g\in G$. If additionally the relations
\[
T(g)T(g')=C_N\big(\omega(g,g')\big)T(gg')\text{   and }
\]
\[
\omega(g,g')\omega(gg',g'')=T(g)\big(\omega(g',g'')\big)\omega(g,g'g'')
\]
hold (compare formula (\ref{formula2.5}) and (\ref{formula2.6}) above), then we call the pair $(T,\omega)$ a \textit{smooth factor system}.
\end{mydef}
We come to the main tool of this section.
\begin{prop}\cite[Proposition II.8]{Neeb}\label{II.8}
If $(T,\omega)$ is a smooth factor system, then $N\times_{(T,\omega)}G$ carries a unique structure of a Lie group for which the map
\[
N\times G\rightarrow N\times_{(T,\omega)}G,\quad(n,g)\mapsto (n,g)
\]
is smooth on a set of the form $N\times U$, where $U$ is an open identity neighbourhood in $G$ and (\ref{formula2.1}) defines the operation on the right-hand side. Then
\[
q\colon N\times_{(T,\omega)}G\rightarrow G,\quad(n,g)\mapsto g
\]
is a Lie group extension of $G$ by $N$. Each Lie group extension $q\colon\widehat{G}\rightarrow G$ of $N$ gives rise to smooth factor systems by choosing a locally smooth normalized section $\sigma\colon G\rightarrow\widehat{G}$ and defining $(T,\omega):=(C_N\circ\sigma,\omega)$ where $\omega(g,g'):=\sigma(g)\sigma(g')\sigma(gg')^{-1}$ (see \ref{II.5}). Then the extension $q\colon\widehat{G}\rightarrow G$ is equivalent to $N\times_{(T,\omega)} G$.
\end{prop}

\subsection{The Lie group structure of the automorphism group}
We will now construct a smooth factor system $(T,\omega)$, or rather the restriction of one, that will turn $\mathrm{Aut}_c(\mathcal{P})\cong \mathrm{Gau}_c(\mathcal{P})\times_{(T,\omega)}\mathrm{Diff}_c(M)_{\mathcal{P}}$ into a Lie group. We begin by fixing several trivializing systems. 

\begin{bem}\label{2.2}
For the rest of this section, we choose and fix particular locally finite smooth trivializing systems $\overline{\mathcal{V}}=(\overline{V}_i,\sigma_i)_{i\in\mathbb{N}}$ and $\overline{\mathcal{U}}=(\overline{U}_i,\tau_i)_{i\in\mathbb{N}}$ of $\mathcal{P}$ such that:
\begin{itemize}
\item $\overline{\mathcal{V}}$ is a closed trivializing system and each $\overline{V}_i$ is a compact submanifold with corners of $M$,
\item $\overline{\mathcal{V}}$ is a refinement of $\overline{\mathcal{U}}$ such that  $\sigma_i=\tau_i\big|_{\overline{V}_i}$,
\item $\overline{\mathcal{U}}$ is a closed trivializing system and each $\overline{U}_i$ is a compact submanifold with corners of $M$,
\item $\mathcal{V}=(V_i,\sigma_i\big|_{V_i})_{i\in\mathbb{N}}$ and $\mathcal{U}=(U_i,\tau_i\big|_{U_i})_{i\in\mathbb{N}}$ are the respective underlying open trivializing systems,
\item $\mathcal{P}$ has the property $\text{SUB}_\oplus$ with respect to $\overline{\mathcal{U}}$ (and thus with respect to $\overline{\mathcal{V}}$ by \ref{1.12}).
\end{itemize}
To construct trivializing systems as above we start with an arbitrary locally finite smooth trivializing system such that $\mathcal{P}$ has the property $\text{SUB}_\oplus$ with respect to this system. Note that this exists because throughout this section we assume $\mathcal{P}$ having the property $\text{SUB}_\oplus$. Applying \ref{refinementcorners} two times, we then find refinements  $\overline{\mathcal{U}}$ and $\overline{\mathcal{V}}$ as needed. We let $k_{ij}$ be the transition functions on $\mathcal{U}$ and since the transition functions on $\overline{\mathcal{V}}$ arise as restrictions we will use the same notation for them.
\end{bem}

We want to follow \ref{II.5} for the construction of our smooth factor system. Hence, we need to find a section $S$ of $\mathcal{Q}$ and it becomes necessary to lift diffeomorphisms on $M$ to bundle automorphisms. Since in general our bundle is not trivial, we initially only achieve this for diffeomorphisms on $M$ whose support is contained in a given trivialization. Later on we explain how to split up any function $g\in\mathrm{Diff}_c(M)$ into suitable diffeomorphisms that are supported in this way. For this, some knowledge of the Lie group structure on $\mathrm{Diff}_c(M)$ is required, which we will introduce in a moment.

\begin{bem}\label{2.3}
Let $U\subseteq M$ be open and trivializing with section $\sigma\colon U\rightarrow P$ and corresponding $k_\sigma\colon\pi^{-1}(U)\rightarrow K$, given by $\sigma(\pi(p))\cdot k_{\sigma}(p)=p$ for $p\in\pi^{-1}(U)$ as in \ref{1.1}. If $g\in\mathrm{Diff}_c(M)$ with $\mathrm{supp}(g)\subseteq U$, then the function $\widetilde{g}$ defined by
\[
\widetilde{g}(p)=
\begin{cases}
   \sigma\big(g(\pi(p))\big)\cdot k_{\sigma}(p) & \text{if } p\in\pi^{-1}(U) \\
   p       & \text{else } 
  \end{cases}
\]
is a smooth bundle automorphism because each $x\in\partial U$ has a neighbourhood on which $g$ is the identity and $\pi\circ\widetilde{g}=g\circ\pi$. The latter relation immediately implies $\mathcal{Q}(\widetilde{g})=\widetilde{g}_M=g$ where $\mathcal{Q}$ is the homomorphism $\mathcal{Q}\colon\mathrm{Aut}_c(\mathcal{P})\rightarrow\mathrm{Diff}_c(M)$ from \ref{1.1b}.
Furthermore, we have $\mathrm{supp}(g)=\mathrm{supp}(g^{-1})$ and 
\begin{align*}
\widetilde{g^{-1}}\circ\widetilde{g}(p)&
=\sigma\bigg(g^{-1}\Big(\pi\Big(\sigma(g(\pi(p)))\cdot k_\sigma(p)\Big)\Big)\bigg)\cdot k_\sigma\bigg(\sigma(g(\pi(p)))\cdot k_\sigma(p)\bigg)\\
&=\sigma(\pi(p))\cdot k_\sigma\bigg(\sigma\big(g(\pi(p))\big)\bigg)\cdot k_{\sigma}(p)=p
\end{align*}
for $p\in\pi^{-1}(U)$. The same calculation for $\widetilde{g}(\widetilde{g^{-1}})(p)$ shows  $\widetilde{g^{-1}}=\widetilde{g}^{-1}$. 
\end{bem}

We continue with some facts about the topology on $\mathrm{Diff}_c(M)$.

\begin{satz}\label{5.11}\cite[Theorem 5.11]{GloPatched}, \cite[cf. Theorem 11.11]{Michor}
Let $M$ be a $\sigma$-compact finite-dimensional smooth manifold and $g$ some smooth Riemannian metric on $M$. The group $\mathrm{Diff}(M)$ can be made into a Lie group such that the following condition is satisfied:\\
There exists a zero-neighbourhood $\Omega_g$ in the space of compactly carried smooth vector fields $\mathcal{V}_c(M)$, such that $\varphi^{-1}\colon\Omega_g\rightarrow\mathrm{Diff}(M),\quad\xi\mapsto\mathrm{exp}_g\circ\xi$ is a  $C^{\infty}$-diffeomorphism from $\Omega_g$ onto an open submanifold $\mathcal{O}_g:=\varphi^{-1}(\Omega_g)\subseteq\mathrm{Diff}(M)$. In particular, $\varphi$ is a chart of $\mathrm{Diff}(M)$. The topology of $\mathrm{Diff}(M)$ is independent of the chosen Riemannian metric. 
\end{satz}

\begin{bem}{(The Lie group structure on $\mathrm{Diff}_c(M)$)}\label{depend}
Note that $\mathcal{O}_g\subseteq\mathrm{Diff}_c$, the subgroup of compactly carried diffeomorphisms. Thus $\mathrm{Diff}_c(M)$ can be considered as an open Lie subgroup of $\mathrm{Diff}(M)$ (\cite[p. 26, 5.12]{GloPatched}). Here, the support of $g\in\mathrm{Diff}_c(M)$ is defined as the closure of $\{x\in M\colon g(x)\neq x\}$.
Choosing $\Omega_g$ small enough (which we always assume) we can achieve that $\varphi^{-1}(\xi)(m)=m$ if and only if $\xi(m)=0_m\in TM$ for $\xi\in\Omega_g$.\\
Now, assume $\Omega_g\ni\xi=\varphi(g)$ and thus $g=\mathrm{exp}_g\circ\xi$. We then have $g(x)=\mathrm{exp}_g\big|_{T_xM}(\xi(x))$ which implies $\varphi(g)(x)=\xi(x)=(\mathrm{exp}_g\big|_{T_xM})^{-1}(g(x))$ and thus $\varphi(g)(x)$ depends only on $g(x)$.
\end{bem}

\begin{bem}\label{2.4}
For the following applications we shall modify the zero neighbourhood $\mathcal{O}_g\subseteq\mathrm{Diff}_c(M)$ from \ref{5.11}. Let $(V_i)_{i\in\mathbb{N}}$ be the open cover of $M$ from \ref{2.2}. Furthermore, let $(f_i)_{i\in\mathbb{N}}$ with $f_i\in C^{\infty}(M,[0,1])$ be a smooth partition of unity subordinate to $(V_i)_{i\in\mathbb{N}}$. 
Then the set 
\[\Omega:=\{\xi\in\Omega_g\colon (f_1+\dots+f_n)\cdot \xi\in\Omega_g\text{ for all }n\in\mathbb{N}\}\] 
is an open zero neighbourhood because for every compact subset $K\subseteq M$ there exists an $N\in\mathbb{N}$ such that $V_n\cap K=\varnothing$ for $n>N$, hence
\[
\mathcal{V}_K(M)\cap\Omega=\{\xi\in\Omega_g\colon (f_1+\dots+f_n)\cdot \xi\in\Omega_g\}\cap\mathcal{V}_K(M)
\]
for $n=0,\ldots, N$.
Because multiplication with the sum of finitely many functions $f_i$ is continuous by \cite[p. 107 Corollary F.13]{GloTF}, this set is open in $\mathcal{V}_K(M)$ and since $K$ was arbitrary we see that $\Omega\subseteq\Omega_g$ is open by definition of the direct limit topology on $\mathcal{V}_c(M)$ (compare \ref{directlim}). This restriction of $\Omega_g$ is crucial in the following steps and was overlooked in \cite[Remark 2.4]{Wockel3}.
We define $\mathcal{O'}:=\varphi^{-1}(\Omega)\cap\varphi^{-1}(\Omega)^{-1}$. \\ 
Let $(W_i)_{i\in\mathbb{N}}$ and $(\widetilde{W}_i)_{i\in\mathbb{N}}$ be two locally finite covers of $M$ with relatively compact sets such that $\overline{W}_i\subseteq \widetilde{W}_i$ for all $i\in\mathbb{N}$. By construction of the topology on $\mathrm{Diff}_c(M)$ in \cite{GloPatched} there exists an open unit neighbourhood $\mathcal{O}_W\subseteq \mathrm{Diff}_c(M)$ such that $\psi(\overline{W}_i)\subseteq \psi(\widetilde{W}_i)$ for all $\psi\in\mathcal{O}_W$ and $i\in\mathbb{N}$. Given the locally finite cover from \ref{2.2} and after shrinking the set $\mathcal{O}'$, we can assume that $\psi(\overline{V}_i)\subseteq U_i$ for all $\psi\in\mathcal{O'}$. 
Finally, there exists a symmetric open unit neighbourhood $\mathcal{O}\subseteq\mathcal{O'}$ with $\mathcal{O}\circ\mathcal{O}\circ\mathcal{O}\subseteq\mathcal{O'}$. We further demand that
\[
(f_1+\dots+f_n)\cdot \xi\in\Omega'\text{ for } \xi\in\varphi(\mathcal{O}), n\in\mathbb{N}
\] 
holds. In the remainder, we fix the sets $\mathcal{O'}$ and $\mathcal{O}$.
\end{bem}

Now we are ready to decompose $g\in\mathcal{O}$ into $(g_i)_{i\in\mathbb{N}}$ such that $\mathrm{supp}(g_i)\subseteq V_i$. By \ref{2.3} we can then lift these maps to bundle automorphisms $\widetilde{g}_i$ which we will use to define a lift for $g$.\\
 While our strategy is the same as in \cite[Lemma 2.5]{Wockel3}, there are some subtle differences. Since we want to define the lift of $g$ as a limit $\lim_{i\to\infty}s_1(g)\circ\cdots\circ s_i(g)=g$, we need to change the order of the factors in the definition of $s_i$ in the following lemma (compared to \cite{Wockel3}).

\begin{lem}\label{2.5}
Let $M$ be a finite-dimensional $\sigma$-compact smooth manifold and let $(f_i)_{i\in\mathbb{N}}$ and $\mathcal{O}\subseteq\mathrm{Diff}_c(M)$ be defined as in \ref{2.4}. Then there exist smooth maps
\[s_i\colon\mathcal{O}\rightarrow\mathcal{O}'^{-1}\circ\mathcal{O}'\]
for $i\in\mathbb{N}$ such that $\mathrm{supp}(s_i(g))\subseteq V_i$ and $\lim_{i\to\infty}s_1(g)\circ\cdots\circ s_i(g)=g$ as a pointwise limit, for all $g\in\mathcal{O}$.
\end{lem}
\begin{proof}
Let $\varphi\colon\mathcal{O}\rightarrow\varphi(\mathcal{O})=:\Omega\subseteq\mathcal{V}_c(M)$ be the chart from \ref{5.11} and let $g\in\mathcal{O}$. Recall that $\varphi^{-1}(\xi)(m)=m$ if and only if $\xi(m)=0_m$.  
By our choice of $\Omega$, we can define $s_i\colon\mathcal{O}\rightarrow\mathcal{O}'^{-1}\circ\mathcal{O}'$ by
\begin{align*}
&s_1(g):=\varphi^{-1}(f_1\cdot\varphi(g)),\\
&s_2(g):=\left(\varphi^{-1}(f_1\cdot\varphi(g))\right)^{-1}\circ\varphi^{-1}\left((f_1+f_2)\cdot\varphi(g)\right)\text{ and }\\
&s_i(g):=\left(\varphi^{-1}((f_1+\dots+f_{i-1})\cdot\varphi(g))\right)^{-1}\circ\varphi^{-1}\left((f_1+\dots+f_i)\cdot\varphi(g)\right)\text{ for }2<i.
\end{align*}
These maps are diffeomorphisms because $\varphi$ is a chart. Furthermore, the limit is well-defined because for each $g\in\mathcal{O}$ we have $\mathrm{supp}(g)=\mathrm{supp}(\varphi(g))$ and since there exists a minimal $N\in\mathbb{N}$ such that $V_n\cap\mathrm{supp}(g)=\varnothing$ for all $n>N$ we have $(f_1+\dots+f_n)\cdot\varphi(g)=(f_1+\dots+f_{n+1})\cdot\varphi(g)$, hence $s_n(g)=\mathrm{id}_{\mathrm{Diff}_c(M)}$. For such an $N$ the rightmost factor that can be nontrivial is $\varphi^{-1}((f_1+\dots+f_N)\cdot\varphi(g))=\varphi^{-1}(\varphi(g))=g$ and all other factors annihilate each other. If $f_i(m)=0$, we see with the same argument that both factors of $s_i$ neutralize each other, i.e., $s_i(g)(m)=m$, which shows $\mathrm{supp}(s_i(g))\subseteq V_i$.
\end{proof}

\begin{mydef}\label{2.6}
If $\mathcal{O}\subseteq\mathrm{Diff}_c(M)$ is the open identity neighbourhood from \ref{2.4} and $s_i\colon\mathcal{O}\rightarrow\mathcal{O}'\circ\mathcal{O}'^{-1}$ are the smooth mappings from \ref{2.5}, then we define
\begin{align}
S\colon\mathcal{O}\rightarrow\mathrm{Aut}_c(\mathcal{P}),\quad g\mapsto S(g):=\lim_{n\to\infty}\widetilde{g_1}\circ\cdots\circ\widetilde{g_n},\label{formula5.1}
\end{align}
where $g_i:=s_i(g)$ and $\widetilde{g_i}:=\widetilde{s_i(g)}$ is defined as the bundle automorphism of $\mathcal{P}$ from \ref{2.3}. The map $S$ is well-defined since $\mathrm{supp}(s_i(g))\subseteq V_i$. Furthermore there exists an $N\in\mathbb{N}$ such that $V_n\cap\mathrm{supp}(g)=\varnothing$ for all $n>N$, which implies $S(g)=\widetilde{g_1}\circ\cdots\circ\widetilde{g_N}$. Hence $\mathcal{Q}\circ S=\mathrm{id}_\mathcal{O}$. In other words $S$ is a local section for the homomorphism $\mathcal{Q}\colon\mathrm{Aut}_c(\mathcal{P})\rightarrow\mathrm{Diff}_c(M), F\mapsto F_M$ from \ref{1.1b}.
\end{mydef}

We would like to derive a formula for $S(g)$ in terms of local trivializations, i.e., express $S(g)\circ\sigma_i$ with $g_j,\sigma_j$ and $k_{jj'}$. This will be needed to show smoothness of maps with values in $\mathrm{Gau}_c(\mathcal{P})$, since we can only access the Lie group structure on the gauge group by the isomorphism $\mathrm{Gau}_c(\mathcal{P})\cong G_{\overline{\mathcal{V}}}(\mathcal{P})$.

\begin{bem}\label{2.7}\cite[Remark 2.7]{Wockel3}
First we look at an easy example to understand the general strategy to obtain such a formula. Let $x\in \overline{V}_i$ be such that $x\notin \overline{V}_j$ for $j<i$ and $g_i(x)\notin \overline{V}_j$ for $j>i$. This means $g_j(x)=x$ for all $j<i$ and $g_i(g_j(x))=g_i(x)$ for all $j>i$, hence  $S(g)(\sigma_i(x))=\sigma_i(g_i(x))=\sigma_i(g(x))$.\\
In general the situation is more complicated since the point $x\in M$ might be moved by successive $g_i$. We will now derive a formula to express these shifts locally. Let $g\in\mathcal{O}$ and $N\in\mathbb{N}$ minimal such that $\overline{V}_n\cap\mathrm{supp}(g)=\varnothing$ for all $n>N$ then $S(g)=\widetilde{g_1}\circ\cdots\circ\widetilde{g_N}$. Let further $j_1\in\mathbb{N}$ be the maximal index such that $x\in \overline{V}_{j_1}$, hence $\widetilde{g_{j_1}}$ is the first factor in $S(g)$ that can move $\sigma_i(x)$. We then get
\[
\widetilde{g_{j_1}}\big(\sigma_i(x)\big)=\widetilde{g_{j_1}}\big(\sigma_{j_1}(x)\big)\cdot k_{j_1i}(x)=\sigma_{j_1}\big(g_{j_1}(x)\big)\cdot k_{j_1i}(x)
\]
directly from the definition of $\widetilde{g_{j_1}}$ in \ref{2.3} since $k_{\sigma_{j_1}}(\sigma_i(x))=k_{j_1i}(x)$. The next $\widetilde{g_{j_2}}$ that could move $\widetilde{g_{j_1}}(\sigma_i(x))$ is the one for the maximal $j_2<j_1$ such that $g_{j_1}(x)\in\overline{V}_{j_2}$ and because $k_{\sigma_{j_2}}(p\cdot g)=k_{\sigma_{j_2}}(p)\cdot g$ we immediately get
\[
\widetilde{g_{j_2}}\left(\widetilde{g_{j_1}}\left(\sigma_i(x)\right)\right)=\sigma_{j_2}\left(g_{j_2}\circ g_{j_1}(x)\right)\cdot k_{j_2j_1}\left(g_{j_1}(x)\right)\cdot k_{j_1i}(x).
\]
We eventually arrive at
\begin{align}\label{formula6}
S(g)\big(\sigma_i(x)\big)=\sigma_{j_l}\big(g(x)\big)\cdot k_{j_lj_{l-1}}\left(g_{j_{l-1}}\circ\cdots\circ g_{j_1}(x)\right)\cdots\cdot\cdot k_{j_1i}(x),
\end{align}
since $g_{j_{l}}\circ\cdots\circ g_{j_1}(x)=g_1\circ\cdots\circ g_{j_1}(x)=g(x)$ because the omitted factors do not move the fixed $x\in M$. Here $\{j_l,\ldots,j_1\}\subseteq\{1,\ldots,j_1\}$ are maximal such that
\[
g_{j_{p-1}}\circ\cdots\circ g_{i_1}(x)\in \overline{V}_{j_p}\text{ for }2\leq p\leq l\text{ and } j_1>\ldots>j_p.
\]
It is not possible to use all indices in this formula because otherwise the $k_{jj'}$ and $\sigma_j$ would not be well-defined, however since $g_i(x)=x$ if $x\notin\overline{V}_i$, we may write
\begin{align}\label{2.7formula}
S(g)\big(\sigma_i(x)\big)=\sigma_{j_l}\big(g(x)\big)\cdot k_{j_lj_{l-1}}\left(g_{1}\circ\cdots\circ g_{j_1}(x)\right)\cdots\cdot\cdot k_{j_1i}(x),
\end{align}
using all $g_i$ from $g_1$ to $g_{j_1}$.\\ 
We will now show that there exist open neighbourhoods $\mathcal{O}_g$ of $g$ and $U_x$ of $x$, such that we may use the formula (\ref{formula6}) for $S(g')(\sigma_i(x'))$ with $g'\in\mathcal{O}_g$, $x'\in U_x$ without changing any of the indices.  For the formula to hold only the conditions
\begin{align}
&g_{j_{p-1}}\circ\cdots\circ g_{j_1}(x)\notin\overline{V}_j\text{ for }2\leq p\leq l\text{ and }j\notin\{j_p,\ldots,j_l\},j\leq N\text{ and}\label{formula7}\\
&g_{j_{p-1}}\circ\cdots\circ g_{j_1}(x)\in U_{j_p}\cap U_{j_{p-1}}\text{ for }2\leq p\leq l\label{formula8}
\end{align}
need to be satisfied for the given indices $j_1>\ldots>j_l$. Since the action of the diffeomorphisms $\alpha\colon\mathrm{Diff}_c(M)\times M\rightarrow M,g.m=g(m)$ is smooth, thus in particular continuous (\cite[cf. Proposition 6.2]{GloPatched}), we find open neighbourhoods of $g$ and $x$ where each of these finitely many conditions hold. After intersecting these neighbourhoods we obtain open neighbourhoods $U_x\subseteq U_i$ of $x$ and $\mathcal{O}_g$ of $g$ for which formula (\ref{formula6}) holds and is well-defined. This is because (\ref{formula7}) implies $g_{j_p}\circ\cdots\circ g_{j_1}\circ g_j(x)=g_{j_p}\circ\cdots\circ g_{j_1}(x)$ (which is obvious if $j>N$) and (\ref{formula8}) implies that the transition functions $k_{j_pj_{p-1}}$ are defined. 
\end{bem}

For many of the following considerations concerning smoothness we will need that the values of $S(g)$ locally only depend on the local values of $g$.

\begin{bem}\label{2.7a}
Using formula (\ref{2.7formula}) above and the definition of $s_i\colon\mathcal{O}\rightarrow\mathcal{O}'^{-1}\circ\mathcal{O}'$ from \ref{2.5}, we arrive at 
\[
S(g)\big(\sigma_i(x)\big)=\sigma_{j_l}\big(g(x)\big)\cdot k_{j_lj_{l-1}}\big(\varphi^{-1}((f_1+\cdots+f_{j_1})\cdot\varphi(g))(x)\big)\cdots\cdot\cdot k_{j_1i}(x)
\]
for $x\in\overline{V}_i$, because subsequent terms cancel except for the last one.
Since every factor only contains terms of the form $\varphi^{-1}((f_1+\cdots+f_{j})\cdot\varphi(g))$ for some $j\in\mathbb{N}$ all the factors are contained in $\mathcal{O}'$ and $S(g)\circ\sigma_i$ depends only on $g\big|_{\overline{V}_i}$ (see \ref{depend}).\\
Furthermore, since $g(x)\in U_i$ we have $\sigma(g(x))_{j_1}\cdot k_{j_1i}(x)=\tau_i(x)$, hence $\pi(S(g(x)))\in U_i$.
\end{bem}

Since we defined the Lie group structure on $\mathrm{Gau}_c(\mathcal{P})$ only indirectly via the isomorphism $\mathrm{Gau}_c(\mathcal{P})\cong C^\infty_c(P,G)^G$, we give a description of how the conjugation of $\mathrm{Aut}_c(\mathcal{P})$ on $\mathrm{Gau}_c(\mathcal{P})$ corresponds to the action of $\mathrm{Aut}_c(\mathcal{P})$ on $C^\infty_c(P,G)^G$. We need this conjugation to proceed as in \ref{II.5} in defining our smooth factor system.

\begin{bem}\label{2.8}
Recall that $\mathrm{Gau}_c(\mathcal{P})$ is a normal subgroup of $\mathrm{Aut}_c(\mathcal{P})$ as it the kernel of a homomorphism, also recall the identification of the gauge group with $C_c^{\infty}(P,G)^G$ via
\[
C_c^{\infty}(P,G)^G\rightarrow\mathrm{Gau}_c(\mathcal{P}),\quad\gamma\mapsto F_\gamma
\]
where $F_{\gamma}(p)=p\cdot\gamma(p)$ (see \ref{1.2}). The conjugation action 
\[
c\colon\mathrm{Aut}_c(\mathcal{P})\times\mathrm{Gau}_c(\mathcal{P})\rightarrow\mathrm{Gau}_c(\mathcal{P})
\]
is given by $c_F(F_\gamma)=F\circ F_\gamma\circ F^{-1}$. Under the above identification this changes into
\[
c\colon\mathrm{Aut}_c(\mathcal{P})\times C_c^{\infty}(P,G)^G \rightarrow C_c^{\infty}(P,G)^G,\quad (F,\gamma)\mapsto\gamma\circ F^{-1},
\]
because
\[
\left(F\circ F_\gamma\circ F^{-1}\right)(p)=F\left(F^{-1}(p)\cdot\gamma\left(F^{-1}(p)\right)\right)=p\cdot\gamma\left(F^{-1}(p)\right)=F_{(\gamma\circ F^{-1})}(p).
\]
\end{bem}

In the following lemmas we define the maps that will constitute our smooth factor system and check the local smoothness conditions. 
We will have to show that certain mappings with values in $C^\infty_c(P,G)^G$, resp. in $\mathrm{Gau}_c(\mathcal{P})$ are smooth. In the case of compact base $M$ it suffices to show that the composition of these maps with the pullback $\sigma_i^\ast$ of the sections $\sigma_i\colon\overline{V}_i\rightarrow P$ are smooth for each of the finitely many $i$. In our situation this is not sufficient and we must use arguments that relate to maps between direct sums.\\
However, we still need to see the smoothness of the local representations and can borrow heavily from the proofs in \cite{Wockel3} to do so. In general, we will use the explicit formula for $S(g)\circ\sigma_i(x)$ derived in \ref{2.7} and show that for some neighbourhood $U_x$ of $x$ and $\mathcal{O}_g$ of $g$ the formula depends smoothly on $g$. Then we can cover $\overline{V}_i$ with finitely many such neighbourhoods and (using the restriction and gluing maps) see that $S(g)\circ\sigma_i$ depends smoothly on $g$.\\
We begin with a lemma that will be used to see the smoothness of these explicit formulae. It uses new tools and simplifies some technical aspects of \cite{Wockel3}.

\begin{lem}\label{smooth}
Let $M$ be finite-dimensional manifold and $G$ be a Lie group. Further, let $\overline{U}\subseteq M$ and $\overline{U}'\subseteq M$ be compact manifolds with corners. If $\mathcal{O}''$ is an open set in $\mathrm{Diff}_c(M)$ such that $g(\overline{U})\subseteq\overline{U}'$ for all $g\in\mathcal{O}''$, then the map
\[
\alpha\colon\mathcal{O}''\times C^\infty(\overline{U}',G)\rightarrow C^\infty(\overline{U},G),\quad
(g,\gamma)\mapsto \gamma\circ g\big|_{\overline{U}}
\]
is smooth.
\end{lem}
\begin{proof}
From \ref{eval} it directly follows that
\[
\epsilon_1\colon C^\infty(\overline{U}',G)\times\overline{U}'\rightarrow G,\quad
(\gamma,x)\mapsto\gamma(x)
\]
is a smooth map. Also, it is known (see \cite{GloPatched}) that
\[
\epsilon_2\colon\mathcal{O}''\times M\rightarrow M,\quad 
(g,x)\mapsto g(x)
\]
is a smooth map.  Now, by \ref{exp}, $\alpha$ is smooth if and only if
\[
\widehat{\alpha}\colon\mathcal{O}\times C^\infty(\overline{U}',G)\times\overline{U}\rightarrow G,\quad (g,\gamma,x)\mapsto \gamma(g(x))
\]
is smooth. But this is obvious because we have 
\[
\widehat{\alpha}\big(g,\gamma,x\big)=\gamma\big(\epsilon_2(g,x)\big)=\epsilon_1(\gamma,\epsilon_2(g,x)).
\]
\end{proof}

Since we want to apply \ref{2.12a} we need to have suitable direct sums of locally convex spaces on the left-hand side. The next lemma provides the necessary steps to achieve this.

\begin{lem}\label{2.9lem}
Let $\pi\colon E\rightarrow M$ be a smooth vector bundle, where $M$ is a finite-dimensional $\sigma$-compact smooth manifold without corners. Also, let $\mathcal{U}=(U_i)_{i\in\mathbb{N}}$ be an open locally finite cover of $M$ by relatively compact sets. If $\mathcal{V}=(V_j)_{j\in\mathbb{N}}$ is a closed refinement of $\mathcal{U}$ such that $V_i\subseteq U_i$, then there exists a continuous linear map
\[
\Phi\colon C^\infty_c(M,E)\rightarrow\bigoplus_{i\in\mathbb{N}}C_c^\infty(M,E)
\]
such that $(\Phi(\eta))_i\big|_{V_i}=\eta\big|_{V_i}$ for all $i\in\mathbb{N}$.
\end{lem}
\begin{proof}
Using \cite[Proposition F.19(a)]{GloTF} we get a continuous linear embedding 
\[
C_c^{\infty}(M,E)\hookrightarrow\bigoplus_{i\in\mathbb{N}} C^{\infty}(U_i,E\big|_{U_i})\quad \eta\mapsto\left(\eta\big|_{U_i}\right)_{i\in\mathbb{N}}.
\] 
Let $h_i\colon U_i\rightarrow\mathbb{R}$ be smooth functions such that $h_i\big|_{V_i}\equiv 1$ and $K_i:=\mathrm{supp}(h_i)\subseteq U_i$ is compact. Because $C^{\infty}(U_i,E\big|_{U_i})$ is a topological $C^{\infty}(U_i,\mathbb{R})$-module (\cite[Corollary F.13]{GloTF}) and $C^{\infty}(U_i,E\big|_{U_i})$ induces the topology on $C_{K_i}^{\infty}(U_i,E\big|_{U_i})$ we have that
\[
\bigoplus_{i\in\mathbb{N}} C^{\infty}(U_i,E\big|_{U_i})\rightarrow
\bigoplus_{i\in\mathbb{N}}C^{\infty}_{K_i}(U_i,E\big|_{U_i}),\quad(\eta_i)_{i\in\mathbb{N}}\mapsto(\eta_i\cdot h_i)_{i\in\mathbb{N}}
\]
is continuous and linear. We further have $C^{\infty}_{K_i}(U_i,E\big|_{U_i})\cong C^{\infty}_{K_i}(M,E)$ for all $i\in\mathbb{N}$ by \cite[Lemma F.15(b)]{GloTF} and using the inclusion map from \cite[Proposition F19(b)]{GloTF} we have $C_{K_i}^\infty(M,E)\hookrightarrow C^\infty_c(M,E)$. By composition we finally get a continuous linear map 
\[
\Phi\colon C^{\infty}_c(M,E)\rightarrow\bigoplus_{i\in\mathbb{N}}C^\infty_c(M,E),\quad \eta\mapsto(\widetilde{h}_i\cdot\eta)_{i\in\mathbb{N}} 
.\] 
Also, for all $\eta\in C^\infty_c(M,E)$ we have $(\Phi(\eta))_i\big|_{V_i}=\eta\big|_{V_i}$ by construction because $h_i\big|_{V_i}\equiv 1$.
\end{proof}

Since we modelled $C^\infty_c(P,G)^G$ on $\mathfrak{g}_{\overline{\mathcal{V}}}(\mathcal{P})$ we will take a slightly different approach than the one used in \cite[Lemma 2.9]{Wockel3} to prove the next lemma. 

\begin{lem}\label{2.9}
Let $\mathcal{O}\subseteq\mathrm{Diff}_c(M)$ be the open identity neighbourhood from \ref{2.4} and let $S\colon\mathcal{O}'\rightarrow\mathrm{Aut}_c(\mathcal{P})$ be the map from \ref{2.6}. Then there exists an open set $U\subseteq C_c^{\infty}(P,G)^G$ such that the map
\[
t\colon U\times\mathcal{O}\rightarrow C_c^{\infty}(P,G)^G,\quad(\gamma,g)\mapsto\gamma\circ S(g)^{-1}
\] 
is smooth.
\end{lem}
\begin{proof}
Let $\varphi\colon\mathcal{O}\rightarrow\Omega$ be the chart from \ref{2.4} and 
let $(\phi_i\colon W_i\rightarrow W'_i)_{i\in\mathbb{N}}$ be a family of centered charts of $G$ for which $G_{\overline{\mathcal{U}}}(\mathcal{P})$ has the property $\text{SUB}_\oplus$. Then the sets $U'':=\mathfrak{g}_{\mathcal{\overline{U}}}(\mathcal{P})\cap\bigoplus_{i\in\mathbb{N}} C^\infty(\overline{U}_i,W'_i)$ and  $U':=G_{\mathcal{\overline{U}}}(\mathcal{P})\cap\prod_{i\in\mathbb{N}}^\ast C^\infty(\overline{U}_i,W_i)$ are open and by the isomorphism $C_c^{\infty}(P,G)^G\cong G_{\mathcal{\overline{U}}}(\mathcal{P})$ we get a respective open set $U\subseteq C_c^{\infty}(P,G)^G$.\\
We want to derive an explicit formula for the function $t$. Let $g\in\mathcal{O}$ and $x\in\overline{V}_i$. As in \ref{2.7}, we find $i_1,\ldots,i_l$ and open neighbourhoods $\mathcal{O}_g$ of $g$ and a relatively open neighbourhood $U_x\subseteq\overline{V}_i$ of $x$ (w.l.o.g. so that $\overline{U}_x$ is a submanifold with corners of $\overline{V}_i$) such that
\begin{align*}
S(g')^{-1}\left(\sigma_{i}(x')\right)&=\sigma_{i_l}\left(g'^{-1}(x')\right)\cdot k_{i_li_{l-1}}\left((g_{l-1}'^{-1}\circ\cdots\circ g'^{-1}_{i_1})(x')\right)\cdots\cdot\cdot k_{{i_1}i}(x')\\
&=\tau_i(g'^{-1}(x'))\cdot\\
&\underbrace{k_{ii_l}(g'^{-1}(x'))\cdot k_{i_li_{l-1}}\left((g_{l-1}'^{-1}\circ\cdots\circ g'^{-1}_{i_1})(x')\right)\cdots\cdot\cdot k_{{i_1}i}(x')}_{\kappa_{x,g'}(x'):=}
\end{align*}
for all $g'\in\mathcal{O}_g$ and $x'\in \overline{U}_x$, because $g'^{-1}(x')\in U_i$. By \ref{2.7a} the values of this formula depend only on $g'\big|_{U_i}$. Since we will not vary $i$ and $g$ in the sequel, we suppressed the dependence of $\kappa_{x,g'}$ on them. We will show that for fixed $x$, the formula for $\kappa_{x,g'}$ defines a smooth function on $\overline{U}_x$ that depends smoothly on $g'$. In fact, we have $g'(\overline{U}_x)\subseteq\overline{U}_i$ and thus 
\[
\mathcal{O}\times C^\infty(\overline{U}_i,G)\rightarrow C^\infty(\overline{U}_x,G),\quad (g,\gamma)\mapsto \gamma\circ g\big|_{\overline{U}_x}
\]
is smooth by \ref{smooth}. Since $C^\infty(\overline{U}_x,G)$ is a Lie group, finite products of Lie group elements are smooth and $g'_i$ depends smoothly on $g'$ because the $s_i$ from \ref{2.5} are smooth, hence $\kappa_{x,g'}$ depends smoothly on $g'$.\\ 
Furthermore, $\kappa_{x_1,g'}(x')$ and $\kappa_{x_2,g'}(x')$ coincide on $\overline{U}_{x_1}\cap\overline{U}_{x_2}$, because both define $S(g')^{-1}(\sigma_{i}(x'))$ there. Now, finitely many $U_{x_1},\ldots,U_{x_m}$ cover $\overline{V}_i$ and since the gluing and restriction maps from \ref{A.17} and \ref{A.18} are smooth 
\[ \theta_{g,i}(g'):=\mathrm{glue}(\kappa_{x_1,g'}\big|_{\overline{U}_{x_1}},\ldots,\kappa_{x_m,g'}\big|_{\overline{U}_{x_m}})\in C^\infty(\overline{V}_i,G)
\]
depends smoothly on $g'$. After intersecting finitely many open sets we may assume $\theta_{g,i}$ is defined on an open neighbourhood $\mathcal{O}_{g,i}\subseteq\mathcal{O}$ of $g$. There exists an $N\in\mathbb{N}$ such that $\mathrm{supp}(g)\cap\overline{U}_n=\varnothing$ for all $n>N$ which implies that $\theta_{g,n}$ is trivial. Because $\theta_{\mathrm{id}_M,i}$ is trivial for all $i\in\mathbb{N}$ we may assume that almost all $\mathcal{O}_{g,i}$ contain the identity. Recall that for a given $(\eta_i)_{i\in\mathbb{N}}\in U''$ the corresponding $\gamma\in C^\infty_c(P,G)^G$ is given by $k_{\tau_i}(p)^{-1}\cdot(\phi_i^{-1}\circ\eta_i(\pi(p)))\cdot k_{\tau_i}(p)$ for $p\in \pi^{-1}(\overline{U}_i)$. 
For $\gamma\circ S(g)^{-1}\circ\sigma_i$ and $\xi=\varphi^{-1}(g)$ we get  
\[
(C^\infty(\overline{V}_i,\mathfrak{g})\times\Omega)\ni(\eta_i,\xi)\mapsto(\eta_i,\varphi^{-1}(\xi))\mapsto\theta_{g,i}^{-1}\cdot(\phi^{-1}\circ\eta_i\circ g^{-1}\big|_{\overline{V}_i})\cdot\theta_{g,i}
\]
as local description of the components of $t$. We want to apply \ref{2.12a} and thus need to find a suitable direct sum. 
Define $\Omega_{g,i}:=\varphi^{-1}(\mathcal{O}_{g,i})$ and note that almost all $\Omega_{g,i}$ are zero neighbourhoods.
Using \ref{2.9lem} there exists a continuous linear map
$\Phi\colon C_c^\infty(M,TM)\rightarrow\bigoplus_{i\in\mathbb{N}}C_c^\infty(M,TM)$
such that $(\Phi(\xi))_i\big|_{\overline{U}_i}=\xi\big|_{\overline{U}_i}$. The map
\[
\Psi\colon\bigoplus_{i\in\mathbb{N}}\bigg( C^\infty(\overline{U}_i,W'_i)\times \Omega_{g,i}\bigg)\rightarrow\sideset{}{^\ast}\prod_{i\in\mathbb{N}} C^\infty(\overline{V}_i,G),
\]
\[
(\eta_i,\xi_i)_{i\in\mathbb{N}}\mapsto\bigg(\theta_{g,i}^{-1}\cdot(\phi_i^{-1}\circ\eta_i\circ\varphi^{-1}(\xi_i)^{-1}\big|_{\overline{V}_i})\cdot\theta_{g,i}\bigg)_{i\in\mathbb{N}}
\]
is smooth by \ref{2.12a}, \ref{smooth} and the above. The open set $\Omega_{g}:=\Phi^{-1}(\bigoplus_{i\in\mathbb{N}}\Omega_{g,i})\cap\Omega$
is a neighbourhood of $\varphi^{-1}(g)$ and $\Psi\circ(\mathrm{id}_{U''}\times\Phi\big|_{\Omega_{g}})$ is smooth. By definition of $\Phi$ the components of this map coincide with the components of the local description of $t$.
\end{proof}

\begin{lem}\label{2.10}\cite[cf. Lemma 2.10]{Wockel3}
If $\mathcal{O}\subseteq\mathrm{Diff}_c(M)$ is the open identity neighbourhood from \ref{2.4} and 
if $S\colon\mathcal{O}\rightarrow\mathrm{Aut}_c(\mathcal{P})$ is the map from \ref{2.6}, then for each 
$\gamma\in C_c^{\infty}(P,G)^G$, the map
\[\mathcal{O}\ni g\mapsto\gamma\circ S(g)^{-1}\in C_c^{\infty}(P,G)^G
\]
is smooth.
\end{lem}
\begin{proof}
Let $\gamma\in C^{\infty}_c(P,G)^G$ with $\mathrm{supp}(\gamma)=K\subseteq M$ compact and $N\in\mathbb{N}$ such that $\overline{U}_i\cap K=\varnothing$ if $i>N$. Using $C^{\infty}_c(P,G)^G\cong G_{\overline{\mathcal{V}}}(\mathcal{P})$ it now suffices to show that $\gamma\circ S(g)^{-1}\circ\sigma_i\big|_{\overline{V}_i}$ depends smoothly on $g$ for $1\leq i\leq N$ because $\gamma$ is trivial outside of $\pi^{-1}(K)$ and $S(g)^{-1}\circ\sigma_i(\overline{V}_i)\subseteq\pi^{-1}(\overline{U}_i)$ by \ref{2.7a}. Let $(\gamma_1,\ldots,\gamma_N)\in G_{\overline{\mathcal{U}}}(\mathcal{P})\cap\prod_{i=1}^NC^\infty(\overline{U}_i,G)\subseteq\bigoplus_{i\in\mathbb{N}}C^\infty(\overline{U}_i,G)$ be the local description of $\gamma$. Fix $g\in\mathcal{O}$ and $x\in\overline{V}_i$. Then \ref{2.7} yields an open neighbourhood $\mathcal{O}_g$ of $g$ and a relatively open neighbourhood $U_x\subseteq\overline{V}_i$ of $x$ (w.l.o.g. such that $\overline{U}_x\subseteq\overline{V}_i$ is a submanifold with corners) such that
\begin{align*}
\gamma\big(S(g')^{-1}\left(\sigma_i(x')\right)\big)&=\gamma\bigg(\left(\sigma_{j_l}(g'^{-1}(x'))\right)\cdot
\underbrace{k_{j_lj_{l-1}}\left(g'^{-1}_{j_{l-1}}\circ\cdots\circ g'^{-1}_{j_1}(x')\right)\cdot\cdot\cdots k_{j_1i}(x')}_{:=\kappa_{x,g'}(x')}\bigg)\\
&=\kappa_{x,g'}(x')^{-1}\cdot\gamma\left(\sigma_{j_l}\left(g'^{-1}(x')\right)\right)\cdot\kappa_{x,g'}(x')\\
&=\underbrace{\kappa_{x,g'}(x')^{-1}\cdot\gamma_{j_l}\left(g'^{-1}(x')\right)\cdot\kappa_{x,g'}(x')}_{:=\theta_{x,g'}(x')}
\end{align*}
for all $g'\in\mathcal{O}_g$ and $x'\in\overline{U}_x$. Since we will not vary $i$ and $g$ in the sequel, we suppressed the dependence of $\kappa_{x,g'}(x')$ and $\theta_{x,g'}(x')$ on them. For fixed $x$, the formula for $\theta_{x,g'}$ defines a smooth function on $\overline{U}_x$ that depends smoothly on $g'$. In fact, we have $g'(\overline{U}_x)\subseteq\overline{U}_i$ and thus 
\[
\mathcal{O}'\times C^\infty(\overline{U}_i,G)\rightarrow C^\infty(\overline{U}_x,G),\quad (g,\gamma)\mapsto \gamma\circ g\big|_{\overline{U}_x}
\]
is smooth by \ref{smooth}. Furthermore, $C^\infty(\overline{U}_x,G)$ is a Lie group, finite products of Lie group elements are smooth and $g'_i$ depends smoothly on $g'$ because the $s_i$ from \ref{2.5} are smooth.\\
Note that $\theta_{x_1,g'}(x')$ and $\theta_{x_2,g'}(x')$ coincide on $\overline{U}_{x_1}\cap\overline{U}_{x_2}$, because both define $\gamma\circ S(g')^{-1}\circ\sigma_i$ there. Now finitely many $U_{x_1},\ldots,U_{x_m}$ cover $\overline{V}_i$ and since the gluing and restriction maps from \ref{A.17} and \ref{A.18} are smooth, 
\[
\gamma\circ S(g')^{-1}\circ\sigma_i=\mathrm{glue}\left(\theta_{x_1,g'}\big|_{\overline{U}_{x_1}},\ldots,\theta_{x_m,g'}\big|_{\overline{U}_{x_m}}\right)
\]
depends smoothly on $g'$.
\end{proof}

\begin{lem}\label{2.11}\cite[cf. Lemma 2.11]{Wockel3}
Let $\mathcal{O}\subseteq\mathrm{Diff}_c(M)$ be the open identity neighbourhood from \ref{2.4} and let $S\colon\mathcal{O}\rightarrow\mathrm{Aut}_c(\mathcal{P})$ be the map from \ref{2.6}. Then for each $F\in\mathrm{Aut}_c(\mathcal{P})$ the map 
$c_F\colon C_c^{\infty}(P,G)^G\rightarrow C_c^{\infty}(P,G)^G,\gamma\mapsto\gamma\circ F^{-1}$ 
is an automorphism of $C_c^{\infty}(P,G)^G$ and the map
\[T\colon C_c^{\infty}(P,G)^G\times\mathcal{O}\rightarrow C_c^{\infty}(P,G)^G,\quad(\gamma,g)\mapsto\gamma\circ S(g)^{-1}
\]
is smooth.
\end{lem}
\begin{proof}
Since $\gamma\mapsto\gamma\circ F^{-1}$ is a group homomorphism it suffices to show that it is smooth on an identity neighbourhood. Let $F_M\in\mathrm{Diff}_c(M)$ be the diffeomorphism corresponding to $F$ from \ref{1.1b}. Then $(F_M(\overline{V}_i))_{i\in\mathbb{N}}$ is a locally finite cover of $M$ by compact sets $\overline{V}'_i:=F_M(\overline{V}_i)$. These sets are trivializing with the sections $\sigma'_i:=F\circ\sigma_i\circ F_M^{-1}$ because $\pi(\sigma_i'(x))=F_M\circ\pi\circ\sigma_i\circ F_M^{-1}(x)=x$ for all $x\in \overline{V}'_i$. Furthermore, $\mathcal{P}$ has obviously the property $\text{SUB}_\oplus$ with respect to $\overline{\mathcal{V}}':=(\overline{V}'_i,\sigma'_i)_{i\in\mathbb{N}}$, using the same family of charts as for $\overline{\mathcal{V}}$. Since we only need to check smoothness on an identity neighbourhood and we have $F^{-1}\circ\sigma_i'=\sigma_i\circ F_M^{-1}$ it is enough to check the smoothness of
\[
F_{M}^\ast\colon\mathfrak{g}_{\overline{\mathcal{V}}}(\mathcal{P})\rightarrow 
\mathfrak{g}_{\overline{\mathcal{V}'}}(\mathcal{P}),\quad(\eta_i)_{i\in\mathbb{N}}\mapsto(\eta_i\circ F_M^{-1})_{i\in\mathbb{N}},
\]
as this is the local description of the map $\gamma\mapsto\gamma\circ F^{-1}$. We can define $F_{M}^\ast$ on the whole locally convex direct sum and there the components of $F_{M}^\ast$ are pullbacks and thus smooth, hence $F_{M}^\ast$ is smooth.\\
Because of \ref{2.9}, there exists a unit neighbourhood $U\subseteq C^{\infty}_c(P,G)^G$ such that 
\[
U\times\mathcal{O}\rightarrow C^{\infty}_c(P,G)^G,\quad(\gamma,g)\mapsto\gamma\circ S(g)^{-1}
\]
is smooth. Now, for each $\gamma_0\in C_c^{\infty}(P,G)^G$, there exists an open neighbourhood $U_{\gamma_0}$ with $\gamma_0^{-1}\cdot U_{\gamma_0}\subseteq U$. Hence
\[
\gamma\circ S(g)^{-1}=(\gamma_0\cdot\gamma_0^{-1}\cdot\gamma)\circ S(g)^{-1}=\left(\gamma_0\circ S(g)^{-1}\right)\cdot\left((\gamma_0^{-1}\cdot\gamma)\circ S(g)^{-1}\right),
\]
where the first factor depends smoothly on $g$ due to \ref{2.10} and the second factor depends smoothly on $\gamma$ and $g$, because $\gamma_0^{-1}\cdot\gamma\in U$.
\end{proof}

\begin{lem}\label{2.12}\cite[cf. Lemma 2.12]{Wockel3}
If $\mathcal{O}\subseteq\mathrm{Diff}_c(M)$ is the open identity neighbourhood from \ref{2.4} and 
if $S\colon\mathcal{O}\rightarrow\mathrm{Aut}_c(\mathcal{P})$ is the map from \ref{2.6}, then
\[
\omega\colon\mathcal{O}\times\mathcal{O}\rightarrow \mathrm{Gau}_c(\mathcal{P}),\quad 
(g,g')\mapsto S(g)\circ S(g')\circ S(g\circ g')^{-1}
\]
is smooth. Furthermore, if $\mathcal{Q}\colon\mathrm{Aut}_c(\mathcal{P})
\rightarrow\mathrm{Diff}_c(M),\quad F\mapsto F_M$ is the homomorphism from \ref{1.1b}, then for each $g\in\mathcal{Q}(\mathrm{Aut}_c(\mathcal{P}))$ there exists an open identity neighbourhood $\mathcal{O}_g\subseteq\mathcal{O}$ such that
\[
\omega_g\colon\mathcal{O}_g\rightarrow \mathrm{Gau}_c(\mathcal{P}),\quad g'\mapsto F\circ S(g')\circ F^{-1}\circ S\left(g\circ g'\circ g^{-1}\right)^{-1}
\]
is smooth for any $F\in\mathrm{Aut}_c(\mathcal{P})$ with $F_M=g$.
\end{lem}
\begin{proof}
Since $\mathcal{Q}$ is a homomorphism of groups and $S$ is a section of $\mathcal{Q}$ we have
\[
\mathcal{Q}\left(\omega(g,g')\right)=\mathcal{Q}\left(S(g)\right)\circ\mathcal{Q}\left(S(g')\right)\circ\mathcal{Q}\left(S(g\circ g')\right)^{-1}=\mathrm{id}_M
\]
for $g,g'\in G$ and thus $\omega(g,g')$ is indeed an element of $\mathrm{ker}(\mathcal{Q})=\mathrm{Gau}_c(\mathcal{P})\cong C_c^{\infty}(P,G)^G$.\\
Next we show the smoothness of $\omega$. We want to do this by applying \ref{2.12a} but first we will derive a formula for $\omega(g,g')\circ\sigma_i\in C^{\infty}(\overline{V}_i,G)$ that depends smoothly on $g$ and $g'$ similarly to the one derived in \ref{2.7}. Fix $g,g'\in\mathcal{O}$, $x\in\overline{V}_i$ and denote $\hat{g}:=g\circ g'$. For $S(\hat{g})^{-1}$ we proceed as in \ref{2.7} to find $i_1,\ldots,i_l$ such that
\[
S(\hat{g})^{-1}\left(\sigma_{i}(x)\right)=\sigma_{i_l}\left(\hat{g}^{-1}(x)\right)\cdot k_{i_li_{l-1}}\left((\hat{g}_{l-1}^{-1}\circ\cdots\circ\hat{g}^{-1}_{i_1})(x)\right)\cdots\cdot\cdot k_{{i_1}i}(x).
\]
Accordingly we find $i'_{l'},\ldots,i'_1$ for $S(g')$ and $i''_{l''},\ldots i''_1$ for $S(g)$. Furthermore, we find open neighbourhoods $\mathcal{O}_g$ and $\mathcal{O}_{g'}$ of $g$ and $g'$ and a relatively open neighbourhood $U_x\subseteq\overline{V}_i$ of $x$ (w.l.o.g. such that $\overline{U}_x\subseteq\overline{V}_i$ is a submanifold with corners) such that for $h\in\mathcal{O}_g,h'\in\mathcal{O}_{g'}$ and $x'\in\overline{U}_x$ we have
\begin{align*}
&S(h)\circ S(h')\circ S(h\circ h')^{-1}\left(\sigma_i(x')\right)\\
=&S(h)\circ S(h')\left(\sigma_{i_l}\left(\hat{h}^{-1}(x')\right)\cdot k_{i_{l}i_{l-1}}\left((\hat{h}_{i_{l-1}})^{-1}\circ\cdots\circ (\hat{h}_{i_1})^{-1}(x')\right)\cdots\cdot\cdot k_{i_1i_{l}}(x')\right)\\
=&S(h)\bigg(\sigma_{i'_{l'}}\left(h'\circ h'^{-1}\circ h^{-1}(x')\right)\\
\cdot &k_{i'_{l'}i'_{l'-1}}\left(h'_{i'_{l'-1}}\circ\cdots\circ h'_{i'_1}\circ \hat{h}^{-1}(x')\right)\cdots\cdot\cdot k_{i'_1i_{l}}\left(\hat{h}^{-1}(x')\right)\\
\cdot& k_{i_{l}i_{l-1}}\left((\hat{h}_{i_{l-1}})^{-1}\circ\cdots\circ (\hat{h}_{i_1})^{-1}(x')\right)\cdots\cdot\cdot k_{i_1i_{l}}(x')\bigg)\\
=&\tau_i(x')\cdot\\
&\left.\begin{aligned}
\Big[&k_{ii''_{l''}}(x')k_{i''_{l''}i''_{l''-1}}\left(h_{i''_{l''-1}}\circ\cdots\circ h_{i''_1}\circ h^{-1}(x')\right)\cdots\cdot\cdot k_{i''_1i'_{l'}}\left(h^{-1}(x')\right)\\
\cdot &k_{i'_{l'}i'_{l'-1}}\left(h'_{i'_{l'-1}}\circ\cdots\circ h'_{i'_1}\circ \hat{h}^{-1}(x')\right)\cdots\cdot\cdot k_{i'_1i_{l}}\left(\hat{h}^{-1}(x')\right)\\
\cdot &k_{i_{l}i_{l-1}}\left((\hat{h}_{i_{l-1}})^{-1}\circ\cdots\circ (\hat{h}_{i_1})^{-1}(x')\right)\cdots\cdot\cdot k_{i_1i_{l}}(x')
\Big].
\end{aligned}
\right\}
\qquad :=\kappa_{x,h,h'}(x')
\end{align*} 
Since we do not change $g$ and $g'$ in the following $\kappa_{x,h,h'}(x')$ is independent of them.
Each of the factors can be interpreted as an element of $C^\infty(\overline{U}_x,G)$ that depends smoothly on $h$ and $h'$, as in \ref{2.10}. Thus  for fixed $x$, the whole formula depends smoothly on $h$ and $h'$.
Furthermore, $\kappa_{x_1,h,h'}$ coincides with $\kappa_{x_2,h,h'}$ on $\overline{U}_{x_1}\cap\overline{U}_{x_2}$ because
\[
\sigma_i(x')\cdot\kappa_{x_1,h,h'}(x')=S(h)\circ S(h')\circ S(h\circ h')^{-1}\left(\sigma_i(x')\right)=\sigma_i(x')\cdot\kappa_{x_2,h,h'}(x')
\]
for $x'\in\overline{U}_{x_1}\cap\overline{U}_{x_2}$. Now, finitely many $U_{x_1},\ldots,U_{x_l}$ cover $\overline{V}_i$ and we thus see that
\[
\gamma_i(h,h'):=\mathrm{glue}\left(\kappa_{x_1,h,h'}\big|_{\overline{U}_{x_1}},\ldots,\kappa_{x_m,h,h'}\big|_{\overline{U}_{x_l}}\right)\bigg|_{\overline{V}_i}
\]
depends smoothly on $h$ and $h'$. After intersecting finitely many sets we may assume that $\gamma_i$ is defined on an open neighbourhood $\mathcal{O}_{g,i}\times\mathcal{O}_{g',i}$ of $(g,g')$. For almost all $j\in\mathbb{N}$ we have that $(\mathrm{supp}(g)\cup\mathrm{supp}(g'))\cap\overline{V}_j=\varnothing$ and in this case $\gamma_j$ is trivial. Hence, we may assume that almost all $\mathcal{O}_{g,j}\times\mathcal{O}_{g',j}$ are a neighbourhood of $(\mathrm{id}_M,\mathrm{id}_M)$, because for the identity all components $\omega(\mathrm{id}_M,\mathrm{id}_M)\circ\sigma_i$ are trivial. 
Following \ref{2.7a} we see that $\gamma_i(h,h')$ only depends on  $h\big|_{U_i}$ and $h'\big|_{U_i}$ because we have at most compositions of three maps from $\mathcal{O}$ in our formula and $x\in\overline{V}_i$.\\ 
Recall from \ref{1.2} that for every $F_\gamma\in\mathrm{Gau}_c(\mathcal{P})$ there exists an $\gamma\in C^{\infty}_c(P,G)^G$ with $F_\gamma(p)=p\cdot \gamma(p)$ for all $p\in P$. It follows that the local description of $F_\gamma$ is given by $F_\gamma\circ\sigma_i(x)=\sigma_i(x)\cdot \gamma\circ\sigma_i(x)$ and thus $\gamma_i$ defines the local representation of $\omega(h,h')$ via the isomorphism $\mathrm{Gau}(\mathcal{P})\cong G_{\overline{\mathcal{V}}}(\mathcal{P})$.\\
Now, it remains to find a suitable locally convex direct sum and open sets so that \ref{2.12a} can be applied. In fact, for the compactly carried smooth vector fields on $M$ we have the natural isomorphism $C_c^{\infty}(M,TM)^2\cong C_c^{\infty}(M,TM\oplus TM)$ (\cite[cf. F.12]{GloTF}) and by using \ref{2.9lem} we get a continuous linear map 
\[
\Phi\colon C^{\infty}_c(M,TM\oplus TM)\rightarrow\bigoplus_{i\in\mathbb{N}}C^\infty_c(M,TM\oplus TM)
\] 
such that $(\Phi(\eta))_i\big|_{U_i}=\eta\big|_{U_i}$
Recall the chart $\varphi\colon\mathcal{O}\rightarrow\Omega\subseteq C_c^{\infty}(M,TM)$ and
let $\Omega_{g,i}:=\varphi^{-1}(\mathcal{O}_{g,i})$, $\Omega_{g',i}:=\varphi^{-1}(\mathcal{O}_{g',i})$ and $\Omega_i:=\Omega_{g,i}\times\Omega_{g',i}$. Note that almost all of these sets are zero neighbourhoods.
Because $\gamma_i\circ(\varphi^{-1}\times\varphi^{-1})\big|_{\Omega_i}\colon\Omega_i\rightarrow C^\infty(\overline{V}_i,G)$ is a smooth map, we can apply \ref{2.12a} and see that 
\[
\Psi\colon\bigoplus_{i\in\mathbb{N}}\Omega_i\rightarrow \bigoplus_{i\in\mathbb{N}}C^{\infty}(\overline{V}_i,G),\quad (\eta)_{i\in\mathbb{N}}
\mapsto (\gamma_i\circ(\varphi^{-1}\times\varphi^{-1})(\eta_i))_{i\in\mathbb{N}}
\]
is smooth (zero obviously maps to the unity element). Identifying $\Omega\times\Omega$ with an open set in $C_c^\infty(M,TM\oplus TM)$,
we define the open neighbourhood  $\Omega':=\Omega\times\Omega\cap\Phi^{-1}(\bigoplus_{i\in\mathbb{N}}\Omega_i)$ of $(\varphi^{-1}(g),\varphi^{-1}(g'))$ and note that $\Psi\circ\Phi\big|_{\Omega'}$ is smooth. By construction $\Psi((\eta_i)_{i\in\mathbb{N}})$ only depends on $(\eta_i\big|_{U_i})_{i\in\mathbb{N}}$ and the components of $\Psi(\eta)$ are determined by $\eta\big|_{U_i}$, hence $\Psi\circ\Phi\big|_{\Omega'}(\eta)=\omega\circ(\varphi^{-1}\times\varphi^{-1})\big|_{\Omega'}(\eta)$ which shows that $\omega$ is smooth.\\
Now, let $g\in\mathcal{Q}(\mathrm{Diff}_c(M))$ and $F\in\mathrm{Aut}_c(\mathcal{P})$ with $F_M=g$. The strategy to prove that
\[
\omega_g\colon\mathcal{O}_g\rightarrow \mathrm{Gau}_c(\mathcal{P}),\quad g'\mapsto F\circ S(g')\circ F^{-1}\circ S\left(g\circ g'\circ g^{-1}\right)^{-1}
\]
is smooth is almost identical to the above: First we derive an explicit formula for $\omega_g(g')\circ\sigma_i\in C^{\infty}(\overline{V}_i,G)$ and then use it to apply \ref{2.12a}. Note that
\[
\mathcal{Q}\left(F\circ S(g')\circ F^{-1}\circ S(g\circ g'\circ g^{-1})^{-1}\right)=g\circ g'\circ g^{-1}\circ g\circ g'^{-1}\circ g^{-1}=\mathrm{id}_M
\]
holds. Let $\mathcal{O}_g\subseteq\mathcal{O}$ be an open identity neighbourhood with $g\circ\mathcal{O}_g\circ g^{-1}\subseteq\mathcal{O}$ and denote $\overline{g'}=g\circ g'\circ g^{-1}$ for $g'\in\mathcal{O}_g$. Fix $g'$ and $x\in\overline{V}_i$. Proceeding as in \ref{2.7} we find $j_l,\ldots,j_1$ such that
\[
S\left(\overline{g'}\right)^{-1}\big(\sigma_i(x)\big)=\sigma_{j_l}\left(\overline{g'}^{-1}(x)\right)\cdot k_{j_lj_{l-1}}\left((\overline{g'}_{j_{l-1}})^{-1}\circ\cdots\circ(\overline{g'}_{j_1})^{-1}(x)\right)\cdot\cdots\cdot k_{j_1i}(x).
\]
Furthermore, let $j_1'$ be maximal such that
\[
\left(F_M^{-1}\circ S(\overline{g'})^{-1}_M\right)(x)=g'^{-1}\circ g^{-1}(x)\in V_{j'_1}
\]
and let $U_x\subseteq\overline{V}_i$ be a relatively open neighbourhood of $x$ (w.l.o.g. such that $\overline{U}_x\subseteq\overline{V}_i$ is a submanifold with corners) such that $\overline{g'}^{-1}(\overline{U}_x)\subseteq V_{j_l}$ and $g'^{-1}\circ g^{-1}(\overline{U}_x)\subseteq V_{j'_1}$. Since $F_M=g$ and
\[
F^{-1}\left(\sigma_{j_l}\left(\overline{g'}^{-1}(x')\right)\right)\in\sigma_{j'_1}\left(g'^{-1}\circ g^{-1}(U_x)\right)\text{ for }x'\in U_x
\]
we have
\[
F^{-1}\left(\sigma_{j_l}\left(\overline{g'}^{-1}(x')\right)\right)=\sigma_{j'_1}\left(g'^{-1}\circ g^{-1}(x')\right)\cdot k_{F,x,g'}(x')\text{ for }x'\in U_x,
\]
for some smooth function $k_{F,x,g'}\colon U_x\rightarrow G$. In fact we have
\[
k_{F,x,g'}(x)=k_{\sigma_{j_1}}\left(F^{-1}\left(\sigma_{j_l}\left(\overline{g'}^{-1}(x')\right)\right)\right).
\]
After shrinking $U_x$, \ref{smooth} shows that $k_{F,x,g'}\big|_{\overline{U}_x}\in C^{\infty}(\overline{U}_x,G)$ depends smoothly on $g'$ for fixed $x$ because $\overline{g'}^{-1}(\overline{U}_x)\subseteq \overline{U}_i$.\\
Using the same strategy for $S(g')(\sigma_{j'_1}(g'^{-1}\circ g^{-1}(x')))$ and $F\circ S(g')$, we find indices $j_2',\ldots,$ $j_{l'+1}'$ and a smooth function $k'_{F,x,g'}\colon\overline{U}_x\rightarrow G$ (possibly after shrinking $U_x$ once more), depending smoothly on $g$ such that
\begin{align*}
\omega_g(g')\big(\sigma_i(x)\big)=&\sigma_{j'_{l'+1}}\big(g'(x)\big)\cdot k'_{F,x,g'}(x)\cdot k_{j'_{l'}j'_{l'-1}}\big(g'(x)\big)\cdot\cdots\cdot k_{j'_2j'_1}\left(g^{-1}(x)\right)\cdot k_{F,x,g'}(x)\\
&\cdot k_{j_lj_{l-1}}\left(\overline{g}'^{-1}(x)\right)\cdot\cdots\cdot k_{j_1i}(x)\\
&=\tau_i(x)\cdot\\
&\left.\begin{aligned}
\Big[&k_{ij'_{l'+1}}\big(g'(x)\big)\cdot k'_{F,x,g'}\cdot k_{j'_{l'}j'_{l'-1}}\big(g'(x)\big)\cdot\cdots\cdot k_{j'_2j'_1}\left(g^{-1}(x)\right)\\
&\cdot k_{F,x,g'}(x)\cdot k_{j_lj_{l-1}}\left(\overline{g}'^{-1}(x)\right)\cdot\cdots\cdot k_{j_1i}(x)
\Big].
\end{aligned}
\right\}
:=\kappa_{x,g'}(x)
\end{align*}
Since we will not vary $F$ and $g$ in the sequel, we suppressed the dependence of $\kappa_{x,g'}$ on them. Following the lines of \ref{2.7} we find an open neighbourhood $\mathcal{O}_{g'}$ and a relatively open neighbourhood $U_x'\subseteq\overline{V}_i$ of $g'$ and $x$ (w.l.o.g. such that $\overline{U'}_x\subseteq\overline{V}_i$ is a submanifold with corners) such that $\sigma_i(x')\cdot\kappa_{x,g'}(x')=\omega_g(h')(\sigma_i(x'))$ for all $h'\in\mathcal{O}_{g'}$ and $x'\in\overline{U'}_x$. As in the first part of the proof, finitely many $U_{x_1},\ldots U_{x_m}$ cover $\overline{V}_i$ and $\kappa_{x_1,g'}=\kappa_{x_2,g'}$ on $\overline{U}_{x_1}\cap\overline{U}_{x_2}$. Since the gluing and restriction maps from
\ref{A.17} and \ref{A.18} are smooth,
\[
\gamma'_i(g'):=\mathrm{glue}\left(\kappa_{x_1,g'}\big|_{\overline{U}_{x_1}},\ldots,\kappa_{x_m,g'}\big|_{\overline{U}_{x_m}}\right)\bigg|_{\overline{V}_i}
\]
depends smoothly on $g'$ and is defined on an open neighbourhood $\mathcal{O}'_i$ of $g$. Again, we may assume that almost all $\mathcal{O}'_i$ contain the identity because $\mathrm{supp}(g)\cap\overline{V}_i=\varnothing$ implies that $\gamma'_i$ is trivial.
This map is the local representation of $\omega_g(g')\in\mathrm{Gau}_c(\mathcal{P})$ and it only depends on $g'\big|_{U_i}$ because $g'\in\mathcal{O}$ and $x\in\overline{V}_i$.
Using \ref{2.9lem} we get a continuous linear map
\[
\Phi'\colon C_c^\infty(M,TM)\rightarrow\bigoplus_{i\in\mathbb{N}}C^{\infty}_c(M,TM)
\]
such that $\eta\big|_{\overline{U}_i}=(\Phi'(\eta))_i\big|_{\overline{U}_i}$ for $\eta\in C_c^\infty(M,TM)$. Almost all of the open sets  $\Omega'_i:=\varphi^{-1}(\mathcal{O}'_i)$ are zero neighbourhoods. Also, the map
\[
\Psi'\colon\bigoplus_{i\in\mathbb{N}}\Omega'_i\rightarrow \sideset{}{^\ast}\prod_{i\in\mathbb{N}}C^\infty(\overline{V}_i,G),\quad
(\eta_i)_{i\in\mathbb{N}}\mapsto (\gamma'_i(\varphi^{-1}(\eta_i)))_{i\in\mathbb{N}}
\]
is smooth by \ref{2.12a} and depends only on the values of $(\eta_i\big|_{U_i})_{i\in\mathbb{N}}$. 
Thus $\omega_g(\varphi^{-1}(\eta))=\Psi'(\Phi'(\eta))$ for $\eta\in\Phi'^{-1}(\bigoplus_{i\in\mathbb{N}}\Omega'_i)$ which implies $\omega_g$ is smooth on an open neighbourhood of $g$, hence smooth.
\end{proof}

After having adapted the preceding results from maps on compact manifolds to compactly carried maps on $\sigma$-compact manifolds, the main result of this work follows exactly as \cite[Theorem 3.4.14]{Diss}. In the proof, we will also give an explicit description of the Lie group structure on $\mathrm{Aut}_c(\mathcal{P})$ by calculating the group operations in some chart. Afterwards, we use this description and \ref{A.4} to show that up to isomorphism the Lie group structure does not depend on the ´choice of the section $S$. 

\begin{satz}\label{2.14}\cite[cf. Theorem 3.4.14]{Diss}
Let $\mathcal{P}$ be a smooth principal $G$-bundle over the $\sigma$-compact finite-dimensional manifold $M$. If $\mathcal{P}$ has the property $\text{SUB}_\oplus$, then $\mathrm{Aut}_c(\mathcal{P})$ carries a Lie group structure such that we have an extension of smooth Lie groups
\[
\mathrm{Gau}_c(\mathcal{P})\hookrightarrow\mathrm{Aut}_c(\mathcal{P})\xtwoheadrightarrow{\mathcal{Q}}\mathrm{Diff}_c(M)_{\mathcal{P}},
\]
where $\mathcal{Q}:\mathrm{Aut}_c(\mathcal{P})\rightarrow\mathrm{Diff}_c(M)_{\mathcal{P}}$ is the homomorphism from \ref{1.1b} and $\mathrm{Diff}_c(M)_{\mathcal{P}}:=\mathrm{im}(\mathcal{Q})$. 
\end{satz}
\begin{proof}
First we note that $\mathrm{Diff}_c(M)_{\mathcal{P}}$ is an open Lie subgroup because $\mathcal{Q}\circ S=\mathrm{id}_{\mathcal{O}}$, hence $\mathcal{O}\subseteq\mathrm{im}(\mathcal{Q})$. Recall the definition of a smooth factor system from \ref{II.6}. We identify $\mathrm{Gau}_c(\mathcal{P})$ with $C_c^\infty(P,G)^G$ and extend $S$ to a (possibly non-continuous) section $S\colon\mathrm{Diff}_c(M)_\mathcal{P}\rightarrow\mathrm{Aut}_c(\mathcal{P})$. 
Now Lemma \ref{2.11} combined with the Remark \ref{2.8} show that $T$ extends to a map satisfying the smoothness conditions of \ref{II.6}. 
If we translate the conditions on $\omega_g$ from \ref{II.6} to our situation we need to check that
\[
\omega(g,x)\circ\omega(g\circ x\circ g^{-1},g)^{-1}=S(g)\circ S(x)\circ S(g)^{-1}\circ S(g\circ x\circ g^{-1})^{-1}
\]
is locally smooth. But with $F=S(g)$, this follows from the second part of \ref{2.12}. The first part of \ref{2.12} and \ref{II.5} show the remaining requirements of \ref{II.6} to be satisfied as well. Thus $(T,\omega)$ is a smooth factor system and now \ref{II.8} yields the assertion. \\
However, using $\ref{A.4}$ we are in a position to give a more explicit description of the Lie group structure on $\mathrm{Aut}_c(\mathcal{P})$. We define a smooth manifold structure on $W:=\mathrm{Gau}_c(\mathcal{P})\circ S(\mathcal{O})$ by requiring the map 
\[
\varphi^{-1}\colon W\rightarrow C^\infty_c(P,G)^G\times\mathcal{O},\quad F_\gamma\circ S(g)\mapsto (\gamma,g)
\]
to be a diffeomorphism (see \ref{II.5}). Let $\mathcal{O}''\subseteq\mathcal{O}$ be a symmetric open identity neighbourhood such that $\mathcal{O}''\circ\mathcal{O}''\subseteq\mathcal{O}$ and for each $g\in\mathrm{Diff}_c(M)$ denote by $\mathcal{O}_g$ the open identity neighbourhood from \ref{2.12}. Then multiplication in terms of $\varphi$ is given by
\[
(C^\infty_c(P,G)^G\times\mathcal{O}'')^2\ni\big((\gamma,g),(\gamma',g')\big)\mapsto\varphi\big(\varphi^{-1}(\gamma,g)\circ\varphi^{-1}(\gamma',g')\big)\in C^\infty_c(P,G)^G\times\mathcal{O},
\]
the inversion in terms of $\varphi$ is given by
\[
C^\infty_c(P,G)^G\times\mathcal{O}\ni(\gamma,g)\mapsto\varphi\big((\varphi^{-1}(\gamma,g))^{-1}\big)\in C^\infty_c(P,G)^G\times\mathcal{O}
\]
and conjugation with $F\in\mathrm{Aut}_c(\mathcal{P})$ is given by
\[
C^\infty_c(P,G)^G\times\mathcal{O}_{\mathcal{Q}(F)}\ni(\gamma,g)\mapsto\varphi\big(F\circ\varphi^{-1}(\gamma,g)\circ F^{-1}\big)\in C^\infty_c(P,G)^G\times\mathcal{O}.
\]
Because of \ref{2.11}, \ref{2.12} and $\mathcal{Q}(S(g))=g$ the smoothness of these maps is shown by the following identities (using $\mathrm{Gau}_c(\mathcal{P})\cong C^\infty_c(P,G)^G$):
\begin{align*}
&\varphi\big(\varphi^{-1}(\gamma,g)\circ\varphi^{-1}(\gamma',g')\big)\\
=&\varphi(F_\gamma\circ S(g)\circ F_{\gamma'}\circ S(g'))\\
=&\big(F_\gamma\circ S(g)\circ F_{\gamma'}\circ S(g')\circ S(g\circ g')^{-1},g\circ g'\big)\\
=&\big(F_\gamma\circ\underbrace{S(g)\circ F_{\gamma'}\circ S(g)^{-1}}_{=T(\gamma,g)}\circ\underbrace{S(g)\circ S(g')\circ S(g\circ g')^{-1}}_{=\omega(g,g')},g\circ g'\big),
\end{align*}
\begin{align*}
&\varphi\big((\varphi^{-1}(\gamma,g))^{-1}\big)\\
=&\big(S(g)^{-1}\circ F_{\gamma^{-1}}\circ S(g^{-1})^{-1},g^{-1}\big)\\
=&\big(\underbrace{S(g)^{-1}\circ S(g^{-1} )^{-1}}_{=\omega(g^{-1},g)^{-1}}
\circ\underbrace{S(g^{-1}\circ F_{\gamma^{-1}}\circ S(g^{-1})^{-1}}_{=T(\gamma^{-1},g^{-1})},g^{-1}\big)\quad\quad\quad\quad
\end{align*}
and
\begin{align*}
&\varphi\big(F\circ\varphi^{-1}(\gamma,g)\circ F^{-1}\big)\\
=&\big(F\circ\varphi^{-1}(\gamma,g)\circ F^{-1}\circ S(F_M\circ g\circ F_M^{-1})^{-1},F_M\circ g\circ F^{-1}_M\big)\\
=&\big(\underbrace{F\circ F_\gamma\circ F^{-1}}_{=c_F(\gamma)}\circ\underbrace{F\circ S(g)\circ F^{-1}\circ S(F_M\circ g\circ F^{-1}_M)^{-1}}_{=\omega_{F_M}(g)},F_M\circ g\circ F_M^{-1}\big)
\end{align*}
\end{proof}

\begin{bem}\label{2.17}\cite[Remark 2.17]{Wockel3}
Of course, the construction of the Lie group structure on $\mathrm{Aut}_c(\mathcal{P})$ from \ref{2.14} depends on the choice of $S$ and thus on the choice of the chart $\varphi\colon\mathcal{O}\rightarrow\Omega$ from \ref{2.4}, the choice of the trivializing systems from \ref{2.2} and the partition of unity chosen in \ref{2.4}.\\
However, different choices lead to the same Lie group structures on $\mathrm{Aut}_c(\mathcal{P})$ and, moreover to equivalent extensions. To see this we show that $\mathrm{id}_{\mathrm{Aut}_c(\mathcal{P})}$ is smooth when choosing two different trivializing systems $\overline{\mathcal{V}}=(\overline{V}_i,\sigma_i)_{i\in\mathbb{N}}$ and $\overline{\mathcal{V}}'=(\overline{V}'_j,\tau'_j)_{j\in\mathbb{N}}$.\\
Denote by $S\colon\mathcal{O}\rightarrow\mathrm{Aut}_c(\mathcal{P})$ and $S'\colon\mathcal{O}\rightarrow\mathrm{Aut}_c(\mathcal{P})$ the corresponding sections of $\mathcal{Q}$. Since
\[
\mathrm{Gau}_c(\mathcal{P})\circ S(\mathcal{O})=\mathcal{Q}^{-1}(\mathcal{O})=
\mathrm{Gau}_c(\mathcal{P})\circ S'(\mathcal{O})
\]
is an open unit neighbourhood and $\mathrm{id}_{\mathrm{Aut}_c(\mathcal{P})}$ is an isomorphism of abstract groups, it suffices to show that the restriction of $\mathrm{id}_{\mathrm{Aut}_c(\mathcal{P})}$ to $\mathcal{Q}^{-1}(\mathcal{O})$ is smooth. Now, given the explicit description of the Lie group structure of $\mathrm{Aut}_c(\mathcal{P})$ in \ref{2.14} we require
\begin{align*}
&\mathcal{Q}^{-1}(\mathcal{O})\ni F\mapsto\big(F\circ S(F_M)^{-1},F_M\big)\in\mathrm{Gau}_c(\mathcal{P})\times\mathrm{Diff}_c(M)\quad\text{    and}\\
&\mathcal{Q}^{-1}(\mathcal{O})\ni F\mapsto\big(F\circ S'(F_M)^{-1},F_M\big)\in\mathrm{Gau}_c(\mathcal{P})\times\mathrm{Diff}_c(M)
\end{align*}
to be diffeomorphisms. We thus only need to show that
\[
\mathcal{O}\ni g\mapsto S'(g)\circ S(g)^{-1}\in\mathrm{Gau}_c(\mathcal{P})
\]
is smooth. By deriving explicit formulae for $S'(g)\circ S(g)^{-1}\circ\sigma_i$ on a neighbourhood $U_x$ of $x\in\overline{V}_i$ and $\mathcal{O}_g$ of $g\in\mathcal{O}$ this follows by the same steps as in \ref{2.12}. Indeed, if the formula for $S(g)^{-1}\circ\sigma_i(x)$ is given by $\sigma_{i_l}(g^{-1}(x))\cdot k(x)$ for some smooth function $k\colon U_x\rightarrow G$, there exists an $i_1'$ such that $\sigma_{i_l}(g^{-1}(x))\in \overline{V}'_{i_1'}$ and we can write $\tau'_{i_1'}(g(x))k_{i_1'i_l}(g(x))\cdot k(x)$ for some smooth transition function $k_{i_1'i_l}$. From there we continue as before.
\end{bem}

\begin{prop}[Lie groups from local data]\cite{GloBO}\label{A.4}
Let $G$ be a group with locally convex manifold structure on some subset $U\subseteq G$ with $1\in U$. Furthermore, assume that there exists an open $V\subseteq U$ such that $1\in V,VV\subseteq U, V=V^{-1}$ and
\begin{enumerate}[(i)]
\item $V\times V\rightarrow U,(g,h)\mapsto gh$ is smooth.
\item $V\rightarrow V,g\mapsto g^{-1}$ is smooth.
\item $\text{for all }g\in G$, there exists an open neighbourhood $W\subseteq U$ such that we have 
$g^{-1}Wg\subseteq U$ and the map $W\rightarrow U,h\mapsto g^{-1}hg$ is smooth.
\end{enumerate}
Then there exists a unique locally convex manifold structure on $G$ which turns $G$ into a Lie group, such that $V$ is an open submanifold of $G$.
\end{prop}
%
\section*{Acknowledgements}
The author wants to express his utmost gratitude to Professor Helge Gl\"ockner of Paderborn University, who gave the initiative and supervised the author's master's thesis from which the present paper arose. 
\def\polhk#1{\setbox0=\hbox{#1}{\ooalign{\hidewidth
  \lower1.5ex\hbox{`}\hidewidth\crcr\unhbox0}}}

{\footnotesize
{\bf Jakob Sch\"{u}tt}\\
Universit\"at Paderborn\\
Institut f\"ur Mathematik\\
Warburger Str.\ 100\\
33098 Paderborn, Germany\\
Email: spoon@math.upb.de}


\begin{thebibliography}{99}
\bibitem{Hamsa}
Alzaareer, H., \emph{Lie groups of mappings on non-compact spaces and manifolds}, 2013,
  {D}issertation, {U}niversit{\"a}t {P}aderborn.
%
\bibitem{Bour}
Bourbaki, N., ``Topological Vector Spaces,'' Elements of Mathematics (Berlin),
  Springer-Verlag, Berlin, 1987.
%
\bibitem{Bour2}
---, N., ``General topology,'' Elements of Mathematics (Berlin),
  Springer-Verlag, Berlin, 1989.
%
\bibitem{DiffTop}
Outerelo Domínguez, E., Margalef-Roig, J., ``Differential topology'',
  North Holland, Amsterdam, 1989.
%
\bibitem{GloComplete}
Gl{\"o}ckner, H., \emph{Infinite-dimensional {L}ie groups without completeness
  restrictions}, in ``Geometry and Analysis on Finite- and Infinite-Dimensional
  {L}ie Groups (B\polhk edlewo, 2000),'' volume~55 of \emph{Banach Center
  Publ.}, 43--59, Polish Acad. Sci., Warsaw, 2002.
%
\bibitem{GloQU}
---, \emph{Lie group structures on quotient groups and universal
  complexifications for infinite-dimensional {L}ie groups}, J. Funct. Anal.
  \textbf{194} (2002), 347--409.
%
\bibitem{GloSS}
---, \emph{Differentiable mappings between spaces of sections}, manuscript, 2002.
%
\bibitem{GloMM}
---, \emph{Lie Groups of Measurable Mappings}, Canad. J. Math.
  \textbf{55}(5) (2003), 969--999.
%
\bibitem{GloTF}
---, \emph{Lie groups over non-discrete topological fields},
  arxiv:math/0408008v1, 2004.
%
\bibitem{GloPatched}
---, \emph{Patched locally convex spaces, almost local mappings and
  diffeomorphism groups of non-compact manifolds}, manuscript, 2006.
%
\bibitem{GloDL}
---, \emph{Direct limits of infinite-dimensional {L}ie groups compared to direct limits in related categories}, J. Funct. Anal.
  \textbf{245} (2002), 19--61.
%
\bibitem{GloBO}
Gl{\"o}ckner, H., Neeb, K.-H., ``Infinite-dimensional {L}ie groups, vol. 1: Basic theory and main examples'',
  Springer-Verlag, Berlin, to appear.
%
\bibitem{Jarch}
Jarchow, H., ``Locally Convex Spaces'',
  Teubner, Stuttgart, 1981.
%
\bibitem{Michor}
Michor, P.~W., ``Manifolds of {D}ifferentiable {M}appings,'' volume~3 of
  \emph{Shiva Mathematics Series}, Shiva Publishing Ltd., Nantwich, 1980, out
  of print, online available from
  \texttt{http://www.mat.univie.ac.at/{\~{}}michor/}.
%
\bibitem{Milnor}
Milnor, J., \emph{Remarks on infinite-dimensional {L}ie groups}, in
  ``Relativity, {G}roups and {T}opology, II (Les Houches, 1983),'' 1007--1057,
  North-Holland, Amsterdam, 1984.
%
\bibitem{Neeb}
Neeb, K.-H., \emph{Non-abelian extensions of infinite-dimensional {L}ie groups}, Ann. Inst. Fourier (Grenoble) \textbf{57}(1) (2007), 209--271.
%
\bibitem{Schaef}
Schaefer, H.~H., ``Topological Vector Spaces'',
  Springer-Verlag, Berlin, 1971.
%
\bibitem{Diss}
Wockel, C., \emph{Infinite-{D}imensional {L}ie {T}heory for {G}auge {G}roups}, 2006,
  {D}issertation, {T}echnische {U}niversit{\"a}t {D}armstadt.
%
\bibitem{Wockel3}
---, \emph{Lie groups structures of symmetry groups of principal bundles}, J. Funct. Anal.
  \textbf{251} (2007), 254--288.
%
\end{thebibliography}
\end{document}